\documentclass[11pt]{amsart}
\usepackage{geometry}                
\usepackage{graphicx,enumerate}
\usepackage{amsmath,amsthm,amsfonts,amssymb,MnSymbol}
\usepackage[colorlinks=true, pdfstartview=FitV, linkcolor=blue, citecolor=blue, urlcolor=blue]{hyperref}

\def\ie{\emph{i.e.} }
\def\as{\emph{a.s.} }
\def\beq{\begin{equation}}
\def\eeq{\end{equation}}
\def\N{\ensuremath \mathbb{N}}
\def\R{\ensuremath \mathbb{R}}
\def\V{\ensuremath \mathcal{V}}

\def\SS{\ensuremath \mathbb{S}}
\def\eps{\ensuremath \varepsilon}
\def\sign{ \mbox{sign}\,}
\newcommand{\la}{\left\langle}
\newcommand{\ra}{\right\rangle}
\def\supp{ \mbox{supp}\,}
\def\Ac{\textbf{(A1)}}
\def\Ad{\textbf{(A2)}}
\def\Abc{\textbf{(A1')}}
\def\v0{\mathbf{v}_0}
\def\vt{\mathbf{v}_\tau}
\def\vm{\v0}
\def\vmin{\mathbf{v}_{\mathrm{min}}}
\def\bbf{\mathbf{f}}
\def\bg{\mathbf{g}}
\def\la{\left\langle}
\def\ra{\right\rangle}
\def\T{\mathcal{T}}

\newtheorem{theorem}{Theorem}
\newtheorem{proposition}[theorem]{Proposition}
\newtheorem{corollary}[theorem]{Corollary}
\newtheorem{definition}[theorem]{Definition}
\newtheorem{lemma}[theorem]{Lemma}
\newtheorem{remark}[theorem]{Remark}

\numberwithin{equation}{section}\numberwithin{theorem}{section}

\title{Chemotactic waves of bacteria at the mesoscale}
\author{Vincent Calvez}
\address{Unit\'e de Math\'ematiques Pures et
Appliqu\'ees, UMR 5669 CNRS \& Ecole Normale Sup\'erieure de Lyon, and
 project-team Inria NUMED, Lyon, France}
\email{vincent.calvez@ens-lyon.fr}

\date{\today}                                           

\begin{document}

\begin{abstract}
The existence of travelling waves for a model  of concentration waves of bacteria is investigated. The model consists in a  kinetic equation for the biased motion of cells following a run-and-tumble process, coupled with two reaction-diffusion equations for the chemical signals. Strong mathematical difficulties arise in comparison with the diffusive regime which was studied in a previous work. The cornerstone of the proof consists in establishing monotonicity properties of the  spatial density of cells. Travelling waves exist under certain conditions on the parameters. Counter-examples to both existence and uniqueness are found numerically after  careful analysis of the discrete velocity problem.  
\end{abstract}
\maketitle

\tableofcontents

\section{Introduction}
\subsection{The run-and-tumble model}

The existence of travelling waves for bacterial chemotaxis is investigated at the mesoscopic scale. We study the standard run-and-tumble model proposed originally by Stroock \cite{stroock_stochastic_1974}, Alt \cite{alt_biased_1980}, and discussed further by Othmer, Dunbar and Alt \cite{othmer_models_1988},
\begin{equation}
\label{eq:meso model}
\displaystyle \partial_t f(t,x,v) + v\cdot \nabla_x f(t,x,v) =  \int  {\bf T} (t,x,v,v')  f(t,x,v')\, d\nu(v') -    \boldsymbol\lambda (t,x,v) f(t,x,v)\, .
\end{equation}
This equation describes the evolution of the density of bacteria $f(t,x,v)$, at time $t>0$, position $x\in \R^d$, and velocity $v\in V\subset \R^d$. Here, $\nu$ is a symmetric probability measure on $\R^d$. The set of admissible velocities is $V = \supp \nu$. It is assumed to be compact all along this work.

Bacteria perform run-and-tumble motion in a liquid medium, as described in the seminal tracking experiments by Berg and Brown \cite{berg_chemotaxis_1972}, Macnab and Koshland \cite{macnab_gradient-sensing_1972} (see also Berg's book \cite{berg_e._2004}). They alternate between so-called run phases of ballistic motion (say, with velocity $v'$), and tumble phases of rotational diffusion (Figure \ref{fig:tracking}). Fast reorientation of the cell occurs during tumble. After tumbling, bacteria make a new run with another velocity $v$. For the sake of simplicity, we assume that tumbles are instantaneous reorientation events, where the cell changes velocity from $v'$ to $v$. On the other hand, duration of each run phase is a random time. It is assumed to follow an exponential distribution with heterogeneous rate $ \boldsymbol\lambda (t,x,v)$. By modulating the time of runs, bacteria are able to distinguish between favourable and unfavourable directions. This strategy based on temporal-sensing chemotaxis allows them to navigate in heterogeneous environments. 

The tumbling frequency distribution ${\bf T}$ can be decomposed as \[{\bf T} (t,x,v,v') = {\bf K}(t,x,v,v') \boldsymbol\lambda (t,x,v')\,,\] 
for some probability density function ${\bf K}(t,x,\cdot,v')$, satisfying $\int {\bf K}(t,x,v,v')\, d\nu(v) = 1$. The kernel ${\bf K}$ describes the distribution of post-tumbling velocities. 

Following \cite{saragosti_mathematical_2010, saragosti_directional_2011}, we make the hypothesis that runs are modulated by two chemoattractant signals. We denote by $S(t,x)$ the concentration of some amino-acid released by the bacterial population ({\em e.g.} aspartate, serine). We denote by $N(t,x)$ the concentration of some nutrient consumed by the bacterial population ({\em e.g.} oxygen, glucose). 

We assume that both signals have an additive effect on  temporal sensing. Furthermore, they proceed with the same signal integration function $\boldsymbol\phi$, with possibly different relative amplitudes $\chi_S,\chi_N>0$. 
Accordingly, the tumbling frequency reads
\beq\label{eq:lambda} \boldsymbol\lambda (t,x,v') = \boldsymbol\lambda_S (t,x,v') + \boldsymbol\lambda_N (t,x,v')\, , \quad \begin{cases}  \boldsymbol\lambda_S (t,x,v') =  \lambda_0\left(\dfrac12 + \chi_S \boldsymbol\phi\left( \left.\dfrac{D\log S}{Dt}\right|_{v'} \right) \right)   \medskip \\   \boldsymbol\lambda_N (t,x,v') = \lambda_0\left(\dfrac12 +  \chi_N \boldsymbol\phi\left( \left.\dfrac{D\log N}{Dt}\right|_{v'} \right)\right) \end{cases}\, , \eeq
where $\left.\frac{D}{Dt}\right|_{v'} = \partial_t + v'\cdot\nabla_x$ denotes the material derivative along the direction $v'$. Chemotaxis is positive when $\boldsymbol \phi$ is globally decreasing (tumbling is more likely when the concentration of chemoattractant is decreasing).

Equation \eqref{eq:meso model} is complemented with two reaction-diffusion equations for both chemical concentrations $S, N$, 
\begin{equation}
\label{eq:chemoattractant}
\begin{cases}
\partial_t S(t,x)  -  D_S \partial^2_{x} S(t,x) + \alpha S (t,x)=  \beta \rho(t,x) \medskip\\
\partial_t N (t,x) -  D_N \partial^2_{x} N(t,x)  = - \gamma \rho(t,x) N(t,x)\, ,
\end{cases}
\end{equation}
where $D_S, D_N, \alpha,\beta,\gamma$ are positive constants, denoting respectively the diffusion coefficient of $S$, the diffusion coefficient of $N$, the rate of degradation of $S$, the rate of production of $S$ by the bacteria, and the rate of consumption of the nutrient $N$ by the bacteria. In addition, $\rho$ denotes the spatial density of bacteria,
\[ 
\rho(t,x) = \int f(t,x,v)\, d\nu(v)\, .
\]

\begin{figure}
\begin{center}
\includegraphics[width = .4\linewidth]{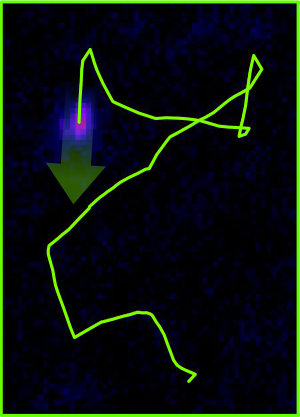}
\caption{\small Typical trajectory of the swimming bacteria {\em E. coli}. Motion alternates between run phases (straight motion), and reorientation events (tumble). At the mesoscopic scale, it is reasonable to assume that both duration of run and duration of tumble follow exponential distributions, but the timescale of run is one order of magnitude longer. (Courtesy of J. Saragosti)\label{fig:tracking}}
\end{center}
\end{figure}

\subsection{Chemotactic waves in micro-channels}

Concentration waves of chemotactic bacteria {\em E. coli} were described in the seminal article by  Adler \cite{adler_chemotaxis_1966}. They inspired the second article of Keller and Segel about mathematical modelling of chemotaxis  \cite{keller_traveling_1971}. We refer to \cite{tindall_overview_2008} for a featured review on the mathematical modelling of these chemotactic waves. Mesoscopic models describing this remarkable propagation phenomenon were proposed independently in \cite{xue_travelling_2010}, and \cite{saragosti_mathematical_2010,saragosti_directional_2011}. Relevance of modelling at the mesoscopic scale relies on tracking experiments. They reveal the directional distribution of individuals, and in particular the spatially dependent biases in the trajectories. 

\begin{figure} 
\begin{center}
\includegraphics[width=.8\linewidth]{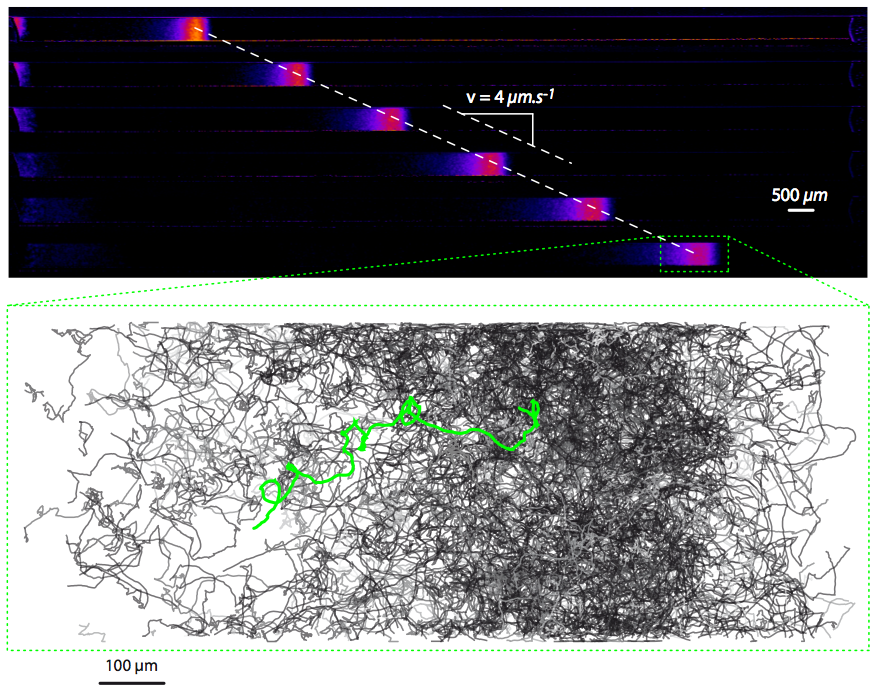}
\caption{\small Concentration waves of bacteria in a micro-channel \protect\cite{saragosti_directional_2011}. (top) A solitary wave of high density of bacteria travels from left to right with constant speed and almost constant profile. (bottom) Massive tracking experiments reveal the mesoscopic structure of the wave. (Courtesy of J. Saragosti)}
\label{fig:wave tracking}
\end{center}  
\end{figure} 

Model \eqref{eq:meso model}-\eqref{eq:lambda}-\eqref{eq:chemoattractant} has been validated on some experimental data in \cite{saragosti_directional_2011}, whereas its macroscopic diffusive limit has been validated in \cite{saragosti_mathematical_2010} on some other experimental set. It is worth noticing that the diffusive regime is valid when the chemotactic biases $\chi_S, \chi_N$ are small. This is not the case in the experiments presented in  \cite{saragosti_directional_2011}. We believe that in the latter, the mesoscopic scale is more appropriate to describe the bacterial population dynamics.

Experimental setting described in \cite{saragosti_mathematical_2010,saragosti_directional_2011} is as follows: cells (approx. $10^5$ bacteria {\em E. coli}) are initially located on the left side of a micro-channel after centrifugation. The width of the micro-channel is $500\mu m$, the height is $100\mu m$, and the total length is $2 cm$. The time span of experiments is about a few hours.   After short time, a significant fraction of the population moves towards the right side of the channel, at constant speed, within a constant profile (see Figure \ref{fig:wave tracking}). 

The goal of the present work is to construct travelling waves for the run-and-tumble equation, coupled to a pair of reaction-diffusion equations \eqref{eq:meso model}-\eqref{eq:lambda}-\eqref{eq:chemoattractant}. This is an original result, up to our knowledge. Note that equation \eqref{eq:meso model} is conservative. Hence, such travelling waves are in some sense solitary waves. They are very different from reaction-diffusion travelling waves that occur in the Fisher-KPP equation, for instance (see {\em e.g.} \cite{fisher_wave_1937,kolmogorov_etude_1937,aronson_multidimensional_1978}).

\subsection{Simplifying modelling assumptions}
The construction of travelling  waves for the system \eqref{eq:meso model}-\eqref{eq:lambda}-\eqref{eq:chemoattractant} requires some modelling simplifications. One crucial assumption is about the signal integration function $\boldsymbol\phi$. In order to maintain the spatial cohesion of the band, it is important that the function $\boldsymbol\phi$ has a sharp transition for small temporal variations. Accordingly, we make the following choice 
\[
\boldsymbol\phi = - \sign\,.
\]
This relies on some well documented amplification mechanism, see \cite[Chapter 11.
Gain Paradox]{berg_e._2004}. This specific choice imposes conditions on the pre-factors $\chi_S, \chi_N$, for the frequencies $\boldsymbol\lambda_S, \boldsymbol\lambda_N$ to be non-negative, namely $0\leq \chi_S,\chi_N\leq 1/2$. 

We further ignore any delay effects during the signal integration process, as the tumbling rate is a function of the time derivative \eqref{eq:lambda}. This is to say that bacteria are supposed to react to instantaneous variations of signal. This assumption is very restrictive, as compared to experimental data \cite{segall_temporal_1986,berg_e._2004}. However, including any memory effect, as in \cite{erban_signal_2005,xue_travelling_2010,perthame_derivation_2015} would make the analysis much more difficult. This is left for future perspective. 

In addition, assume that the post-tumbling velocity $v$ is drawn uniformly at random,   \beq \boldsymbol K\equiv 1\, . \label{eq:A1}\eeq
Assumption \eqref{eq:A1} ignores angular persistence during tumbling, as opposed to experimental evidence \cite{berg_e._2004}. 

Finally, we restrict to dimension $d = 1$, as we are looking for planar waves. 

We discuss possible relaxations of these assumptions in Section \ref{sec:perspectives}.

\subsection{The measure $\nu$ on the velocity set $V$}
We deal with two classes of probability measures separately, namely the case of a continuum of velocities, and the discrete case. 

\subsubsection{The continuous case}
\begin{description}
\item[(A1)] $\nu$ is absolutely continuous with respect to the Lebesgue measure. Moreover, the probability density function belongs to $L^p$ for some $p>1$. 
Precisely, we have:
\[
d\nu(v) = \omega(v)\,dv\, , \quad \omega\in L^p\, , \quad  p\in (1,\infty]\, , \quad  \text{$V = \supp\nu$ is compact}\,. 
\] 
\end{description} 

We do not include the case $p = 1$, for technical reasons that involve H\"older regularization by averaging lemma. On the other hand, the case $p=\infty$ is also excluded in the course of analysis, as it is borderline. However, it is included in the final result because any $L^\infty$ p.d.f. having a compact support belongs to all $L^p$ spaces. 

\paragraph{Examples:} the projection of the uniform measure on the unit circle $\SS^1$ onto the interval $[-1,1]$ under the hypothesis of axisymmetry writes $d\nu(v) = \frac1{\pi}(1 - v^2)^{-1/2}\, dv$. This imposes $p <2$. The same procedure in case of the unit sphere $\SS^2$ yields the uniform measure $d\nu(v) = \frac12\, dv$. 

During the course of analysis, the following assumption will be required in order to refine the asymptotic decay of solutions. 
\begin{description}
\item[(A1')] In addition to $\Ac$, the p.d.f. $\omega$ is bounded below by some $\omega_0>0$ on $V$. 
\end{description} 

\paragraph{Example:} the projection of the uniform measure on the two-dimensional unit ball $B(0,1)$ onto the unit interval $[-1,1]$ under the hypothesis of axisymmetry writes $d\nu(v) = \frac2\pi(1 - v^2)^{1/2}\, dv$. It belongs to $L^p$, but it vanishes on $v = \pm 1$. Similarly, the same procedure in case of the the three-dimensional unit ball $B(0,1)$  yields $d\nu(v) = \frac34(1 - v^2)\, dv$. 
\subsubsection{The discrete case}

\begin{description}
\item[(A2)] $\nu$ is a finite sum of dirac masses, 
\[ \nu = \sum_{i = 1}^N \omega_i \delta(v - v_i)\, , \quad (\forall i)\; v_i\in \R\, , \quad (\forall i)\;  \omega_i>0\, , \quad \sum_{i = 1}^N \omega_i = 1\,.\] 
\end{description}

The latter assumption  is motivated by numerical analysis, where $(\omega_i)_i$ are  quadrature weights. This case is not included in the main result. However, comprehensive numerical analysis performed in Section \ref{sec:discrete velocity} enables to find counter-examples for the existence and uniqueness of travelling waves. This completes our main result, as it shows that additional conditions on the parameters are mandatory, on the contrary to the macroscopic model obtained in the diffusion limit (see Section \ref{sec:diffusion limit}). 

A companion paper develops accurate well-balanced schemes for solving the Cauchy problem \cite{companion}. Consequences of the lack of existence and uniqueness are investigated in the long-time asymptotic of \eqref{eq:meso model}-\eqref{eq:lambda}-\eqref{eq:chemoattractant}.
We also refer to \cite{rousset_simulating_2013,filbet_inverse_2013,yasuda_monte_2015} for numerical analysis of variants of \eqref{eq:meso model}, based either on Monte Carlo simulations, or inverse Lax-Wendroff methods. 

\subsection{Notations and conventions}

As the problem is linear with respect to the cell density $f$, we can normalize the latter to have unit mass,
\beq\label{eq:f unit mass} \iint f(z,v)\, d\nu(v) dz = 1\, . \eeq

The space variable in the moving frame at speed $c$ is denoted by $z = x-ct$.

We denote by $\v0 = \max (\supp\nu)$ the maximal velocity. 

We assume that the amplitude of the tumbling rate is $\lambda_0 = 1$ without loss of generality. 

We adopt the following convention: the superscripts ${}^+$ and ${}^-$ denote the position of the velocity $v$ with respect to the speed $c$ (resp. $v>c$ and $v<c$), whereas the subscripts ${}_+$ and ${}_-$ denote the relative position with respect to the origin (resp. $z>0$ and $z<0$). For instance, $f_+^-(z,v)$ denotes the density of bacteria on the right side ($z>0$), moving towards the origin ($v<c$). Similarly, $\rho_+^-(z) = \int_{\{v<c\}} f_+(z,v)\, d\nu(v)$ denotes the spatial density of the same group of bacteria. 

The tumbling rate can take only four values $1\pm \chi_S\pm \chi_N$, depending on the sign of the gradients. Due to the particular ansatz that underlies the existence of travelling waves, $T$ is can only change value if $v$ crosses $c$, or if $z$ crosses the origin. Accordingly, the four possible values are denoted by $T_+^+, T_+^-, T_-^+, T_-^-$. We refer to \eqref{eq:rule of sign}, and Figure \ref{fig:sign convention} for the rule of signs. 

We also denote in short $\chi_+,\chi_-$ the effective chemotactic biases on each side of the origin.  
\[\begin{cases} 
\chi_+ = \chi_S - \chi_N \in (-1/2,1/2) \medskip\\
\chi_- = \chi_S + \chi_N \in (0,1)
\end{cases} \]

We denote by $[\cdot]_\theta = |\cdot|_{\mathcal C^{0,\theta}}$ the H\"older  semi-norm.  

\subsection{Existence of travelling waves}

We address the existence of travelling wave solutions for the coupled problem \eqref{eq:meso model}-\eqref{eq:lambda}-\eqref{eq:chemoattractant}. The wave speed is denoted by $c$, and the spatial coordinate in the moving frame is denoted by  $z = x-ct$. We keep the notation $f$ for some particular solution of  \eqref{eq:meso model} of the form $f(x-ct,v)$. Travelling waves are solutions of the following stationary problem in the moving frame, 
\beq\label{eq:TW}
\begin{cases}
  \displaystyle   (v-c) \partial_z f(z,v) =  \int   T(z,v'-c)f(z,v') \, d\nu(v') - T(z,v-c)f(z,v)  
\medskip\\
 - c\partial_z S(z)  -  D_S \partial^2_{z} S(z)  + \alpha S(z) = \beta \rho(z) \medskip\\
- c\partial_z N(z)  - D_N \partial_{z}^2 N(z) = - \gamma \rho(z) N(z)
\end{cases}
\eeq
where the tumbling rate $T$ is given by
\beq \label{eq:T}
T(z,v-c) =   1 - \chi_S \sign((v-c) \partial_z S(z)) - \chi_N\sign((v-c) \partial_z N(z))   \, . 
\eeq
Note that we have reduced to $\lambda_0 = 1$ without loss of generality. Also, we restrict to the case $c\geq 0$ in order to fix the global direction of the wave.

As the measure $\nu$ may vary during the analytical procedure, it is fruitful to write a weak formulation of the first equation in \eqref{eq:TW}, so that the role of the measure $\nu$ clearly appears. 

\begin{definition}[Weak formulation]
A weak solution $f$ of the first equation in \eqref{eq:TW} is such that for all test function $\varphi\in \mathcal{D}(\R\times \R)$, 
\beq\label{eq:weak formulation}  
- \iint (v-c) f(z,v) \partial_z \varphi(z,v)\, d \nu(v)dz = 
\iiint T(z,v'-c) f(z,v')\left(\varphi(z,v) - \varphi(z,v') \right)\, d\nu(v')d \nu(v)dz
\eeq
\end{definition}

We begin with the existence of stationary clusters, in the absence of nutrient. 

\begin{theorem}[Stationary cluster]\label{theo:TW stat}
Assume $\chi_S\in (0,1/2)$, and $\chi_N = 0$. Under assumption $\Ac$, there exist symmetric,  positive functions $(f,S)\in \left( L^1\cap L^\infty (\R\times V)\right)\times \mathcal C^2(\R)$ solution of the  stationary wave problem 
\[ 
\begin{cases}
  \displaystyle   v \partial_x f(x,v) =  \int   T(x,v')f(x,v') \, d\nu(v') - T(x,v)f(x,v)  
\medskip\\
-  D_S \partial^2_{x} S(x)  + \alpha S(x) = \beta \rho(x)\medskip\\
T(x,v) =   \dfrac12 - \chi_S \sign(v \partial_x S(x)) \, , 
\end{cases}
\]
Assume in addition that $\Abc$ is satisfied. Then,  there exists an exponent $\lambda>0$, and a velocity profile $F(v)$ such that $e^{\lambda z}f(z,v)$ (resp. $e^{-\lambda z}f(z,v)$)   converges exponentially fast to $F(v)$ as $z\to +\infty$ (resp. to $F(-v)$ as $z\to -\infty$). 
\end{theorem}

In order to establish our main result, existence of travelling waves driven by consumption of the nutrient, we need to impose some extra conditions on the parameters. As the conditions are not simple to state, we refer to the lines of the proof for the details. 

The first two conditions are linked with the coupling with $S$. The first condition \eqref{eq:condition c_*} writes
\beq \label{eq:cond 1} \dfrac{c_* + \sqrt{(c_*)^2 + 4\alpha D_S}}{c_* + \sqrt{(c_*)^2 + 4\alpha D_S} + 2D_S \lambda_-(c_*)} \leq \text{explicit constant}\,, \eeq
where $c_*$ is defined implicitly such that
\[  \int_{\{v<c_*\}} \dfrac{v-c_*}{T_+^-}\, d\nu(v) + \int_{\{v>c_*\}} \dfrac{v-c_*}{T_+^+}\, d\nu(v) = 0\, , \]
and $\lambda_-(c_*)$ is the smallest positive root of 
\[ \int \dfrac{v-c_*}{T_-(v-c_*) + \lambda (v-c_*)}\, d\nu(v)  = 0 \, . \]
The second condition \eqref{eq:condition c^*} writes,
\beq  \label{eq:cond 2} 
\dfrac{-c^* + \sqrt{(c^*)^2 + 4\alpha D_S}}{-c^* + \sqrt{(c^*)^2 + 4\alpha D_S} + 2D_S \lambda_+(c^*)}
\leq \text{explicit constant}\, , 
\eeq
where $c^*$ is defined implicitly such that
\[  \int_{\{v<c^*\}} \dfrac{v-c^*}{T_-^-}\, d\nu(v) + \int_{\{v>c^*\}} \dfrac{v-c^*}{T_-^+}\, d\nu(v) = 0\, , \]
and $\lambda_+(c^*)$ is the smallest positive root of 
\[ \int \dfrac{v-c^*}{T_+(v-c^*) - \lambda (v-c^*)}\, d\nu(v)  = 0 \, . \]
(see Sections \ref{sec:decoupled} and \ref{sec:matching S} for more details). 

The third condition \eqref{eq:condition c=0} is related to the coupling with $N$. It reads
\beq  \label{eq:cond 3} \text{Either $\chi_N\geq \chi_S$, or}\quad \dfrac{ \sqrt{\alpha/ D_S} +   \lambda_+(0)}{\sqrt{\alpha / D_S} +  \lambda_-(0)} \leq \text{explicit constant}\, , \eeq
where $\lambda_-(0)$ and $\lambda_+(0)$ are the smallest positive roots of, respectively,
\[ \int \dfrac{v}{T_-(v) + \lambda v}\, d\nu(v)  = 0\, , \quad \int \dfrac{v}{T_+(v) - \lambda v}\, d\nu(v)  = 0\, . \]
(see Sections \ref{sec:decoupled} and \ref{sec:matching N} for more details).

\begin{remark}
There exist parameters that fulfil all conditions \eqref{eq:cond 1}-\eqref{eq:cond 2}-\eqref{eq:cond 3}. Condition \eqref{eq:cond 3} is satisfied if $\chi_N\geq \chi_S$ (obviously), or if  $ \sqrt{\alpha/ D_S} \ll \lambda_+(0) \ll \lambda_-(0)$. But, $\lambda_+(0) / \lambda_-(0)\to 0$ as $\chi_N \nearrow \chi_S$. Conditions \eqref{eq:cond 1} and \eqref{eq:cond 2} are both satisfied if  $\sqrt{\alpha/ D_S} \ll \lambda_-(c_*), \lambda_+(c^*)$, together with $D_S \gg c_*/\lambda_-(c_*), c^*/\lambda_+(c^*)$.
\end{remark}

\begin{theorem}[Travelling wave]\label{theo:kin TW}
Assume $(\chi_S, \chi_N)\in (0,1/2)\times[0,1/2)$. Under assumption $\Ac$, and conditions \eqref{eq:cond 1}-\eqref{eq:cond 2}-\eqref{eq:cond 3}, there exist a non-negative $c$, and positive functions $(f,S,N)\in \left( L^1\cap L^\infty (\R\times V)\right)\times \mathcal C^2(\R)\times \mathcal C^2(\R)$ solution of the  travelling wave problem \eqref{eq:TW}-\eqref{eq:T}.\\
Assume in addition that $\Abc$ is satisfied. Then,  there exist two exponents $\lambda_\pm>0$, and two velocity profiles $F_\pm(v)$ such that $e^{\lambda_+ z}f(z,v)$ (resp. $e^{-\lambda_- z}f(z,v)$)   converges exponentially fast to $F_+(v)$ as $z\to +\infty$ (resp. to $F_-(v)$ as $z\to -\infty$). 
\end{theorem} 

As mentioned above, the discrete case $\Ad$ is not included in Theorem \ref{theo:kin TW}. Nonetheless, it will be discussed with lots of details in Section \ref{sec:discrete velocity}.

Counter-examples for both existence and uniqueness of the travelling waves have been found numerically, based on the analysis performed in Section \ref{sec:discrete velocity}.

%
%

\subsection{Acknowledgements} The author is indebted to B. Perthame who introduced him to kinetic equations with applications in chemotaxis, and  biology in general. This work originated from a fruitful collaboration with J. Saragosti, A. Buguin and P. Silberzan. The author is grateful to N. Bournaveas, C. Emako, and N. Vauchelet for preliminary discussions about this work. He has benefited from stimulating discussions with E. Bouin, N. Caillerie, F. Golse, G. Raoul, Ch. Schmeiser, M. Twarogowska, and J. Vovelle. The author is particularly grateful to L. Gosse for his enlightening explanation of the numerical analysis of \eqref{eq:meso model}.\\ 
This work was supported by the LABEX MILYON (ANR-10-LABX-0070) of Universit\'e de Lyon, within the program "Investissements d'Avenir" (ANR-11-IDEX-0007) operated by the French National Research Agency (ANR). This project has received funding from the European Research Council (ERC) under the European Union’s Horizon 2020 research and innovation programme (grant agreement No 639638).

\section{Sketch of proof}

As the complete proof of Theorem \ref{theo:kin TW} is quite technical, the strategy of proof is sketched below. We begin with the corresponding macroscopic problem, obtained in the diffusion limit. The latter is by far easier to handle with. This will drive the analysis at the mesoscopic scale. Then, we draw two main obstacles arising at the mesoscopic scale. Section \ref{ssec:sketch proof} contains the main argument to overcome the first obstacle in full generality. The second obstacle yields restriction on the parameters, as conditions \eqref{eq:cond 1}-\eqref{eq:cond 2}-\eqref{eq:cond 3}.  

\subsection{Construction of travelling waves at the macroscopic level}\label{sec:diffusion limit}
The formal diffusion limit of \eqref{eq:meso model}, derived when the biases $\chi_S, \chi_N$ are small, is given by the following drift-diffusion/reaction-diffusion coupled system \cite{saragosti_mathematical_2010}, 
\begin{equation}
\label{eq:macro model}
\begin{cases}
\partial_t \rho -  D_\rho \partial^2_{x} \rho + \partial_x \left( \rho \left( \chi_S \sign(\partial_x S) +  \chi_N \sign(\partial_x N) \right)\right) = 0\medskip\\
\partial_t S  -  D_S \partial^2_{x} S + \alpha S = \beta \rho \medskip\\
\partial_t N  -  D_N \partial^2_{x} N  = - \gamma \rho N
\end{cases}
\end{equation}
Let recall the existence and uniqueness of one-dimensional solitary waves. 
\begin{theorem}[\cite{saragosti_mathematical_2010}]\label{theo:TW macro}
There exist a speed $c>0$, and two positive values $N_{-}<N_+$, such that the system \eqref{eq:macro model} admits travelling wave solutions $(\rho(z - ct), S(z - ct), N(z-ct))$ with the following boundary conditions:
\begin{equation*}
\lim_{|z|\to \infty} \rho = \lim_{|z|\to \infty}  S = 0 \, , \quad \lim_{z\to - \infty} N = N_{-} < N_+ = \lim_{z\to + \infty} N\, .
\end{equation*}
Assume in addition that $\partial_z S$ changes sign only once, and that $\partial_z N$ is positive everywhere. Then, the speed $c>0$ is uniquely determined by the following implicit relation,
\beq\label{eq:speed}
\chi_N - c = \chi_S\dfrac{c}{\sqrt{c^2 + 4 \alpha D_S}}\, .
\eeq 
\end{theorem}

\begin{remark}
We believe that there cannot exist travelling waves for which $\partial_z S$ (or $\partial_z N$) changes sign several times. Alternatively speaking, we conjecture that  uniqueness holds true in Theorem \ref{theo:TW macro} without extra conditions, up to standard transformations (multiplication by a constant factor, translation, symmetry). 
\end{remark}

It is instructive to recall the proof of Theorem \ref{theo:TW macro}, as it will guide the proof of Theorem \ref{theo:kin TW}. The former is based on explicit computations. However, this will no longer be the case at the mesoscopic level.


\begin{proof}
Firstly, system \eqref{eq:macro model} is written in the moving frame $z = z - ct$. We keep the notations $\rho,S,N$ for the unknown functions. 
\begin{equation}
\label{eq:macro model TW}
\begin{cases}
-c \partial_z \rho -  D_\rho \partial^2_{z} \rho + \partial_z \left( \rho \left( \chi_S \sign(\partial_z S) +  \chi_N \sign(\partial_z N) \right)\right) = 0\medskip\\
-c \partial_z S  -  D_S \partial^2_{z} S + \alpha S = \beta \rho \medskip \\
 -c \partial_z N  -  D_N \partial^2_{z} N  = - \gamma \rho N
\end{cases}
\end{equation}
The starting point consists in assuming that the communication signal $S$ reaches a unique maximum, say at $z = 0$, and that the nutrient $N$ is increasing. The validation of this ansatz {\em a posteriori}, will set the equation for $c$ \eqref{eq:speed}. 

This ansatz determines the sign of the chemical gradients. It enables to decouple the drift-diffusion equation from the reaction-diffusion system in \eqref{eq:macro model TW}. 
Accordingly, the cell density $\rho$ is solution of the following equation,
\[ -c \partial_z \rho - D_\rho \partial^2_{z}\rho + \partial_z \left( \rho \left( \chi_S \sign(-z) + \chi_N \right)\right) = 0\, . \]
The density profile is explicitly given as a combination of two exponential distributions,
\beq\label{eq:cell density profile}
\rho(z) = 
\begin{cases}
\exp\left(\dfrac{ - c + \chi_S + \chi_N }{D_\rho} z \right) & \text{for $z <0$, \ie both signals are attractive,}\medskip\\
\exp\left(\dfrac{ - c - \chi_S + \chi_N }{D_\rho} z \right) & \text{for $z >0$, \ie the signals have opposite contributions.}
\end{cases}
\eeq
We denote $\lambda_-(c) =  (- c + \chi_S + \chi_N)/D_\rho $, and $\lambda_+(c) = - ( - c - \chi_S + \chi_N )/D_\rho$ the exponents on both sides. We impose positivity of both exponents,  $\lambda_-(c), \lambda_+(c)>0$. 
 
In a second step, the concentration $S$ is computed  through the elliptic equation, 
\beq\label{eq:elliptic} -c \partial_z S - D_S \partial^2_{z} S + \alpha S = \beta \rho\, . \eeq
The solution is given by $S = \mathcal S *(\beta \rho)$, where $\mathcal S$ is the Green function associated with the elliptic operator $-c \partial_z - D_S \partial^2_{z} +  \alpha $. We have
\[
\mathcal S(z) = 
\begin{cases}
\mathcal S_0 \exp\left(\dfrac{ - c + \sqrt{c^2 + 4 \alpha D_S}}{2D_S} z \right) & \text{for $z <0$,}\medskip\\
\mathcal S_0\exp\left(\dfrac{ - c - \sqrt{c^2 + 4 \alpha D_S}}{2D_S} z \right) & \text{for $z >0$.}
\end{cases}
\]
The exact value of the  normalization  factor $\mathcal S_0$ does not matter here. 
We introduce $\mu_-(c) = \left (- c + \sqrt{c^2 + 4 \alpha D_S}\right )/(2D_S)$, and $\mu_+(c) = \left (  c + \sqrt{c^2 + 4 \alpha D_S}\right )/(2D_S)$ the exponents on both sides. 

On the other hand, according to the general result stated in Proposition \ref{prop:monotonicity S} below, the monotonicity of $S$ is basic,  \ie first increasing, then decreasing, just because $\rho$ is so. However, the change of monotonicity is not necessarily located at $z=0$, as in the initial ansatz.

The quantity $\partial_z S(0)$ enables to locate the maximum of $S$. If $\partial_z S(0)>0$, then $S$ reaches its maximum for $z>0$, whereas if $\partial_z S(0)<0$, then $S$ reaches its maximum for $z<0$. According to the previous ansatz, the maximum of $S$ must be located at $z = 0$. 

Consequently, we need to compute the value of $\partial_z S(0)$, denoted by $\Upsilon(c)$, in order to validate the initial ansatz. We have,
\begin{align*} 
\Upsilon(c) = \partial_z S(0) & = \beta \int_{\R} \partial_z \mathcal S(-z) \rho(z) \, dz \\
& = \beta \mathcal S_0 \left(  \int_{z<0} - \mu_+(c) e^{\mu_+(c)z}e^{\lambda_-(c)z}\, dz + \int_{z>0}  \mu_-(c) e^{-\mu_-(c)z}e^{-\lambda_+(c)z}\, dz \right)\\
& = \beta \mathcal S_0 \left( - \dfrac{\mu_+(c) }{\mu_+(c) + \lambda_-(c) } + \dfrac{\mu_-(c)}{ \mu_-(c) + \lambda_+(c)} \right)
\\
& = \beta \mathcal S_0 \left( - \dfrac{1 }{1 + \lambda_-(c)/\mu_+(c) } + \dfrac{1}{ 1 + \lambda_+(c)/\mu_-(c)} \right)\, .
\end{align*}
It is immediate to see that the function $\Upsilon$ is decreasing with respect to $c$, as the ratio $\lambda_-/\mu_+$ is decreasing ($\lambda_-$ is decreasing, whereas $\mu_+$ is increasing), and the ratio $\lambda_+/\mu_-$ is increasing ($\lambda_+$ is increasing, whereas $\mu_-$ is decreasing). Analysis of the extremal behaviours reveal that there is indeed a unique root $c$ of $\Upsilon$. Moreover, it belongs to the interval $(0,\chi_N)$. Indeed, we have $\mu_-(0) = \mu_+(0) = 2\sqrt{\alpha D_S}$, but $ \lambda_-(0) = (\chi_S + \chi_N)/D_\rho >\lambda_+(0) = (\chi_S - \chi_N)/D_\rho$.  Therefore, $\Upsilon(0) >0$. On the other hand, we have $\Upsilon(\chi_N) < 0$, since $\lambda_-(\chi_N) = \lambda_+(\chi_N) = \chi_S/D_\rho$, whereas $\mu_-(\chi_N) < \mu_+(\chi_N)$. 

Finally, notice that the dispersion relation $\Upsilon(c) = 0$ simply reads
\[ \lambda_+(c) \mu_+(c) = \lambda_-(c) \mu_-(c)\, ,  \]
which is equivalent to \eqref{eq:speed}. 
\end{proof}

\subsection{Main obstacles at the mesoscopic level.}
When trying to mimic the previous proof at the mesoscopic level, several difficulties arise. Some are technical, and dedicated tools from kinetic theory can overcome them (\emph{e.g.} averaging lemma). However, at least two main obstacles require much more work. We describe them below. The first obstacle can be resolved in full generality. However, the second obstacle requires additional conditions on the parameter. This is the main reason why Theorem \ref{theo:kin TW} is more restrictive than its macroscopic version Theorem \ref{theo:TW macro}. 

\subsubsection{Non monotonicity of the mesoscopic density $f(z,v)$ with respect to $z$.}\label{sec:Non monotonicity}
As emphasized in Section \ref{sec:diffusion limit}, the construction of  travelling waves for the macroscopic problem relies on the fact that $\partial_z S$ changes sign exactly once. 
To summarize, it is consistent to assume that $\partial_z S$ changes sign exactly once, say at $z = 0$ w.l.o.g. Then, explicit computations prove that $\rho$ reaches its maximum value at the same location.
It is an immediate consequence that $S$, being solution of the elliptic problem \eqref{eq:elliptic} with $\rho$ as a source term, has a unique critical value, corresponding to a maximum for $S$. However, this might not coincide with $z = 0$. In fact, there is a unique $c$ such that the maximum of $S$ is located at $z = 0$.


We follow the same approach at the mesoscopic level. We freeze $S$ and $N$, simply assuming the same monotonicity conditions as above. Then, we solve $f$ in the first line of system \eqref{eq:TW}, and we deduce the spatial density $\rho = \int f\, d\nu$. At this stage, however, it is not immediately clear that $S$, being solution of the same elliptic problem \eqref{eq:elliptic} with $\rho$ as a source term, has a unique critical value.  This would hold true provided $\rho$ has basic monotonicity (increasing, then decreasing), as in the macroscopic case. In fact, such monotonicity would be an equivalent requirement,  without any extra condition on $D_S$ and $\alpha$, since the Green function $\mathcal S$ converges to a Dirac mass for $c = 0$  and  $\sqrt{\alpha/D_S}\to +\infty$.
  
The first major problem can be rephrased as follows: 
\begin{quote}
{\em Let $f(z,v)$ be the density obtained from the first equation in system \eqref{eq:TW} under the hypotheses that $\partial_z S$ changes sign at $z  =0$ only, and $\partial_z N$ is positive. Prove that $\rho(z) = \int f(z,v)\, d\nu(v)$ is increasing for $z<0$, and decreasing for $z>0$.}
\end{quote}

\begin{figure}[t]
\begin{center}
\includegraphics[width = 0.48\linewidth]{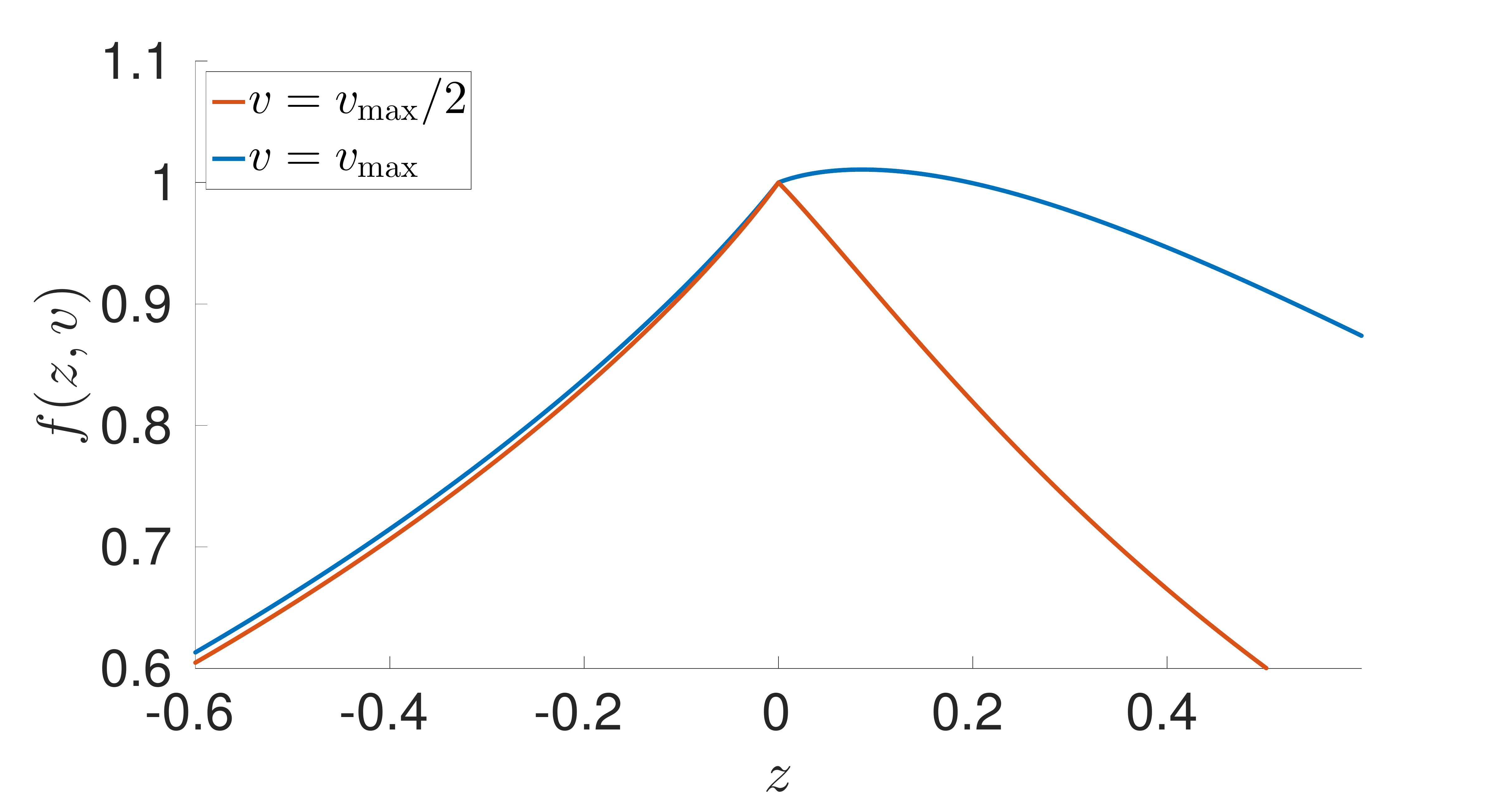}\; 
\includegraphics[width = 0.48\linewidth]{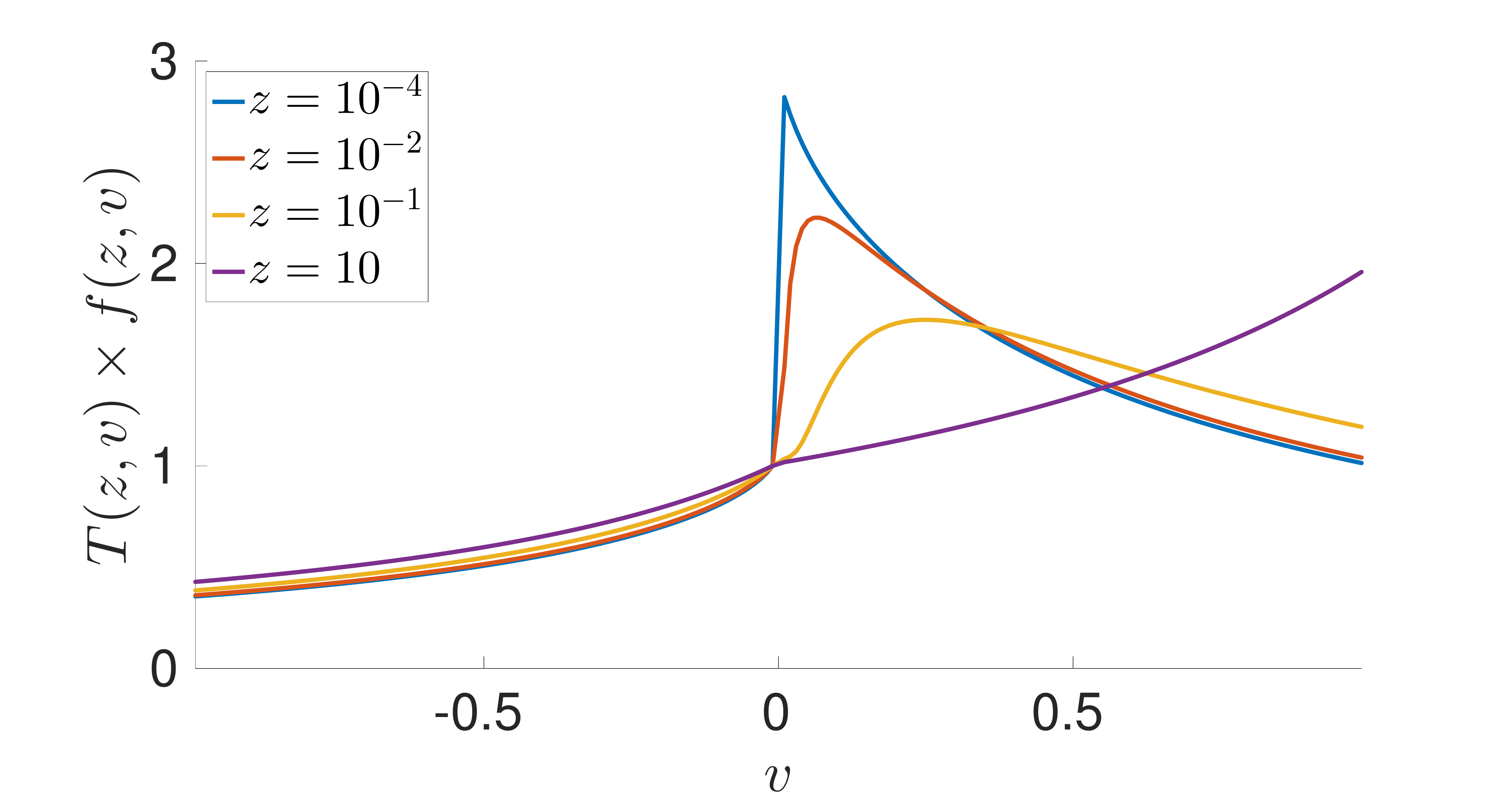}
\caption{\small Overshoot phenomenon. Here, $V = (-1,1)$ with uniform measure $\nu$, $\chi_S = 0.48$ and $\chi_N = 0$. (Left) Shape of the spatial profile $z\mapsto f(z,v)$ for two different values of $v$: large velocity ($v = v_{\rm max}$; blue line) and intermediate velocity  ($v = v_{\rm max/2}$; red line). Clearly, the maximum with respect to $z$ is shifted to the right of the origin, so that $\partial_z f(z,v)$ has no sign for $z>0$. In fact, this can happen for large velocity only, and compensations in the integral are sufficient to give a constant sign to $\partial_z\rho$. (Right) To decipher the compensations, it is necessary to unravel the shape of velocity profiles. Four examples are shown, for increasing values of $z$ (profiles are normalized to have value 1 at $v = 0$). Although the shape is simple for negative velocity, positive velocities exhibit a non monotonic transition for $z$ close to the origin (blue line) to large $z$ (magenta line). It is an analytical challenge to understand this transition. This is the purpose of Section \ref{sec:monotonicity}.
}
\label{fig:overshoot}
\end{center}
\end{figure}

Obviously, the monotonicity of $\rho$, say $\frac {d\rho}{dz} (z) <0$ for all $z>0$, would follow from the monotonicity of $f$, if we were able to show that $\frac {\partial f}{\partial z}(z,v) <0$ for all $z>0$, and all $v\in V$. However, the latter is not expected to hold true in full generality, as can be seen on numerical simulations for $\chi_S$ large, $\chi_N = 0$, and accordingly $c = 0$ (see Figure \ref{fig:overshoot}). 
Therefore, the monotonicity of $\rho(z) = \int f(z,v) \, d\nu(v)$ for $z>0$ can only result from compensations when integrating with respect to velocity.

Our problem can be recast as follows,
\begin{quote}
{\em Let $f(z,v)$ be the density obtained from the first equation in system \eqref{eq:TW} under the hypotheses that $\partial_z S$ changes sign at $z  =0$ only, and $\partial_z N$ is positive. Describe the velocity profiles in order to show that compensations yield the appropriate monotonicity for $\rho$ when integrating with respect to velocity.}
\end{quote}

It turns out that such compensations always occur, with no additional condition on the parameters. We are able to prove that inappropriate monotonicity of $f$ may  arise for  large positive velocities (for $z>0$). However, it is compensated by small positive velocities. Moreover, negative velocities contribute with the appropriate sign. We shall decipher velocity profiles  in order to capture these compensations.

\subsubsection{Non monotonicity of the location of the maximal value of $S$ with respect to $c$.}\label{ssec:sketch proof}
The conclusion of the proof of Theorem \ref{theo:TW macro} is facilitated both by the monotonicity of $\Upsilon$, and by the unambiguous extremal behaviours ($\Upsilon(0)>0$, $\Upsilon(\chi_N)<0$).

It turns out that, none of these two facts is generally true at the mesoscopic level. The corresponding function $\Upsilon(c)$ may vary in a non monotonic way, so that it crosses $\Upsilon = 0$ several times. This clearly prevents uniqueness of the speed $c$. 

On the other hand, there exists a range of parameters for which $\Upsilon(0)<0$, and $\Upsilon(c^*)<0$, where $c^*$ is the maximal admissible value for $c$. Furthermore, $\Upsilon(c)<0$ for all admissible values of $c$. This prevents existence of travelling waves in our framework. 


This justifies {\em a posteriori} the additional conditions on parameters \eqref{eq:cond 1}-\eqref{eq:cond 2}-\eqref{eq:cond 3}. We did our best to make all these conditions quantitative in the course of analysis. However, we believe that existence and uniqueness hold for a much larger class of parameters.

\subsection{Strategy of proof}

Here, we sketch briefly the argument for resolving the first issue raised in Section \ref{sec:Non monotonicity}. 

We use a kind of homotopy argument, by deforming the measure $\nu$. If $\nu$ is a symmetric combination of two Dirac masses, this is the two velocity case, which is very similar to the macroscopic problem, for which  monotonicity is straightforward, see \eqref{eq:cell density profile}. When the support of $\nu$ is concentrated around two symmetric velocities, it is not difficult to prove that monotonicity still holds true. 

As $\nu$ changes continuously, we establish that $\rho$ cannot change monotonicity. There is some subtlety here because $\rho$ does not appear in the kinetic equation in \eqref{eq:TW}. The natural macroscopic quantity to look at is indeed the density of tumbling events per unit time, namely
\beq\label{eq:I} I(z) = \int   T (z,v-c)  f(z,v)\, d\nu(v)\, , \eeq
so that the first line in \eqref{eq:TW}
can be rewritten as
\begin{equation}\label{eq:f I}
(v-c) \partial_z f(z,v) =  I(z) - T(z,v-c)f(z,v)\, . 
\end{equation}
The function $f$ can be reconstructed entirely from $I$ by the Duhamel formulation along characteristics lines, for $z>0$:
\beq\label{eq:duhamel} \begin{cases}
\displaystyle(\forall v<c)\quad  f_+^-(z,v) = \int_0^{+\infty}  I_+(z - s(v-c)) \exp(-T_+^- s)\, ds\medskip\\
\displaystyle(\forall v>c)\quad f_+^+(z,v) = \left( \int_0^{+\infty}  I_-(-s(v-c)) \exp(-T_-^+ s)\, ds \right)\exp\left(-T_+^+\frac z{v-c}\right)\medskip\\ 
\displaystyle \quad\qquad\qquad\qquad\qquad\qquad\qquad\qquad+ \int_0^{\frac z{v-c}}  I_+(z - s(v-c)) \exp(-T_+^+ s)\, ds
\end{cases} \eeq
and similarly for $z<0$.

The integral quantity $I$ \eqref{eq:I} is a combination of the integral over $\{v<c\}$, and the integral over $\{v>c\}$, with different weights ($1 \pm \chi_S \pm \chi_N$, see Figure~\ref{fig:sign convention}). Accordingly, we focus on the quantities 
\[  \rho^+(z) = \int_{\{v<c\}} f(z,v)d\nu(v) \, , \quad \rho^-(z) = \int_{\{v<c\}} f(z,v)d\nu(v)\, , \]
because they are common to both $\rho$ and $I$, according to the following rules
\[
\begin{cases} \rho = \rho^- + \rho^+ \medskip \\
I = T^- \rho^- + T^+ \rho^+\, . 
\end{cases} 
\]
The key point is to establish the following enhancement of monotonicity: 
\begin{quote}
{\em Assume that both $\rho^+$ and $\rho^-$ satisfy 
\[ \begin{cases}  \dfrac{d\rho^+}{dz}(z) \geq 0\quad  (\text{resp. }\leq 0) & \text{for $z<0$} \quad (\text{resp. }\text{for $z>0$}) \medskip\\
 \dfrac{d\rho^-}{dz}(z) \geq 0\quad  (\text{resp. }\leq 0) & \text{for $z<0$} \quad (\text{resp. }\text{for $z>0$})
\end{cases}   \]
Then inequalities in the large are in fact strict inequalities. 
Therefore, $\rho^+$ and $\rho^-$ necessarily keep the same monotonicity as the measure $\nu$ varies continuously. 
}
\end{quote}

As mentioned in Section \ref{sec:Non monotonicity}, it is mandatory to understand compensations through the refined description of velocity profiles. The argument goes as follows, say for $z>0$. Assume that both $\rho^+_+$ and $\rho^-_+$ are non-increasing functions. We deduce from the Duhamel formula \eqref{eq:duhamel} that $T_+^-f_+^-(z,v)<I_+(z)$ for $v<c$. Therefore, 
\[ \dfrac {d\rho^-_+}{dz}(z) = \int_{\{v>c\}} \dfrac1{v-c}\left( I_+(z) - T_+^-f_+^-(z,v)  \right)\, d\nu(v) < 0\, . \]
On the other hand, we have
\beq\label{eq:sketch rho+} \dfrac {d\rho^+_+}{dz}(z) = \int_{\{v<c\}} \dfrac1{v-c}\left( I_+(z) - T_+^+ f_+^+(z,v)  \right)\, d\nu(v) \, . \eeq
The quantity $I_+ - T_+^+f_+^+$ has no sign in general, but we are able to prove that, whenever it has the inappropriate sign (here, positive), this can happen for large velocities only. Due to the decreasing weight $(v-c)^{-1}$ in the integral formula \eqref{eq:sketch rho+}, the following deduction holds true,
\[ \int_{\{v>c\}} \left(  I_+(z) - T_+^+f_+^+(z,v)  \right)\, d\nu(v) < 0 \quad \Rightarrow \quad \dfrac {d\rho^+_+}{dz}(z) < 0 \, , \]
On the other hand, we deduce from the  very definition of $I$ \eqref{eq:I} that
\[  \int_{\{v>c\}} \left(  I_+(z) - T_+^+f_+^+(z,v)  \right)\, d\nu(v) = - \int_{\{v<c\}} \left( I_+(z) - T_+^-f_+^-(z,v)  \right)\, d\nu(v) \, . \]
The latter quantity is negative because $T_+^-f_+^-(z,v)<I_+(z)$ for $v<c$. QED 

Of course, there is some additional work to ensure that monotonicity is indeed the appropriate one for large $z$, in order to get enough compactness in the argument. Also, regularity of the macroscopic quantities is required in order to justify the use of differential calculus.


\section{Preliminary results}

In this section, we make the following {\em a priori} hypotheses which enable to decouple   system \eqref{eq:TW}: 
\beq \label{eq:ansatz}
\begin{cases}
 \text{$c$ is given in a suitable interval $(c_*,c^*)$, defined in \eqref{eq:def c_*}-\eqref{eq:def c^*}}\,, \medskip\\
 \text{$(\forall z<0)\;   \partial_z S(z)>0$, and $(\forall z>0)\;   \partial_z S(z)<0$}\,, \medskip\\
\text{$(\forall z)\;  \partial_z N(z)>0$}
\end{cases}
\eeq
Accordingly, we are reduced to examine the following linear equation, 
\beq (v-c)\partial_z f(z,v) = \int T(z,v'-c) f(z,v')\, d\nu(v') - T(z,v-c) f(z,v)\, ,  
\label{eq:linear c}
\eeq
where $T(z,v-c) = T_\pm^\pm  =   1\pm \chi_S \pm \chi_N$, with the sign convention as follows  
(see also Figure~\ref{fig:sign convention}),
\begin{equation}\label{eq:rule of sign}
\begin{cases}
T_-^- = 1 + \chi_S + \chi_N\, ,  & T_-^+ = 1 - \chi_S - \chi_N\quad \text{(both signals are attractive),} \medskip\\
T_+^- = 1 - \chi_S + \chi_N\, ,  & T_+^+ = 1 + \chi_S - \chi_N\quad \text{(signals have opposite contributions).} 
\end{cases}
\end{equation}

\begin{figure}
\begin{center}
\includegraphics[width = .8\linewidth]{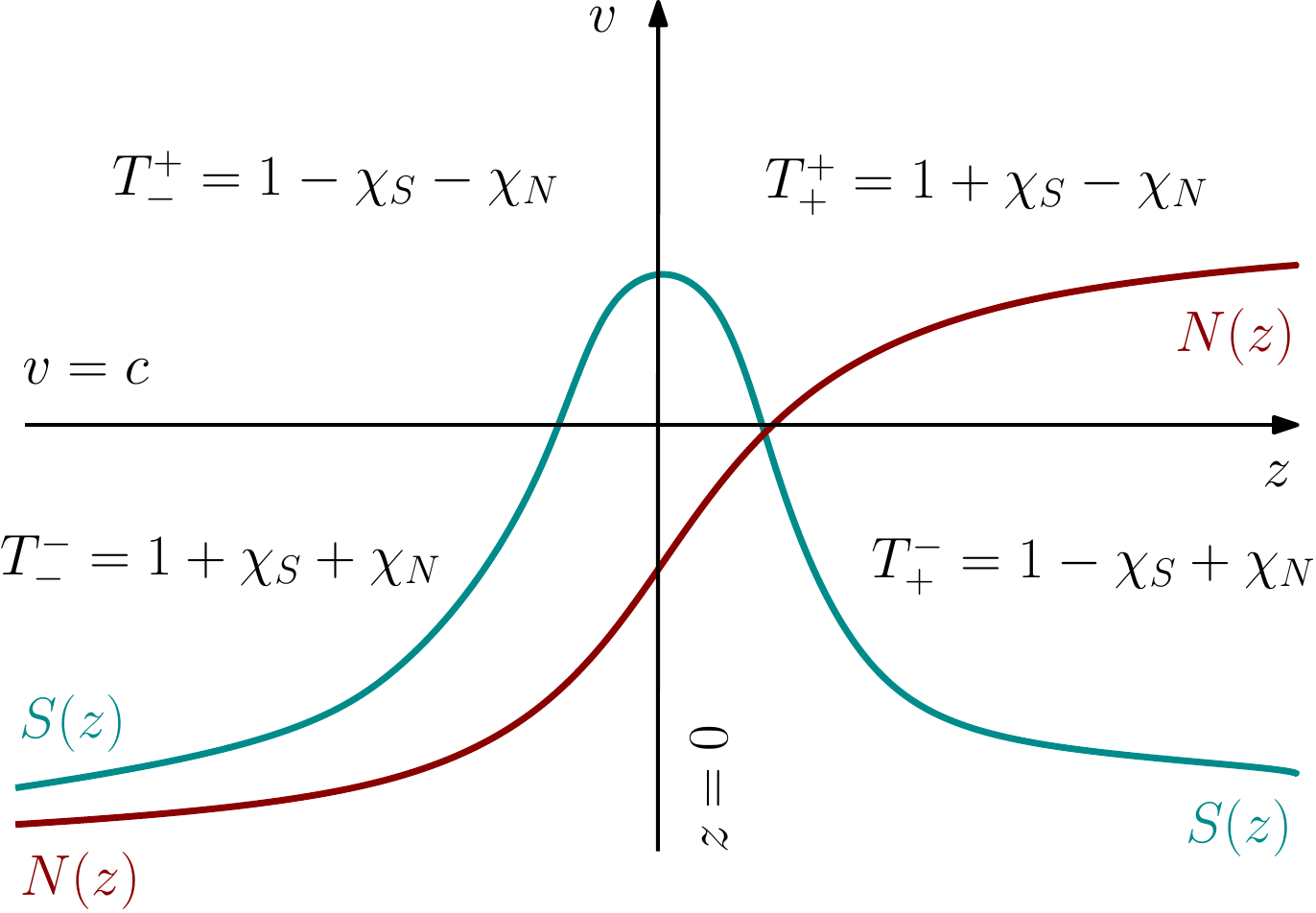}
\caption{\small Expression of the tumbling kernel $T_\pm^\pm$  depending on the signs of $z$ and $v-c$. The expected profiles of $S$ and $N$ are plotted in order to get the correct value of $T_\pm^\pm$ at a glance. Note the dependency with respect to $c$, as some signs change relatively to $v-c$ from bottom to top of the $(z,v)$ plane. }
\label{fig:sign convention}
\end{center}
\end{figure}

\subsection{Solution of the uncoupled problem}
\label{sec:decoupled}

Throughout this section, we assume that  $\Ac+\Abc$ is satisfied. We shall make clear why we need uniform boundedness from below on the support of $\nu$. Assumption $\Abc$ will be removed later on by some approximation procedure, at the expense of missing the refined asymptotic description of the solution as $|z|\to \infty$. 

We proceed as in \cite{calvez_confinement_2015}, where the existence of a stationary density solution to \eqref{eq:linear c} is established in the case $\chi_N = 0$, and accordingly $c = 0$. Here,  this result is extended to the asymmetric case $\chi_N>0$, and $c\neq 0$. We also refer to a more recent work \cite{mischler_linear_2016} that extends this confinement result to the higher dimensional case. 

\begin{theorem}[Exponential confinement by biased velocity-jump processes] Let $c\in (c_*,c^*)$. Assume $(\chi_S, \chi_N)\in (0,1/2)\times[0,1/2)$. Assume that  $\Ac+\Abc$ is satisfied.  There exists a density $f\in L_+^1\cap L^\infty (\R\times V)$ solution of \eqref{eq:linear c}. Moreover, there exist two exponents $\lambda_\pm>0$, and two velocity profiles $F_\pm(v)$ such that $e^{\lambda_+ z}f(z,v)$ (resp. $e^{-\lambda_- z}f(z,v)$)   converges exponentially fast to $F_+(v)$ as $z\to +\infty$ (resp. to $F_-(v)$ as $z\to -\infty$). 
\label{theo:cluster} 
\end{theorem}



The extremal values $c_*$ and $c^*$ are defined precisely during the course of the proof  \eqref{eq:def c_*}-\eqref{eq:def c^*}. The condition $c>c_*$ guarantees that $\lambda_+(c)>0$. On the other hand, $c<c^*$ if, and only if, $\lambda_-(c)>0$. For values of $c$ outside $(c_*,c^*)$, the solution does not decay exponentially, either on the left, or on the right side (see Figure~\ref{fig:flatness}).

\begin{figure}
\begin{center}
\includegraphics[width = 0.48\linewidth]{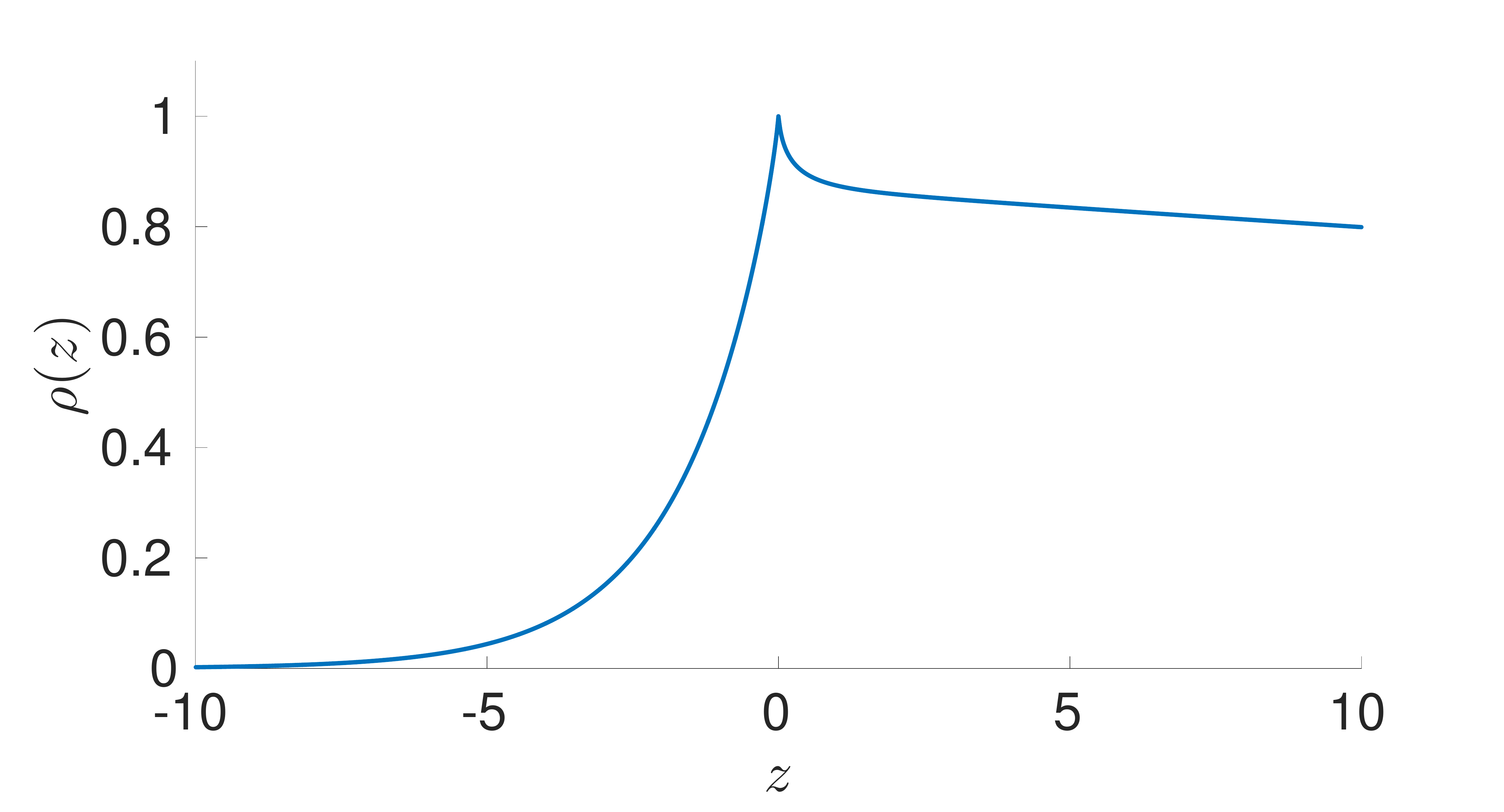}\;
\includegraphics[width = 0.48\linewidth]{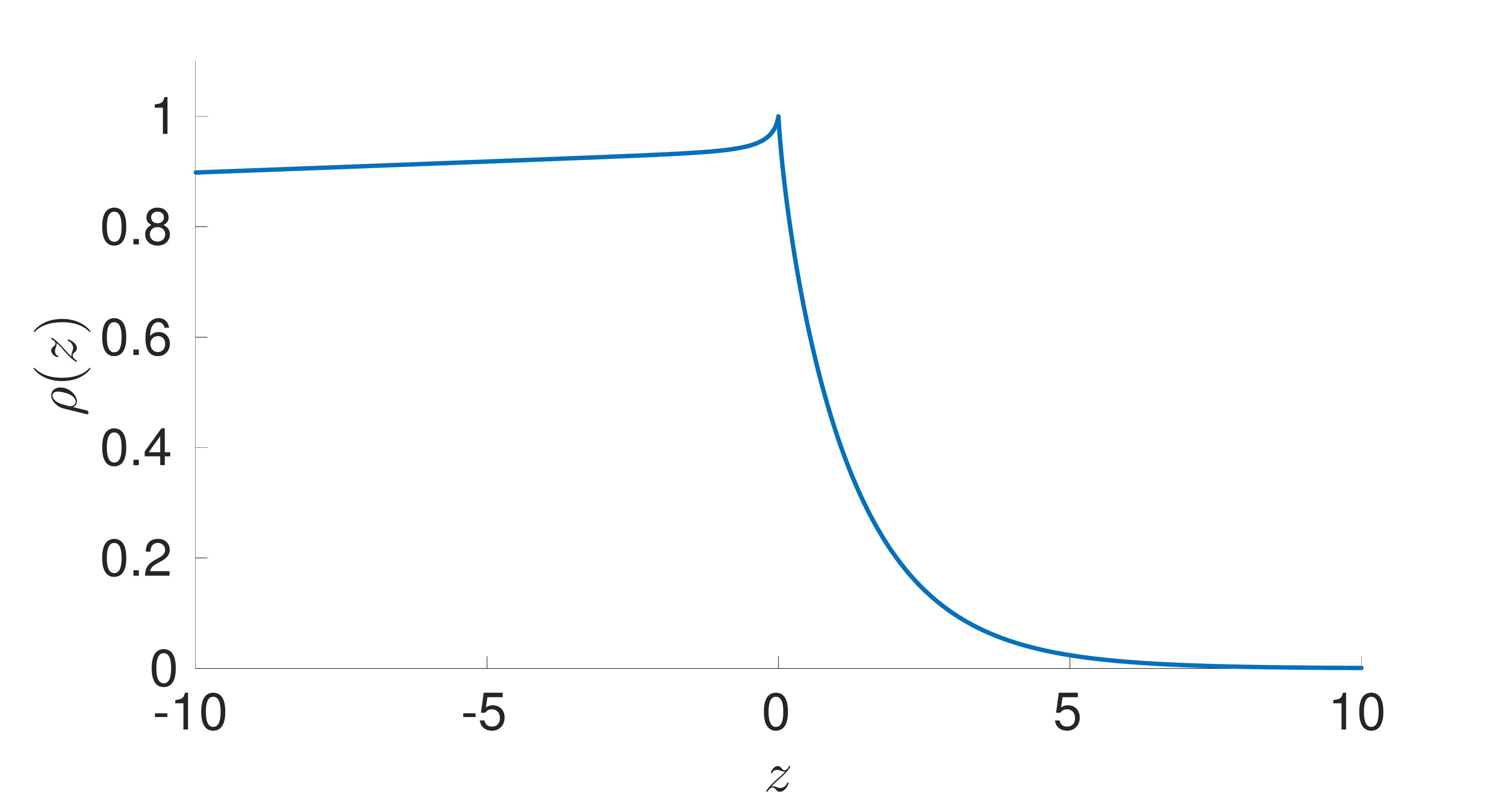}
\caption{\small Flatness. (Left) Macroscopic density profile as $c$ get close to $c_*$: the rate of exponential decay $\lambda_+(c)$ vanishes, so that the profile get flat on the right side. (Right) Idem as $c$ get close to $c^*$, so that $\lambda_-(c)$ vanishes.}
\label{fig:flatness}
\end{center}
\end{figure}

\begin{proof} 
As stated in the Theorem, we restrict here to the continuous setting, under Assumptions $\Ac+\Abc$. The  case of discrete velocities $\Ad$ can be solved directly by a finite dimensional approach based on Case's normal modes, see Section \ref{sec:spectral}. \medskip

\noindent\textbf{Step  \#1. Identification of the asymptotic profiles.}
We make the following ansatz for large (positive) $z$,
\[ f(z,v)\sim e^{- \lambda_+  z}  F_+ (v)\quad \text{as}\quad z\to +\infty\,,\]
where 
\beq \label{eq:F+} F_+(v) = \dfrac{1}{T_+(v-c) - \lambda_+ (v-c)}\, , \eeq
and $\lambda_+>0$ is characterized by the following  relation, \beq\label{eq:dispersion lambda+} \int \dfrac{T_+(v-c)}{T_+(v-c) - \lambda_+ (v-c)}\, d\nu(v)  = 1\quad \Leftrightarrow\quad
\int \dfrac{v-c}{T_+(v-c) - \lambda_+ (v-c)}\, d\nu(v)  = 0
\, .\eeq
The latter relation determines a unique positive $\lambda_+$, under the restriction that $F_+\geq 0$ on the support of $\nu$,
provided that the mean algebraic run length is negative in the moving frame, {\em i.e.}
\beq  \label{eq:condition c*}
\int \dfrac{v-c}{T_+(v-c)}\, d\nu(v) <0 \, .
\eeq
Indeed, the function 
\[ Q_+(\lambda) =   \int \dfrac{v-c}{T_+(v-c) - \lambda  (v-c)}\, d\nu(v) \, , \]
is continuous, and increasing. Moreover, as $\omega$ is bounded below on $V$, it diverges to $+\infty$ when $\lambda$ attains the maximal possible value ensuring $F_+\geq 0$, \ie $\lambda\to \frac{T_+(\vm -c)}{\vm - c} = \frac{T_+^+}{\vm - c}$ (recall that $\vm$ is the maximal velocity). Moreover, the value at $\lambda = 0$  coincides with the mean algebraic run length \eqref{eq:condition c*}. 

\begin{remark}[About the additional condition $\Abc$ on $\nu$]\label{rem:no boundedness}
If $\omega$ vanishes on the boundary of $V$, say at $\vm$, then it may happen that $Q_+$ converges to a finite negative value, as $\lambda$ converges to $\frac{T_+^+}{\vm - c}$. Then, the formal ansatz for the velocity profile \eqref{eq:F+} must be corrected with a Dirac mass, as follows,
\beq\label{eq:asymptot dirac mass} f(z,v)d\nu(v) \approx e^{-\lambda_+z} \left( \dfrac{\omega(v)}{T_+(v-c) - \frac{T_+^+}{\vm - c}(v-c) }dv + \varpi_0 \delta(v - \vm)  \right)\, ,  \eeq
where $\varpi_0>0$ is defined so as to compensate the negativeness of the integral, namely \eqref{eq:dispersion lambda+} rewrites as,
\[ \int   \dfrac{T_+(v-c)}{T_+(v-c) - \frac{T_+^+}{\vm - c}(v-c) }\omega(v)\, dv + \varpi_0 = 1\, . \]
We refer to \cite{caillerie_large_2016}, where a similar issue was identified and resolved, in the large deviations asymptotic of a basic  kinetic model, namely the BGK equation. 
\\
Noticeably, the asymptotic profile proposed in the formal ansatz \eqref{eq:asymptot dirac mass} is not a  bounded function. This is clearly an obstacle in using the approach developed below, which requires uniform boundedness of $F_+$. To overcome this issue, we shall derive uniform estimates for $f$ at further stage, independently of any bound on $F_+$ from above, see Section \ref{sec:regularity alpha0}.
\end{remark}

On the other hand,  condition \eqref{eq:condition c*} determines a minimal admissible value $c_*\in V$ for $c$. More precisely, $c_*$ is defined as the critical speed such that the mean algebraic run length vanishes,
\beq  \label{eq:def c_*} \int \dfrac{v-c_*}{T_+(v-c_*)}\, d\nu(v) = 0\, . \eeq
To see this, let introduce the function $R$, defined as follows,
\beq\label{jmbujvb}
R(c) = \int \dfrac{v-c}{T_+(v-c)}\, d\nu(v) \, .\eeq 
It is continuous and decreasing with respect to $c$. Indeed, in the case \Ac, the derivative writes
\[ \dfrac{dR}{dc}(c) = \int \dfrac{-1}{T_+(v-c)}\, d\nu(v) + 2 \chi_+ \int \dfrac{v-c}{T_+(v-c)^2} \delta(v-c) \, d\nu(v)\, . \]
The last contribution is written formally. However, it can be shown that it vanishes by approximation of the Dirac mass via mollifiers, provided that $\omega \in L^p$ for $p>1$(\footnote{Take $\eta_\eps = \eps^{-1}\eta(\eps^{-1}\cdot)$ an approximation of the Dirac mass, then \[\int (v-c) \eta_\eps(v-c) \omega(v)\, d\nu(v) \leq \|(v-c) \eta_\eps(v-c)\|_{p'} \|\omega\|_p  = \mathcal{O}\left (\eps^{1-1/p}\right )\,.\]}). 


By examining the limiting cases, we deduce that there is a unique $c_*\in V$ such that $R(c_*) = 0$. Notice that the sign of $c_*$ is the same as the sign of $R(0)$, which is the same as the sign of $\chi_N - \chi_S$ by symmetry of $\nu$: 
\beq\label{knùioùoib}
R(0) = \int_{\{v<0\}} \dfrac v{T_+^-} \, d\nu(v) +  \int_{\{v>0\}} \dfrac v{T_+^+} \, d\nu(v) =   \int_{\{v>0\}} v\left(  \dfrac {-\chi_S + \chi_N}{1  - (\chi_S - \chi_N)^2} \right)\, d\nu(v)
\,. \eeq

We make a similar ansatz for large (negative) $z$,
\[f(z,v) \sim e^{\lambda_-z}F_-(v)\quad \text{as}\quad z\to -\infty\,,\] 
where 
\[ F_-(v) = \dfrac{1}{T_-(v) + \lambda_- (v-c)}\, , \]
The rate of exponential decay $\lambda_->0$ is uniquely determined by the following  relation, 
\[ \int \dfrac{T_-(v-c)}{T_-(v) + \lambda_- (v-c)}\, d\nu(v)  = 1 \quad \Leftrightarrow\quad  \int \dfrac{v-c}{T_-(v) + \lambda_- (v-c)}\, d\nu(v)  = 0 \, , \]
provided that the mean algebraic run length is positive in the moving frame, {\em i.e.}
\beq \label{eq:condition c* 2} 
\int \dfrac{v-c}{T_-(v-c)}\, d\nu(v) >0 \, .
\eeq
Again, this determines a maximal admissible value $c^*\in (0,\v0)$, because the mean algebraic run length is positive for $c = 0$, and the dependency with respect to $c$ is Lipschitz continuous and decreasing. More precisely, it is defined as  
\beq \label{eq:def c^*} 
\int \dfrac{v-c^*}{T_-(v-c^*)}\, d\nu(v) =0 \, .
\eeq


Notice that $F_\pm, \lambda_\pm$ both depend on $c$.  For the sake of clarity, we may omit this dependency in the notation. From now on, we assume that the speed $c$ belongs to the  interval $c\in (c_*,c^*)$, and we keep in mind the extremal behaviours,
\[\lim_{c\to c_*} \lambda_+(c) = 0 \, , \quad \lim_{c\to c^*} \lambda_-(c) = 0\, . \]
Alternatively speaking, the spatial density becomes  flat on the far right side as the speed $c$ decreases to $c_*$, whereas it becomes spatially flat on the far left side as it increases to $c^*$  (see Figure~\ref{fig:flatness}). This property will play a crucial role in the last step of the proof, when determining the  speed $c$ for the coupled problem.\medskip

\begin{figure}
\begin{center}
\includegraphics[width = 1\linewidth]{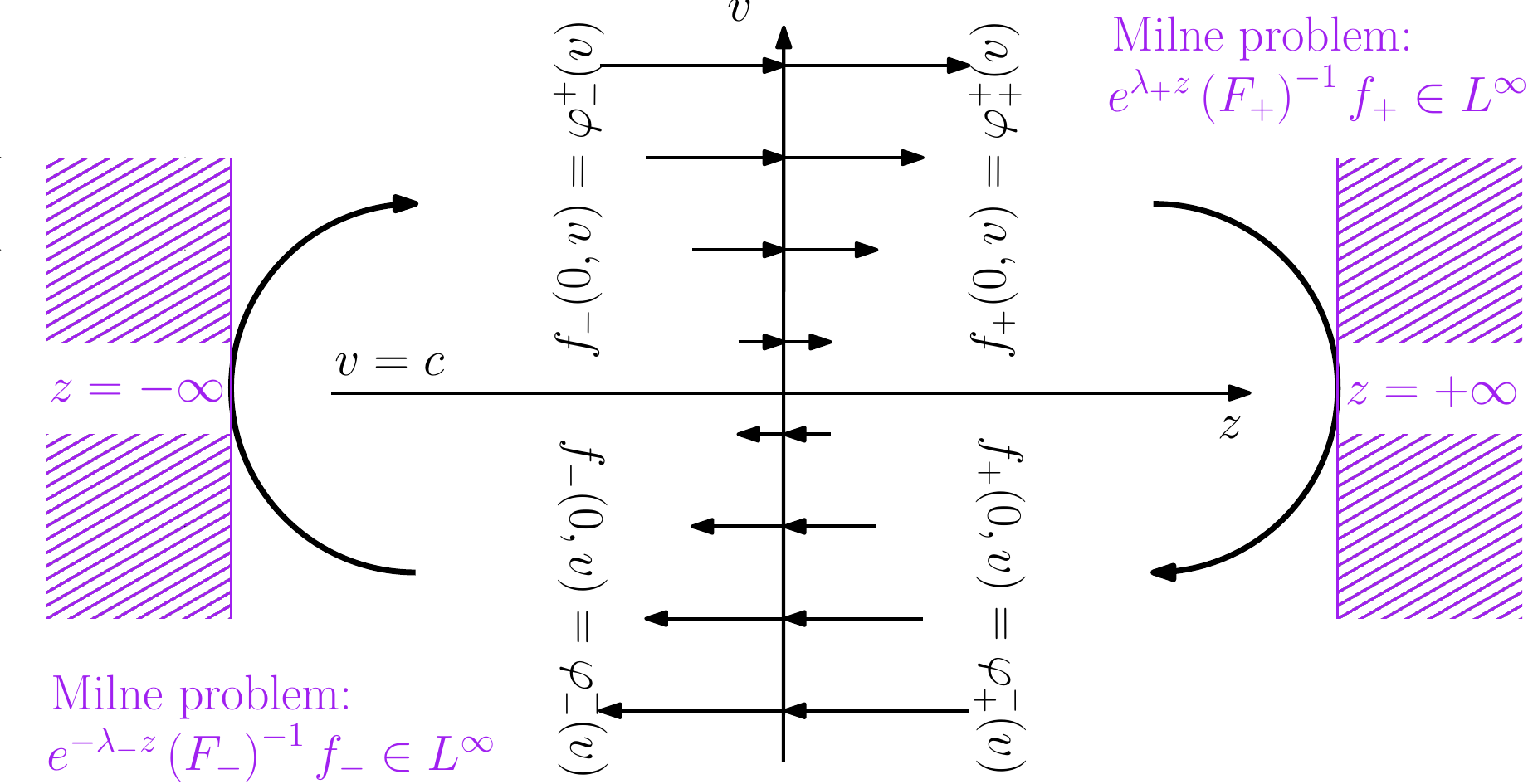}
\caption{\small Cartoon of the map $\mathbf B: \varphi_+^+ \mapsto \varphi_-^+$ after two successive applications of the Milne problem at $z=\pm \infty$.}
\label{fig:Milne twice}
\end{center}
\end{figure}

\noindent\textbf{Step  \#2. Connection with the  Milne problem.}
One key observation is that the function $u_+(z,v) = e^{\lambda_+ z}(F_+(v))^{-1}f(z,v)$ verifies a kinetic equation on the half space $\R_+\times V$, which satisfies the maximum principle, 
\beq\label{eq:u(z,v)}
(v-c)\partial_z u_+(z,v) = \int \dfrac{T_+(v'-c)F_+(v')}{F_+(v)}(u_+(z,v')-u_+(z,v))\, d\nu(v')\, , \quad z>0\, , \; v\in V\, .
\eeq
However, we have no idea yet about the incoming data, {\em i.e.} the inward velocity profile at $z = 0$: $(u_+^+(0,v))_{\{v>c\}}$. For the "boundary value" at $z = +\infty$, we only impose boundedness of $u_+$. This prescribes uniquely $u_+$ as a solution  of the so-called Milne problem in radiative transfer theory \cite{bardos_diffusion_1984}: being given the inward  velocity profile $\varphi_+^+(v) = F_+(v) u_+^+(0,v)$ for $v>c$, there exists  a unique $u_+\in L^\infty(\R_+\times V)$, such that 
\[
\begin{cases}
 \displaystyle (v-c)\partial_z u_+(z,v) = \int \dfrac{T_+(v'-c)F_+(v')}{F_+(v)}(u_+(z,v')-u_+(z,v))\, d\nu(v')\, , \quad z>0\, , \; v\in V\medskip\\
 u_+^+(0,v) = (F_+(v))^{-1}\varphi_+^+(v)\, , \quad v>c\, .
\end{cases}
\] 
Moreover, we deduce the following bound from the maximum principle (see \cite{calvez_confinement_2015} for details), 
\[ \|u_+\|_{L^\infty(\R_+\times V)} \leq \|u_+^+(0,\cdot)\|_{L^\infty(\{v>c\})}\, . \]

We apply the same construction for negative $z$. We define accordingly  $L^\infty \ni u_-(z,v) = e^{-\lambda_- z}(F_-(v))^{-1}f(z,v)$, which solves a similar problem, 
\[
\begin{cases}
 \displaystyle (v-c)\partial_z u_-(z,v) = \int \dfrac{T_-(v'-c)F_-(v')}{F_-(v)}(u_-(z,v')-u_-(z,v))\, d\nu(v')\, , \quad z<0\, , \; v\in V\medskip\\
 u_-^-(0,v) = (F_-(v))^{-1}\varphi_-^-(v)\, , \quad v<c\,  ,
\end{cases}
\] 
together with the matching inward velocity profile, \ie 
\[(\forall v<c)\quad\varphi_-^-(v) = \varphi_+^-(v) =  F_+(v)u_+^-(0,v)\, .\]
We refer to Figure \ref{fig:Milne twice} for a sketch of the construction of both $u_+$ and $u_-$ starting at $z = 0$ with boundedness reflection at $\infty$ on both sides.  
\medskip

\noindent\textbf{Step \#3. A fixed point argument.}
We define the linear map $\mathbf B: {\mathcal C}^0(V\cap \{v>c\}) \to {\mathcal C}^0(V\cap \{v>c\})$ as follows, 
\[ (\forall v>c )\quad \left(\mathbf B \varphi_+^+\right)(v) =  F_-(v)u_-^+(0,v)\, , \]
so as being a fixed point is equivalent to closing the loop $\varphi^+_+ \mapsto \varphi^-_+ = \varphi_-^- \mapsto \varphi^+_- = \varphi_+^+$, see Figure \ref{fig:Milne twice}. 

Compactness is required in order to apply the Krein-Rutman theorem. It can be deduced from classical averaging lemma, since the outward velocity profile is given by
\beq u_+^-(0,v) =  \int_0^{+\infty} e^{- t} J_+( -  tF_+(v) (v-c))\, dt\, ,\quad v<c \,, \label{eq:duhamel u}\eeq
where the macroscopic quantity $J_+$ is defined as,
\[J_+(z) = \int  T_+(v-c)F_+(v) u_+(z,v) \, d\nu(v) = e^{\lambda_+ z}I_+(z)\, , \quad z>0\, .\]
One-dimensional averaging lemma \cite{bardos_diffusion_1984} guarantees that $A_+(z)$ has  H\"older regularity (see also the proof of Proposition \ref{prop:holder}\#2). 


As a consequence, $(u_+(0,v))_{\{v<c\}}$ has H\"older regularity too. Alternatively speaking, negative velocities at $z=0$ necessarily emerge after some tumbling event, for which regularization occurs via velocity averaging.  Similarly, $u_-^+(0,v)$ has H\"older regularity on $\{v>c\}$. Hence $\mathbf B$ is a compact map.  On the other hand, positivity is an immediate consequence of the Duhamel formula \eqref{eq:duhamel u}. 

Applying the Krein-Rutman theorem, there exists a velocity profile $\varphi_+^+$, associated with a dominant eigenvalue $\Lambda$ such that $\mathbf B \varphi^+_+ = \Lambda \varphi^+_+ $. We deduce that  $\Lambda = 1$ from the conservation laws 
\beq\label{eq:fluxes} \begin{cases}\displaystyle\partial_z\left(\int (v-c) f_+(z,v)\, d\nu(v)\right) = 0\quad \Rightarrow \quad \int (v-c) f_+(z,v)\, d\nu(v)  = 0 \medskip\\\displaystyle\partial_z\left(\int (v-c) f_-(z,v)\, d\nu(v)\right) = 0\quad \Rightarrow \quad \int (v-c) f_-(z,v)\, d\nu(v)  = 0\, . \end{cases}  \eeq
Indeed, we have 
\begin{align*} 
\int_{\{v<c\}} (v-c) \varphi_+^+(v)\, d\nu(v) 
& = - \int_{\{v<c\}}(v-c) \varphi_+^-(v)(v)\, d\nu(v) 
 = - \int_{\{v<c\}}(v-c) \varphi_-^-(v)(v)\, d\nu(v)\\ 
& =  \int_{\{v<c\}} (v-c) \varphi_-^+(v)\, d\nu(v) = \int_{\{v<c\}} (v-c) \left(\mathbf B \varphi_+^+\right)(v)\, d\nu(v)  \\ 
& = \Lambda \int_{\{v<c\}} (v-c) \varphi_+^+(v)\, d\nu(v)\, . \end{align*}
\medskip

\noindent\textbf{Step \#4. Refined asymptotic behaviour as $|z|\to \infty$.}
The zero rate of decay $\lambda = 0$ is a trivial solution of the l.h.s. of \eqref{eq:dispersion lambda+}. It corresponds to a flat mode which is excluded to ensure integrability of $f$. It comes with the conservation of flux \eqref{eq:fluxes}. Besides, the non trivial, positive root $\lambda_+$ comes with another conservation law. By multiplying \eqref{eq:u(z,v)} against $(v-c)F_+(v)^2$, we obtain
\begin{align*}
&\partial_z \left( \int (v-c)^2 F_+(v)^2 u_+(z,v)\, d\nu(v)\right)\\
& = \left( \int T_+(v'-c) F_+(v') u_+(z,v')\, d\nu(v')\right) \left( \int v F_+(v)\, d\nu(v)\right) \\
& \quad - \left( \int T_+(v'-c) F_+(v')  \, d\nu(v')\right) \left( \int (v-c) F_+(v) u_+(z,v)\, d\nu(v)\right) \\
& = 0\, ,
\end{align*}
since $\int (v-c) F_+(v)\, d\nu(v)= 0$, and $\int (v-c) F_+(v) u_+(z,v)\, d\nu(v) = e^{\lambda_+ z}\int (v-c) f_+(z,v)\, d\nu(v) = 0 $.
We define accordingly $\kappa_+>0$ such that 
\beq\label{eq:def mu} (\forall z>0)\quad \int (v-c)^2 F_+(v)^2 (u_+(z,v) - \kappa_+)\, d\nu(v) = 0\, . \eeq
Then, it is possible to prove exponential decay towards the asymptotic profile as $z\to+\infty$ using second order entropy method with respect to the space variable. 
In fact, the two auxiliary quantities 
\[ E_0(z) = \int  (v-c)^2 F_+(v) ^2 \left( u_+(z,v) - \kappa_+  \right)^2  \, d\nu(v)\, , \]
and 
\[
E_1(z) =\int (v-c)  T_+(v) F_+(v)^2 \left( u_+(z,v) - \kappa_+  \right)^2  \, d\nu(v) \, ,   \]
satisfy the following system of (in)equalities,
\begin{equation}\label{eq:damped system}
\begin{cases}
\dfrac {d E_0}{dz} (z) = 2 \lambda_+ E_0(z) - 2 E_1(z)\medskip\\
\dfrac {d E_1}{dz} (z) \leq - 2  \kappa E_0(z)\,,  \end{cases} 
\end{equation}
where 
\[
\kappa_+ = \left(\inf_{v\in V} \frac{T_+(v-c)}{(v-c)^2 F_+(v)} \right)^2 \left(\int  (v-c)^2 (F_+(v))^2 \,d\nu(v)\right)  \, .
\]
We refer to \cite{calvez_confinement_2015} for further details.
The system \eqref{eq:damped system} can be turned into the following second-order damped differential inequality,
\[ -\frac12 \frac{d^2E_0}{dz^2} (z) + \lambda_+\dfrac {d E_0}{dz} (z) + 2\kappa_+ E_0(z) \leq 0\, .  \]
Using the "boundary condition" $E_0\in L^\infty$, we  deduce that there exists a constant $C$ depending on $E_0(0), \frac {d E_0}{dz} (0)$, $\lambda_+$ and $\kappa_+$, such that,
\[ E_0(z) \leq C e^{-\vartheta_+ z}\,, \quad \text{where $\vartheta_+$ is the positive root of $\frac12 \vartheta_+^2 + \lambda_+ \vartheta_+ - 2\kappa_+ = 0$}\, .  \]
As a conclusion, we have obtained the following weighted $L^2$ estimate,
\beq \label{eq:spectral gap 1} (\forall z>0)\quad \left\| \dfrac{f_+(z,v)}{e^{-\lambda_+ z}F_+(v)} - \kappa_+ \right\|_{L^2((v-c)^2 F_+(v)^2d\nu(v))} \leq C e^{-\vartheta_+ z}\, . \eeq
Intuitively, $\vartheta_+>0$ is an estimate of the spectral gap in the spatial decay. Loosely speaking, we have,
\begin{equation*}
f_+(z,v) = \kappa_+ e^{-\lambda_+ z}F_+(v) + \mathcal O\left ( e^{-(\lambda_++ \vartheta_+) z}  \right ) \, .
\end{equation*}
A similar estimate holds true for $z<0$. 
\end{proof}

\subsection{Regularity (uniform with respect to $\omega_0$)}
\label{sec:regularity alpha0}
We investigate the regularity of the macroscopic quantity 
\[ I(z) = \int  T(z,v-c) f(z,v)\, d\nu(v)\, , \]
from which the solution can be reconstructed entirely by the Duhamel formulation along characteristics lines \eqref{eq:duhamel}.

In brief, we obtain that all quantities are smooth expect possibly at $z=0$, and at $v=c$. More precisely, assuming that $c\in \mathrm{Int}\, V$, the following properties hold true: 
\begin{itemize}
\item $f$ is continuous with respect to $z\in \R$, for all $v\in V$,
\item $f$ is discontinuous at $v=c$. This is a consequence of the discontinuity of $T$. Indeed, we can deduce from  \eqref{eq:f I} that $T_+^+ f_+^+(z,c^+) = I_+(z)$ for $z>0$, and also that $T_+^- f_+^-(z,c^-) = I_+(z)$  (see Figure \ref{fig:overshoot}). Thus, $f_+$ is discontinuous at $v=c$ for $z>0$, because $T_+^-\neq T_+^+ $, except in the special case $\chi_S = \chi_N$. It is always discontinuous at $v=c$ for $z<0$.  
\item the spatial density  $\rho$ is smooth on both sides of the origin. Noticeably, $\frac{d\rho}{dz}$ may have a logarithmic singularity at the origin. This singularity plays some important role in the analysis. 
\end{itemize}

\begin{proposition}\label{prop:holder}
Assume that $f$ is normalized to have unit mass \eqref{eq:f unit mass}. Then $f$ is uniformly bounded, and exponentially decreasing as $|z|\to \infty$. Moreover, these two estimates are uniform with respect to $\omega_0$, the lower bound on $\omega$. 
\end{proposition}

\begin{proof}
We develop a bootstrap argument, focusing on the regularity of $I$. The first immediate observation is that $I\in L^1$, more precisely $\|I\|_{L^1}\leq (\max T)\|f\|_{L^1}$. 
\medskip

\noindent\textbf{Step  \#1. $I$ is uniformly bounded.}
We start from the self-convolution structure of $I$. Indeed, we deduce from \eqref{eq:duhamel} that we have for $z>0$,
\beq\label{eq:self convol 1} \rho_+^-(z) = \int_{-\infty}^0 I_+(z-y) \left(  \int_{\{v<c\}} \exp\left( T_+^-\dfrac y{|v-c|}  \right)\dfrac{\omega(v)}{|v-c|}\, dv\right) \, dy\, .  \eeq
On the other hand, we have
\begin{multline}\label{eq:self convol 2} 
\rho_+^+(z) = 
  \int_{\{v>c\}} \left( \int_{-\infty}^0  I_-(y)   \exp\left( T_-^+ \dfrac y{v-c} \right)\, dy \right) \exp\left(-T_+^+\frac z{v-c}\right) \dfrac{\omega(v)}{v-c}\, dv  
\\ + \int_0^{ z }  I_+(z - y) \left(\int_{\{v<c\}} \exp\left(-T_+^+ \dfrac y{v-c}\right)\dfrac{\omega(v)}{v-c}\, dv \right)\, dy\, .
\end{multline}

We define accordingly the two auxiliary functions 
\[ \begin{cases}
(\forall y<0)\quad \displaystyle G^-(y) =   \int_{\{v<c\}} \exp\left( T_+^-\dfrac y{|v-c|}  \right)\dfrac{\omega(v)}{|v-c|}\, dv \medskip\\
(\forall y>0)\quad \displaystyle G^+(y) =   \int_{\{v>c\}} \exp\left( - T_+^+\dfrac y{v-c}  \right)\dfrac{\omega(v)}{v-c}\, dv
\end{cases}\]

\begin{lemma}\label{lem:Gpm}
Let $p\in (1,\infty)$ such that $\omega \in L^p$ $\Ac$. Then there exist positive constants $C$ and $\varsigma$ such that 
\[ \begin{cases}
(\forall y\in (0,1))\quad G^\pm(|y|) \leq C |y|^{-1/p} \medskip\\
(\forall y\in (1,+\infty)) \quad G^\pm(|y|) \leq C \exp(-\varsigma |y|)\, . 
\end{cases} \] 
\end{lemma}

\begin{remark}
In the case $p = \infty$, the first item should be replaced with a logarithmic growth, \ie  $G^\pm(y) \leq C \log(1/|y|)$.
\end{remark}

\begin{proof}
We deduce from the H\"older inequality that,
\begin{align*}
G^-(|y|) &\leq \left( \int_{-\vm}^c \exp\left( - p' T_+^-\dfrac {|y|}{|v-c|}  \right) \dfrac{1}{|v-c|^{p'}}\, dv \right)^{1/p'} \|\omega\|_p \\
& \leq  |y|^{1/p'-1} \left( \int_{y/(c+\v0)}^{+\infty} \exp\left( - p' T_+^- s  \right) s^{p'-2} \, ds \right)^{1/p'} \|\omega\|_p\,.
\end{align*}
In the case $p<\infty$, \ie $p'-2>-1$, we clearly have $G^-(|y|)\leq C |y|^{-1/p}\|\omega\|_p$ everywhere since the first integral is convergent as $y\to 0$. We get immediately exponential decay for some absolute $\varsigma>0$ for large $y$, say $\varsigma = (1/2)T_+^-/(c+\v0)$. In the case $p = \infty$, then the integral is not convergent as $y \to 0$, but is equivalent to $C \log(1/|y|)$ for some constant $C$. 

Similar estimates hold true for $G^+$. 
\end{proof}

Lemma \ref{lem:Gpm} is useful because it enables to reformulate \eqref{eq:self convol 1}--\eqref{eq:self convol 2} as follows
\[ \rho_+^-(z) = \int_{-\infty}^0 I_+(z-y) G^-(y)\, dy\, , \]
with $G^- \in L^q$ for $1\leq q < p$. 
Similarly, we have 
\begin{multline*}
\rho_+^+(z) = 
  \int_{\{v>c\}} \left( \int_{-\infty}^0  I_-(y)   \exp\left( T_-^+ \dfrac y{v-c} \right)\, dy \right) \exp\left(-T_+^+\frac z{v-c}\right) \dfrac{\omega(v)}{v-c}\,dv
 \\ + \int_0^{ z }  I_+(z - y) G^+(y)\, dy\, .
\end{multline*}
Fix $1< q < p$ to be chosen later. The first iteration of the bootstrap argument yields \[\|\rho_+^-\|_{L^q} \leq \|I_+\|_{L^1} \|G^-\|_{L^q} \,,\]
and also
\[ \|\rho_+^+\|_{L^q} \leq \|I_-\|_{L^1} \|G^+\|_{L^q} + \|I_+\|_{L^1} \|G^+\|_{L^q} \, , \]
where we have simply used $\exp\left( T_-^+ \frac y{v-c} \right)\leq 1$ in the first contribution. 

Similar estimates hold true on the negative side, for $\rho_-^\pm$. Thus, we can prove that $I_\pm = T_\pm^- \rho_\pm^- + T_\pm^+ \rho_\pm^+$ belongs to $L^q$. 

Then, we proceed recursively, in order to gain integrability in a finite number of iterations.  We get immediately the following recurrence formula,
\beq\label{eq:recursive exponent}\|\rho_+^-\|_{L^{r_{n+1}}} \leq \|I_+\|_{L^{r_n}} \|G^-\|_{L^q} \,, \quad \frac1{r_{n+1}} = \frac1{r_{n}} - \frac1{q'}\, . \eeq
The analogous one for $\rho_+^+$ requires one  more argument because the transfer term from negative to positive side is not of convolution type. Indeed, we have
\begin{align*} \int_{-\infty}^0  I_-(y)   \exp\left( T_-^+ \dfrac y{v-c} \right)\, dy 
&  \leq \|I_-\|_{L^{r}} \left( \int_{-\infty}^0  \exp\left( r' T_-^+ \dfrac y{v-c} \right)\, dy \right)^{1/r'} \\
& \leq  C \|I_-\|_{L^{r}}  (v-c)^{1/r'}\, .
\end{align*}
Now, back to the proof of Lemma \ref{lem:Gpm}, we realize that 
\begin{multline}\label{eq:568}  \int_{\{v>c\}} (v-c)^{1/{r}'}   \exp\left(-T_+^+\frac z{v-c}\right) \dfrac{\omega(v)}{v-c}\, dv \\ \leq   z^{1/r' -1 + 1/p'} \left( \int_{z/(\v0-c)}^{+\infty} \exp\left( - p' T_+^+ s  \right) s^{(-p'/r' + p'-1)-1} \, ds \right)^{1/p'} \|\omega\|_p \, , 
\end{multline}
where we impose the constraint $1/r' - 1 + 1/p'<0$, to ensure $-p'/r' + p'-1 >0$. 
We deduce, 
\begin{align}
\|\rho_+^+\|_{L^{r_{n+1}}} & \leq  \|I_-\|_{L^{r_n}} \left\| \int_{\{v>c\}} (v-c)^{1/{r_n}'}   \exp\left(-T_+^+\frac z{v-c}\right) \dfrac{\omega(v)}{v-c}\, dv  \right\|_{L^{r_{n+1}}} + \|I_+\|_{L^{r_n}} \|G^+\|_{L^{q}} \label{eq:457} \\
& \leq C \|I_-\|_{L^{r_n}} + \|I_+\|_{L^{r_n}} \|G^+\|_{L^{q}} \, , \nonumber 
\end{align}
provided that the following condition is fulfilled,
\[ r_{n+1}\left( \dfrac1{r_n'} -1 + \frac1{p'} \right) >-1 \quad \Leftrightarrow \quad \frac1{r_{n+1}} - \frac1{r_{n}} + \frac1{p'} >0 \, , \]
to get integrability of the first contribution in  the r.h.s. of \eqref{eq:457}. This last inequality holds true when $(r_n)$ satisfies the recurrence relation \eqref{eq:recursive exponent}, because $q<p$. 

Now, choose $1<q<p$ such that $q' = N$ is the smallest integer larger than $p'$. Then the recurrence relation on $(r_n)$ \eqref{eq:recursive exponent}, with the initial condition $r_0 = 1$,   guarantees that $1/r_N  = 0$, \ie $r_N = \infty$. In fact, the last step of the iteration should be done with caution. At the $N-1$ step, we have obtained 
$ I_{\pm} \in L^{N} $. Young's inequality still holds, namely $\|\rho_+^-\|_\infty \leq \|I_+\|_{L^{N}} \|G^-\|_{L^{N/(N-1)}}$ (it is simply H\"older's inequality here). However, the estimate \eqref{eq:568} becomes
\begin{equation*}
\int_{\{v>c\}} (v-c)^{1-1/N}   \exp\left(-T_+^+\frac z{v-c}\right) \dfrac{\omega(v)}{v-c}\, dv \leq \left( \int_{\{v>c\}} (v-c)^{-p'/N} \, dv \right)^{1/p'} \|\omega\|_p \, .
\end{equation*}
The last contribution is bounded since $N>p'$. 
This procedure enables to bootstrap from $L^1$ to $L^\infty$ in a finite number of iterations, provided $p>1$.

As a corollary, we get that $f$ is uniformly bounded too. This is a direct consequence of the Duhamel formulation \eqref{eq:duhamel},
\[
\begin{cases}
\|f_+^-\|_\infty \leq \dfrac{\|I_+\|_\infty}{T_+^-}\,, \medskip\\
\|f_+^+\|_\infty \leq \dfrac{\|I_-\|_\infty}{T_-^+} + \dfrac{\|I_+\|_\infty}{T_+^+}\, .
\end{cases}
\]
\medskip

\noindent\textbf{Step  \#2. $I$ is H\"older continuous.}
We apply one-dimensional averaging lemma \cite{golse_regularity_1988}. For $z,z'>0$, we have
\begin{align*} 
I_+(z) - I_+(z') &= \left( \int_{|v-c|<\delta/2} + \int_{|v-c|>\delta/2} \right ) T_+(v-c) \left(f_+(z,v) - f_+(z',v)\right)  \omega(v)\, dv\\
& \leq  2 (\max T) \|f_+\|_\infty \delta^{1/p'}  \|\omega\|_p +  (\max T) \left(\int_{|v-c|>\delta/2} \dfrac{\omega(v)}{|v-c|}\, dv\right) \|(v-c) \partial_z f_+\|_\infty |z-z'| \\
& \leq 2 (\max T) \|f_+\|_\infty \delta^{1/p'}  \|\omega\|_p  +  2(\max T)^2 \|f_+\|_\infty C \delta^{1/p'-1} \|\omega\|_p |z-z'|\, . 
\end{align*}
By choosing $\delta = C (\max T) |z - z'|$, we obtain
\beq\label{eq:I_+ holder continuous} I_+(z) - I_+(z')  \leq C \|f_+\|_\infty \|\omega\|_p |z-z'|^{1/p'}\, . \eeq
Note that in the case $p=\infty$ ($p' = 1$), we cannot get any Lipschitz estimate, but there is a logarithmic correction. However, this is not important for the remaining of this work.
\medskip

\noindent\textbf{Step  \#3. $I$ is exponentially decaying.}
Recall that the maximum principle for the normalized solution $u_+(z,v) = e^{\lambda_+ z}(F_+(v))^{-1}f_+(z,v)$ yields the following estimate,
\beq \|u_+\|_{L^\infty(\R_+\times V)} \leq  \|(F_+(v))^{-1} f_+(0,v)\|_{L^\infty(\{v>c\})}\, . \label{eq:5856} \eeq
The key point is that, although $F_+(v)$ may not be uniformly bounded with respect to $\omega_0$ (see Remark \ref{rem:no boundedness}), it happens surely for $(F_+(v))^{-1} = T_+(v-c) - \lambda_+ (v-c)$. Therefore, $u_+$ is uniformly bounded. Of course, this does not control all velocities uniformly with respect to $\omega_0$, but this is yet an important information. 

Next, we use the additional conservation law \eqref{eq:def mu}. It reads 
\begin{align}   
\int (v-c)^2 F_+(v) f_+(z,v)\, d\nu(v) &  = e ^{-\lambda_+z}  \int (v-c)^2 F_+(v) f_+(0,v)\, d\nu(v)\label{eq:567967} \\
& \leq C e ^{-\lambda_+z} \|f_+\|_\infty \int F_+(v) d\nu(v)\, .\nonumber
\end{align} 
The last quantity is uniformly bounded, because $\lambda_+$ is defined such as $\int T_+ F_+ \, d\nu  = 1$ \eqref{eq:dispersion lambda+}. In order to control the exponential decay of $I_+$ from \eqref{eq:567967}, it remains to control small relative velocities, \ie $|v-c|\ll1$. But this is guaranteed by the pointwise estimate \eqref{eq:5856} which reads as follows,
\beq\label{eq:589808} f_+(z,v) \leq  \|u_+\|_{L^\infty(\R_+\times V)} F_+(v) e^{-\lambda_+z} \, . \eeq
The function $F_+$ possibly diverges only at $v = \v0$ in the limit  $\omega_0 \to  0$, see \eqref{eq:asymptot dirac mass}. 
On the other hand, the obvious estimate $c^*<\v0$ ensures that there exists $\delta>0$ such that $F_+$ is uniformly bounded on $(c-\delta,c+\delta)$. 

To conclude, we combine \eqref{eq:567967} and \eqref{eq:589808} as follows,
\begin{align*} 
I_+(z) & = \left (\int_{|v-c|<\delta} + \int_{|v-c|>\delta}\right ) T_+(z,v-c)  f_+(z,v) \, d\nu(v)  \\
& \leq C (\max T) \left(\sup_{v\in (c-\delta,c+\delta)} F_+(v)\right) e^{-\lambda_+z} + C \delta^{-2} \left(\max T \right) \left( \sup F_+(v)^{-1}\right)   e^{-\lambda_+z}\\
& \leq C e^{-\lambda_+z}\, .
\end{align*}
A similar estimate holds true for $I_-$.  
\end{proof}

\subsection{Further regularity}\label{sec:Further regularity}
We can obtain better regularity for $z>0$ (resp. $z<0$) by bootstrap. This will be needed in Section \ref{sec:monotonicity}, as we shall compute derivatives of $\rho^{\pm}$ with respect to $z$.

First, we establish that $f$ is H\"older continuous with respect to both $z,v$. We use the formulation of $f$ along characteristic lines \eqref{eq:duhamel}.
Let $0<z_1<z_2$. We have successively
\begin{equation*}
| f_+^-(z_1,v) - f_+^-(z_2,v) |  \leq \int_0^{+\infty} [I_+]_{1/p'}  |z_1-z_2|^{1/p'}  \exp(-T_+^- s)\, ds = \dfrac{[I_+]_{1/p'}}{T_+^-} |z_1-z_2|^{1/p'} \,,
\end{equation*}
and
\begin{multline*}
| f_+^+(z_1,v) - f_+^+(z_2,v) | \leq \dfrac{\|I_-\|_\infty}{T_-^+} 
\left[\exp\left(-T_+^+\frac \bullet{v-c}\right)\right]_{\mathcal C^{0,1/p'}(z_1,z_2)}|z_1 - z_2|^{1/p'}
\\
+  \dfrac{\|I_+\|_\infty}{T_+^+} \left| \exp\left(-T_+^+\frac{ z_1}{v-c}\right) - \exp\left(-T_+^+\frac {z_2}{v-c}\right) \right|   + \dfrac{[I_+]_{1/p'}}{T_+^+}|z_1 - z_2|^{1/p'}  \,.
\end{multline*}
The H\"older regularity of the exponential function with respect to $z$ is evaluated as follows,
\begin{align*}
\left| \exp\left(-T_+^+\frac{ z_1}{v-c}\right) - \exp\left(-T_+^+\frac {z_2}{v-c}\right) \right|
& \leq \left( \int_{z_1}^{z_2} \left(\dfrac{T_+^+}{v-c}\right )^p \exp\left ( - p \frac{y}{v-c}\right )\, dy \right )^{1/p} |z_1 - z_2|^{1/p'}\\
& \leq C \left( \int_{z_1}^{z_2} y^{-p}\, dy \right )^{1/p} |z_1 - z_2|^{1/p'}
\\
&\leq C z_1^{-1/p'}|z_1 - z_2|^{1/p'}\, .
\end{align*}
We deduce the following H\"older regularity for $f_+^+$,
\begin{equation*}
| f_+^+(z_1,v) - f_+^+(z_2,v) | \leq C \left( z_1^{-1/p'} \vee 1  \right) |z_1 - z_2|^{1/p'}\,.
\end{equation*}
We notice that regularity of $f$ degenerates as $z\to 0$. We will see a more quantitative feature of this degeneracy in Lemma \ref{lem:zero monotonicity} below. Now, we are in position to improve slightly the regularity of $I_+$, locally uniformly for $z\in (0,+\infty)$. We perform the same computation as \eqref{eq:I_+ holder continuous}:
\begin{align} 
\left |I_+(z_1) - I_+(z_2) \right |&\leq\left ( \int_{|v-c|<\delta} + \int_{|v-c|>\delta} \right ) T_+(v-c) \left|f_+(z_1,v) - f_+(z_2,v)\right|  \omega(v)\,dv\label{jbhlmiugbi}\\
& \leq   C \left( z_1^{-1/p'}\vee 1  \right)  \delta^{1/p'}  \|\omega\|_p |z_1 - z_2|^{1/p'}  +  C \delta^{1/p'-1} \|\omega\|_p |z_1-z_2|\nonumber\\
& \leq C  \left( z_1^{-1/(pp')}\vee 1  \right)|z_1-z_2|^{1-1/p^2}\nonumber\, .
\end{align}
%
%
Therefore, we have gained regularity, since $1-1/p^2> 1-1/p$. However, this  is clearly not uniform up to $z = 0$ due to the pre-factor. 

It is possible to iterate the whole argument, in order to prove that $I_+$ belongs to the family of H\"older spaces $\mathcal C^{0,\theta_n}$, with  $\theta_n = 1-1/p^n$, for all $n\in \N^*$. 
Indeed, at the next step, the same lines of calculation yields
\begin{align*}
| f_+^-(z_1,v) - f_+^-(z_2,v) | & \leq    \int_0^{+\infty} C  \left( z_1^{-1/(pp')}\vee 1  \right)  |z_1-z_2|^{1-1/p^2}  \exp(-T_+^- s)\, ds \\
& \leq C  \left( z_1^{-1/(pp')}\vee 1  \right) \dfrac{1}{T_+^-} |z_1-z_2|^{1-1/p^2} \,,
\end{align*}
and
\begin{align}
| f_+^+(z_1,v) - f_+^+(z_2,v) | &\leq \left (\dfrac{\|I_-\|_\infty}{T_-^+} +  \dfrac{\|I_+\|_\infty}{T_+^+} \right )
\left[\exp\left(-T_+^+\frac \bullet{v-c}\right)\right]_{\mathcal C^{0,1-1/p^2}(z_1,z_2)}|z_1 - z_2|^{1-1/p^2}
\nonumber\\
& \quad + \int_{0}^{\frac{z_1}{v-c}} \left| I_+(z_1 - s(v-c)) - I_+(z_2 - s(v-c))\right | \exp(-T_+^+ s)\, ds \nonumber \\
&\leq C z_1^{1/p^2-1} |z_1 - z_2|^{1-1/p^2}\nonumber \\
& \quad + C \int_{0}^{\frac{z_1}{v-c}} \left (\left( z_1 - s(v-c) \right )^{-1/(pp')}\vee 1\right ) |z_1-z_2|^{1-1/p^2}   \exp(-T_+^+ s)\, ds\, .\label{eqhbglygvbl}
\end{align}
The last integral contribution can be evaluated as follows,
\begin{align}
& \int_{0}^{\frac{z_1}{v-c}} \left (\left( z_1 - s(v-c) \right )^{-1/(pp')}\vee 1\right )    \exp(-T_+^+ s)\, ds \nonumber\\
& =  \exp\left(-T_+^+ \frac{z_1}{v-c}\right ) \int_{0}^{\frac{z_1}{v-c}} \left ((s'(v-c))^{-1/(pp')}\vee 1\right )    \exp(-T_+^+ s')\, ds'\nonumber\\
& \leq C (v-c)^{-1/(pp')} \exp\left(-T_+^+ \frac{z_1}{v-c}\right ) \int_{0}^{+\infty} \left (s^{-1/(pp')}\vee 1\right )    \exp(-T_+^+ s)\, ds\nonumber\\
& \leq C z_1 ^{-1/(pp')}\,.\label{eqhbglygvbl2}
\end{align}
Combining \eqref{eqhbglygvbl} and \eqref{eqhbglygvbl2}, we deduce
\begin{align*}
| f_+^\pm(z_1,v) - f_+^\pm(z_2,v) |& \leq C \left( z_1^{1/p^2-1}\vee z_1^{-1/(pp')}\vee 1\right ) |z_1-z_2|^{1-1/p^2} \\
\leq &  C \left( z_1^{1/p^2-1}\vee 1\right ) |z_1-z_2|^{1-1/p^2}\, .
\end{align*}
The last inequality is a consequence of $1-1/p ^2 > 1/(pp')$. 
Then, we update the H\"older estimate \eqref{jbhlmiugbi} to
\begin{equation*}
\left |I_+(z_1) - I_+(z_2) \right | \leq C  \left( z_1^{1/p'+1/p^3-1}\vee 1  \right)|z_1-z_2|^{1-1/p^3}\, .
\end{equation*}
Iteratively, for all $n\in \N$,  we obtain the existence of $C_n$ such that  
\begin{equation}\label{eq:macro holder I}
\left |I_+(z_1) - I_+(z_2) \right | \leq C_n  \left( z_1^{1/p'+1/p^n-1}\vee 1  \right)|z_1-z_2|^{1-1/p^n}\, .
\end{equation}
A similar result holds true on the negative side, \ie for $z<0$. 

We can establish the following result as a by-product of this recursive H\"older regularity. 

\begin{proposition}\label{prop:rho lip}
The macroscopic quantities $\rho_+^\pm$ (resp. $\rho_-^\pm$)  are Lipschitz continuous, locally uniformly on $(0,+\infty)$ (resp. on $(-\infty,0)$).
\end{proposition}
\begin{proof}
Firstly, we transfer the macroscopic H\"older regularity \eqref{eq:macro holder I} to partial H\"older regularity with respect to velocity. We have successively, for $-\v0<v_1<v_2<c$,
\begin{align}
| f_+^-(z,v_1) - f_+^-(z,v_2) | & \leq C_n  \int_0^{+\infty}  \left( z_1^{1/p'+1/p^n-1}\vee 1  \right) s^{1-1/p^n} |v_1-v_2|^{1-1/p^n}   \exp(-T_+^- s)\, ds \nonumber \\
& \leq C_n \left( z_1^{1/p'+1/p^n-1}\vee 1  \right) |v_1-v_2|^{1-1/p^n}\,, \label{eq:byofftu}
\end{align}
and for  $c<v_1<v_2<\v0$,
\begin{align*}
& | f_+^+(z,v_1) - f_+^+(z,v_2) |
\\
& \leq C_n  \left(\int_0^{+\infty} \left( (s(v_1-c))^{1/p'+1/p^n-1}\vee 1  \right)s^{1-1/p^n} |v_1-v_2|^{1-1/p^n}  \exp(-T_-^+ s)\, ds\right)
\exp\left(-T_+^+\frac z{v_1-c}\right)\\
&\quad + \left (\dfrac{\|I_-\|_\infty}{T_-^+} +  \dfrac{\|I_+\|_\infty}{T_+^+} \right ) \left[\exp\left(-T_+^+\frac z{\bullet-c}\right)\right]_{\mathcal C^{0,1-1/p^n}(v_1,v_2)}|v_1 - v_2|^{1-1/p^n}\\
&\quad + C_n \int_0^{\frac{z}{v_2-c}} \left( (z-s(v_2-c))^{1/p'+1/p^n-1}\vee 1  \right)s^{1-1/p^n} |v_1-v_2|^{1-1/p^n}  \exp(-T_+^+ s)\, ds
\\
& \leq C_n  \left(\int_0^{+\infty} \left( s^{1/p'}\vee s^{1-1/p^n}  \right)  \exp(-T_-^+ s)\, ds\right)
\left( (v_1-c)^{1/p'+1/p^n-1}\exp\left(-T_+^+\frac z{v_1-c}\right)\right)|v_1-v_2|^{1-1/p^n} \\
&\quad + \left (\dfrac{\|I_-\|_\infty}{T_-^+} +  \dfrac{\|I_+\|_\infty}{T_+^+} \right ) \left[\exp\left(-T_+^+\frac z{\bullet-c}\right)\right]_{\mathcal C^{0,1-1/p^n}(v_1,v_2)}|v_1 - v_2|^{1-1/p^n}\\
&\quad + C_n \exp\left (-T_+^+ \frac{z}{v_2-c}\right ) \left(\int_0^{\frac{z}{v_2-c}} \left( (s'(v_2-c))^{1/p'+1/p^n-1}\vee 1  \right)\left(\frac{z}{v_2-c}\right)^{1-1/p^n}  \exp(-T_+^+ s')\, ds'\right)  |v_1-v_2|^{1-1/p^n}\,.
\end{align*}
Let $q = p^n$ be the exponent such that $1/q' = 1-1/p^n$, and $\epsilon>0$ to be chosen below. The H\"older regularity of the exponential function with respect to $v$ is evaluated as follows,
\begin{align*}
& \left| \exp\left(-T_+^+\frac{ z}{v_1-c}\right) - \exp\left(-T_+^+\frac {z}{v_2-c}\right) \right| \\
& \leq \left( \int_{v_1}^{v_2} \left(\dfrac{T_+^+ z}{(w-c)^2}\right )^q \exp\left ( - q \frac{z}{w-c}\right )\, dw \right )^{1/q} |v_1 - v_2|^{1/q'}\\
& \leq C z  \left( \int_{v_1}^{v_2} \left(\dfrac{1}{w-c}\right )^{2q-1+\epsilon q} \exp\left ( - q \frac{z}{w-c}\right ) (w-c)^{-1+\epsilon q}  \, dw \right )^{1/q} |v_1 - v_2|^{1/q'}
\\
&\leq C z^{1 -2  + 1/q-\epsilon} \left( \int_{v_1}^{v_2}  (w-c)^{-1+\epsilon q} \, dw \right )^{1/q} |v_1 - v_2|^{1/q'}  \\
& \leq C(\epsilon) z^{1/p^n-1-\epsilon} |v_1 - v_2|^{1-1/p^n} \, .
\end{align*}
Consequently, we deduce:
\begin{align}
| f_+^+(z,v_1) - f_+^+(z,v_2) | \nonumber
& \leq C_n  
z^{1/p'+1/p^n-1}|v_1-v_2|^{1-1/p^n} \\
&\quad + C(\epsilon) z^{1/p^n-1-\epsilon}  |v_1 - v_2|^{1-1/p^n} + C_n  
z^{1/p'+1/p^n-1}  |v_1 - v_2|^{1-1/p^n}\nonumber\\
& \leq C_n(\epsilon)  \left(z^{1/p^n-1-\epsilon}\vee z^{1/p'+1/p^n-1}   \right ) |v_1 - v_2|^{1-1/p^n}  \label{eq:juhoiuphgl}\,.
\end{align}
Combining \eqref{eq:byofftu} and \eqref{eq:juhoiuphgl}, we obtain eventually,
\begin{equation}
| f_+^\pm(z,v_1) - f_+^\pm(z,v_2) | \leq C_n  (\epsilon) \left(z^{1/p^n-1-\epsilon}\vee 1\right ) |v_1 - v_2|^{1-1/p^n}\,.
\label{eq:jbmgiufvghl}
\end{equation}

Secondly, we compute directly the derivatives of $\rho_+^-$ and $\rho_+^+$, based on \eqref{eq:f I}: 
\begin{align*}
\dfrac {d \rho_+^-}{dz}(z) & = \int_{\{v<c\}}\dfrac1{v-c}\left( I_+(z) - T_+^- f_+^-(z,v)\right)\, d\nu(v) 
\\
\dfrac {d \rho_+^+}{dz}(z) & = \int_{\{v>c\}}\dfrac1{v-c}\left( I_+(z) - T_+^+ f_+^+(z,v)\right)\, d\nu(v)\, . 
\end{align*}
We observe that for $z>0$, $I_+(z) = \lim_{v\nearrow c} T_+^- f_+^-(z,v) $, as can be seen directly on \eqref{eq:duhamel}. We have $I_-(z) = \lim_{v\searrow c} T_+^+ f_+^+(z,v) $ as well. In the latter case, the condition $z>0$ is crucial. 
Consequently, we can use the H\"older regularity of $f$, with respect to velocity \eqref{eq:jbmgiufvghl}, so as to get
\begin{align*}
\left|\dfrac {d \rho_+^+}{dz}(z)\right| & \leq   T_+^+ \int_{\{v>c\}}\dfrac1{v-c}\left| f_+^+(z,c) -  f_+^+(z,v)\right|\, d\nu(v)\\
& \leq   T_+^+ \left( \int_{\{v>c\}}  \left(\dfrac1{v-c}\left| f_+^+(z,c) -  f_+^+(z,v)\right|\right)^{p'}\, dv\right)^{1/p'} \|\omega\|_p\\
& \leq C_n(\epsilon) \left(z^{1/p^n-1-\epsilon}\vee 1\right ) \left( \int_{\{v>c\}}\left (  (v-c)^{-1/p^n}\right )^{p'}\, dv\right)^{1/p'}\,. 
\end{align*}
Therefore, it is sufficient to choose $n$ large enough such that $p^n>p'$, in order to ensure integrability in the r.h.s., and thus boundedness of the derivative of $\rho_+^+$. The same estimate holds true for $\rho_+^-$, and also on the negative side $z<0$. 

Note that $\epsilon>0$ can be chosen almost arbitrarily at this stage. The next result will impose some condition on it.  
\end{proof}

\begin{remark}
There is some subtlety in the last argument, as it may not be legitimate to take the limit $v\to c$ if $c\notin \mathrm{Int}\,V$. To circumvent this small issue, it is possible to extend the formula along characteristic lines \eqref{eq:duhamel} for all $v$. This leads to a non optimal result, as the density is certainly smooth if $c$ is outside $V$. 
\end{remark}

The next corollary is the last step of our bootstrap argument. Such regularity will be required in the proof of Lemma \ref{lem:open}. 

\begin{corollary}\label{cor:vyotflu}
The derivative of the macroscopic quantity $\frac{dI_+}{dz}$ (resp. $\frac{dI_-}{dz}$) is H\"older continuous, locally uniformly on $(0,+\infty)$ (resp. on $(-\infty,0)$).
\end{corollary}

\begin{proof}
Recall that $n$ has been chosen such that $p^n>p'$. Let $\epsilon$ belong to $(0, 1/p^n)$. By differentiating \eqref{eq:duhamel} with respect to $z$, we obtain successively
\begin{equation*}
\left |\partial_z f_+^-(z,v) \right |\leq  \int_0^{+\infty} \left |\frac{dI_+}{dz}(z - s(v-c))\right | \exp(-T_+^- s)\, ds \leq C  \left(z^{1/p^n-1-\epsilon}\vee 1\right ) \, ,
\end{equation*}
and
\begin{align*}
\left |\partial_z f_+^+(z,v)\right | &\leq  - \left( \int_0^{+\infty}  I_-(-s(v-c)) \exp(-T_-^+ s)\, ds \right)  \frac {T_+^+}{v-c}\exp\left(-T_+^+\frac z{v-c}\right)\medskip\\ 
& \quad +   \frac {I_+(0)}{v-c}\exp\left(-T_+^+\frac z{v-c}\right) + \int_0^{\frac z{v-c}}\left |  \frac{dI_+}{dz}(z - s(v-c))\right | \exp(-T_+^+ s) \, ds  \\
& \leq C z^{-1} + C \int_0^{\frac z{v-c}}  \left((z - s(v-c))^{1/p^n-1-\epsilon}\vee 1\right ) \exp(-T_+^+ s) \, ds \\
& \leq C z^{-1} + C \exp\left(-T_+^+\frac z{v-c}\right) \int_0^{\frac z{v-c}}  \left((s'(v-c))^{1/p^n-1-\epsilon}\vee 1\right ) \exp(-T_+^+ s') \, ds' \\
& \leq C z^{-1} + C (v-c)^{1/p^n-1-\epsilon}\exp\left(-T_+^+\frac z{v-c}\right) \int_0^{+\infty}  \left (s^{1/p^n-1-\epsilon}\vee 1\right ) \exp(-T_+^+ s) \, ds\\
& \leq  C \left(  z^{-1}\vee z^{1/p^n-1-\epsilon}\right )  \, . 
\end{align*}
Note that the condition $\epsilon< 1/p^n$ is compulsory  for the integrability of $ s^{1/p^n-1-\epsilon} $ at $s = 0$. 
All in all, we get eventually,
\begin{equation}\label{jgdlpg	yvl}
\left |\partial_z f_+^\pm(z,v)\right |\leq   C \left(  z^{-1}\vee 1\right )\, .
\end{equation}

We conclude by using classical averaging lemma. By differentiating \eqref{eq:f I} with respect to $z$, we obtain
\begin{equation}\label{jelgbcf	uzhve}
(v-c) \partial^2_z f_+(z,v) = \frac{dI_+}{dz}(z) - T_+(v-c)\partial_z f_+(z,v)\, . 
\end{equation}
We deduce from \eqref{jgdlpg	yvl} and Proposition \ref{prop:rho lip} that the r.h.s. in \eqref{jelgbcf	uzhve} is bounded by $C \left(  z^{-1}\vee 1\right )$. Therefore, we have for $0<z_1<z_2$,
\begin{align*}
\left |\frac{dI_+}{dz}(z_1) - \frac{dI_+}{dz}(z_2) \right |&\leq\left ( \int_{|v-c|<\delta} + \int_{|v-c|>\delta} \right ) T_+(v-c) \left|\partial_z f_+(z_1,v) - \partial_z f_+(z_2,v)\right|  \omega(v)\,dv\\
&\leq C \left(  z_1^{-1}\vee 1\right ) \left( \delta^{1/p'}  + \delta^{1/p' - 1} |z_1 - z_2| \right) \|\omega\|_p \\
&\leq C \left(  z_1^{-1}\vee 1\right ) |z_1 - z_2|^{1/p'}  \, .
\end{align*}
This closes the series of bootstrap estimates. 
\end{proof}

\subsection{Regularity within exponential tails}

We derive regularity estimates for the  normalized density $g$, defined as follows,
\begin{equation}\label{eq:renormalized g}
g(z,v) = 
\begin{cases} e^{\lambda_+ z}f_{+}(z,v) & (z>0) \medskip\\
e^{-\lambda_- z}f_{-}(z,v) & (z<0)
\end{cases}
\end{equation}
For the sake of simplicity, we assume in this section that the p.d.f. $\omega$ is bounded below and above: there exists $\omega_0$ such that 
\begin{equation}\label{jlzgumlzvlvl}
\omega_0 \leq \omega\leq \omega_0^{-1} \text{ on $\supp\nu_\tau$ }\,.
\end{equation}
In fact, the results derived afterwards will be used under this restrictive condition (Lemma \ref{lem:asymptotic monotonicity}). 
The key observation is that $g$ satisfies the following equation for $z>0$,
\begin{equation}\label{eq:ugmfiugf}
(v-c)\partial_z g_+(z,v)  = J_+(z) - (T_+(v-c) - \lambda_+(v-c)) g_+(z,v) \,, \quad J_+(z) = e^{\lambda_+ z}I_{+}(z)\, .
\end{equation}
and similarly for $z<0$. We deduce from \eqref{eq:5856}-\eqref{eq:589808} that $g_+$ is uniformly bounded (recall that $F_+$ is bounded under the condition \eqref{jlzgumlzvlvl}).  Hence, the r.h.s. of \eqref{eq:ugmfiugf} is uniformly bounded. We deduce from  the one-dimensional averaging lemma that $J_+ = \int T_+ g_+\, d\nu$ is almost Lipschitz continuous, namely,
\begin{equation}\label{zefhglzyfgl}
\left | J_+(z_1)- J_+(z_2) \right | \leq 
C \left( \log\left (\left |z_1 - z_2\right |^{-1}\right )\vee 1 \right) |z_1 - z_2|\, . 
\end{equation}
With this estimate at hand, it is possible to reproduce the estimates of Section \ref{sec:Further regularity}, with the tumbling rate $R_+(v-c) = T_+(v-c) - \lambda_+(v-c) = (F_+(v))^{-1}$, which is bounded below and above. Notice that it is not even necessary to iterate the argument in order to obtain Lipschitz regularity as in Proposition \ref{prop:rho lip}. Indeed, we have obviously $p>p'$, since we assume $\omega \in L^\infty$. We cannot readily choose $p = \infty$ in the estimates because of the logarithmic correction in \eqref{zefhglzyfgl}, but taking any $p>2$ is sufficient for our purpose. Note also that the calculation leading to H\"older regularity with respect to velocity, as in \eqref{eq:jbmgiufvghl}, requires some adaptation because the new tumbling rate  $R_+$ has a linear dependency with respect to $v$. But this additional contribution is partially Lipschitz continuous, so that it does not affect the conclusion. 

\begin{corollary}\label{cor:exp lip}
Under the additional condition \eqref{jlzgumlzvlvl}, the macroscopic quantities $J_\pm$ are Lipschitz continuous, locally uniformly on $\R^*_\pm$, and  the normalized densities $g_\pm$ are H\"older continuous with respect to velocity locally uniformly on $\R^*_\pm$. More precisely, for all $p>2$, and $\eps <1/p$, there exists a constant $C$ such that for all $z>0$,
\begin{equation*}
| g_+^\pm(z,v_1) - g_+^\pm(z,v_2) | \leq C  \left(z^{1/p-1-\epsilon}\vee 1\right ) |v_1 - v_2|^{1-1/p}\,,
\end{equation*}
and similarly for $z<0$.  
\end{corollary}

\section{Monotonicity of the macroscopic quantities $\rho$, $I$}
\label{sec:monotonicity}

\subsection{Monotonicity of $\rho_\pm^\pm$.}

\label{ssec:monotonicity}

Let $c\in (c_*,c^*)$, and define $f$ as in Theorem \ref{theo:cluster}. The aim of this section is to establish monotonicity of some macroscopic quantities, despite the lack of monotonicity of $f$ with respect to space variable (see Figure \ref{fig:overshoot}). 

We recall the notation $\rho^+(x) = \int_{\{v>c\}} f^+(x,v)\, d\nu(v)$, resp. $\rho^-(x) = \int_{\{v<c\}} f^-(x,v)\, d\nu(v)$. Accordingly, we have
\beq\begin{cases} 
\displaystyle(\forall z>0)\quad I_+(z) = \int T_+(v-c) f(z,v)\, d\nu(v) =  T^+_+ \rho_+^+(z) + T^-_+\rho_+^-(z) \medskip\\
\displaystyle(\forall z<0)\quad I_-(z) = \int T_-(v-c) f(z,v)\, d\nu(v) =  T^+_- \rho_-^+(z) + T^-_-\rho_-^-(z)
\end{cases}\label{bumivùozibipb}
\eeq

The following theorem is the cornerstone of the present study. 
\begin{theorem}\label{theo:mono}
Under the assumptions of Theorem \ref{theo:cluster}, 
both $\rho_+^+$ and $\rho_+^-$ are decreasing for $z>0$. Reversely,   $\rho_-^+$ and $\rho_-^-$ are increasing for $z<0$.
\end{theorem}

We split the proof into several steps. We begin with an easy case, when $\nu$ is close to the symmetric combination of two Dirac masses. Then, we deform continuously the measure $\nu$, and check carefully that none of the macroscopic quantities $\rho_\pm^\pm$ can change monotonicity. This requires many technical estimates in order to control the continuous deformation of~$f$.\medskip

\noindent\textbf{Step  \#1. Initialization (the easy case).}
We first state a proposition establishing monotonicity under the very restrictive condition that the measure $\nu$ is concentrated around two symmetric velocities $\{\pm \v0\}$. 
In the extreme case where
\beq\label{eq:2vel} \nu = \frac12 \delta_{-\v0} + \frac12 \delta_{\v0}\, , \eeq
the conditions \eqref{eq:condition c*} and \eqref{eq:condition c* 2} which are required for exponential decay on both sides, read as follows,
\beq \label{eq:condition v0} \frac{1+\chi_+}{\v0 - c} - \frac{1-\chi_+}{\v0+c} >0 \quad \text{and}\quad \frac{1+\chi_-}{\v0+c} - \frac{1-\chi_-}{\v0 - c}   >0\, . \eeq
They are equivalent to the following conditions, provided that $\v0>|c|$,
\[ \v0 \chi_+ + c >0 \quad \text{and}\quad  \v0 \chi_- - c >0\,. \]
Hence, it is required that $c$ belongs to the interval $\left(- \v0\chi_+,\v0\chi_-\right)$. Note that the latter is not empty since $\chi_- + \chi_+ = 2 \chi_S>0$. 

\begin{proposition}[The approximate two velocity case]\label{prop:2 vel case}
Let $\v0 >|c| $ be such that both conditions in \eqref{eq:condition v0} are satisfied.   Assume that the support of $\nu$ is contained in $V_0\cup(-V_0)$, where $V_0 = [\v0-\delta_0,\v0]$, for $\delta_0$ small enough. Then $I$ is decreasing for $z>0$ (resp. increasing for $z<0$). 
\end{proposition}

\begin{proof}
The key observation is that the equilibrium of the model with only two velocities \eqref{eq:2vel} can be computed explicitly. In this case, $I$ has the expected monotonicity. 
The computations are facilitated by the fact that $(\v0-c)\rho^+ = (\v0+c)\rho^-$, since the flux \eqref{eq:fluxes}  vanish at equilibrium. 

If $V$ is not too far from $\{\pm \v0\}$, then the same result should hold, using again the fact that the flux $ \int (v-c) f(z,v)\, d\nu(v)$ is identically zero at equilibrium. This is indeed the case,
\begin{align}
 \dfrac {d I_+}{dz}(z)  & = \int T_+(v-c)\partial_z f_+(z,v)\, d\nu(v) = \int  \dfrac{T_+(v-c)}{(v-c)}\left( I_+(z) - T_+(v-c) f_+(z,v-c) \right)\, d\nu(v) \nonumber\\ 
&  = \left(\int   \dfrac{T_+(v-c)}{(v-c)} d\nu(v)\right)\left(\int   T_+(v-c) f_+(z,v-c)  d\nu(v)\right)\nonumber \\
&\quad - \int \dfrac{\left (T_+(v-c)\right )^2}{(v-c)} f_+(z,v-c)d\nu(v)  +   \zeta_+  \int (v-c) f_+(z,v)\, d\nu(v) \nonumber\\
& = \int \left(\la \frac {T_+(v-c)}{v-c}\ra T_+(v-c) - \dfrac{\left (T_+(v-c)\right )^2}{v-c} + \zeta_+ (v-c)\right) f_+(z,v)\, d\nu(v)\, , \label{eq:I+ decreasing}
\end{align}
where $\zeta_+\in \R$ is arbitrary, and $\la \cdot\ra$ denote the average with respect to measure $\nu$. On the support of $\nu$, we find as $\delta_0 \to 0$,
\begin{multline}\label{eq:2 vel}\la \frac {T_+(v-c)}{v-c}\ra T_+(v-c) - \dfrac{\left (T_+(v-c)\right )^2}{v-c} +  \zeta_+  (v-c) \\ \sim  \frac12\left(\dfrac{T_+^+}{\v0 - c} -  \frac{T_+^-}{\v0+c}  \right) T_+^{\pm} - \dfrac{(T_+^\pm) ^2}{\pm \v0 - c} + \zeta_+(\pm \v0 - c)\, ,\end{multline}
where the sign $\pm$ is determined by the sign of $v$. We seek a condition on $\zeta_+$ such that the r.h.s. of \eqref{eq:2 vel} is always negative. The following two inequalities should hold, corresponding to the two possible choices of sign,
\beq\label{eq:bmupigmli}
\begin{cases}
\displaystyle -\frac12 \frac{(T_+^+)^2}{\v0 - c}-\frac12 \frac{T_+^- T_+^+}{\v0 + c} + \zeta_+(\v0-c) < 0 \medskip\\
\displaystyle \frac12 \frac{T_+^+ T_+^-}{\v0 - c}+\frac12 \frac{(T_+^-) ^2}{\v0 + c} - \zeta_+(\v0+c) < 0
\end{cases}
\eeq
Clearly, such a real number $\zeta_+ >0$ exists if, and only if,
\begin{align*} 
\frac12 \frac{T_+^+ T_+^-}{(\v0 - c)(\v0+c)}+\frac12 \frac{(T_+^-) ^2}{(\v0 + c) ^2} & < \frac12 \frac{(T_+^+)^2}{(\v0 - c)^2}+\frac12 \frac{T_+^- T_+^+}{(\v0 + c)(\v0 - c)} \\
\Leftrightarrow\quad \frac{(T_+^-) ^2}{(\v0 + c) ^2} & <  \frac{(T_+^+)^2}{(\v0 - c)^2} \, . \end{align*} 
We find that this is equivalent to the first condition in \eqref{eq:condition v0}. Within this choice of $\zeta_+$, we deduce easily from \eqref{eq:I+ decreasing}-\eqref{eq:2 vel} that $I_+$ is decreasing, provided that $\delta_0$ is small enough.

The same holds true for $I_-$, with possibly another choice of constant $\zeta_-$. 
\end{proof}

\noindent\textbf{Step  \#2. Enhancement of monotonicity.}
The next proposition is the crucial step in the proof. It includes the cancellation property that accounts for the compensations in the velocity profile.   

\begin{proposition}[Enhancement of monotonicity]
\label{lem:monotonicity}
Assume that $I_+$ is non-increasing (resp. $I_-$ is non-decreasing). 
Then monotonicity turns out to be strict. Moreover we have the following stronger set of inequalities,
\[ \begin{cases}  \dfrac{d\rho_+^+}{dz}(z) < 0\quad \text{and}\quad \dfrac{d\rho_+^-}{dz}(z) < 0 & \text{for $z\geq 0$}  \medskip\\
\dfrac{d\rho_-^+}{dz}(z) > 0 \quad \text{and}\quad\dfrac{d\rho_-^-}{dz}(z) > 0 & \text{for $z\leq0$}
\end{cases}   
\] 
\end{proposition}

\begin{proof}
We prove this statement by deciphering the shape of the velocity profile $f(z,\cdot)$. 


We consider first $z>0$. 
The easiest part of the statement concerns negative relative velocities, $v<c$: by integration along the characteristics \eqref{eq:duhamel}, we obtain
\beq\label{eq:f-charac}(\forall v<c)\quad  f_+^-(z,v) = \int_0^{+\infty}  I_+(z - s(v-c)) \exp(-T_+^- s)\, ds\, . \eeq
This quantity is clearly increasing with respect to $v$, provided that $I_+$ is non-increasing, and non constant on $(z,+\infty)$ for all $z>0$. This is indeed the case by assumption. 

Similarly, $f_-^+$ is decreasing with respect to $v$, for $z<0$, and $v>c$. 

Let introduce the notation $\bbf = Tf$. 
Interestingly, the stationary equation \eqref{eq:f I} is rewritten as \begin{equation}\label{eq:stat-cluster}
(v-c) \partial_z f(z,v) =  I(z) - \bbf(z,v)\, . 
\end{equation}
From \eqref{eq:f-charac} we also get the following inequality:
\beq(\forall v<c)\quad \bbf_+^-(z,v) = \int_0^{+\infty}  I_+(z - s(v-c)) T_+^-\exp(-T_+^- s)\, ds < I_+(z)\, .
\label{eq:g- I}\eeq
We deduce immediately that the density $\rho_+^-(z)$ is decreasing, since the combination of \eqref{eq:stat-cluster} and \eqref{eq:g- I} implies that  
\[ (\forall v<c) \quad \partial_z f_+^-(z,v) <0\; \Rightarrow \quad \dfrac{d\rho_+^-}{dz}(z) < 0\, .\] 

\begin{figure}
\begin{center}
\includegraphics[width = 0.8\linewidth]{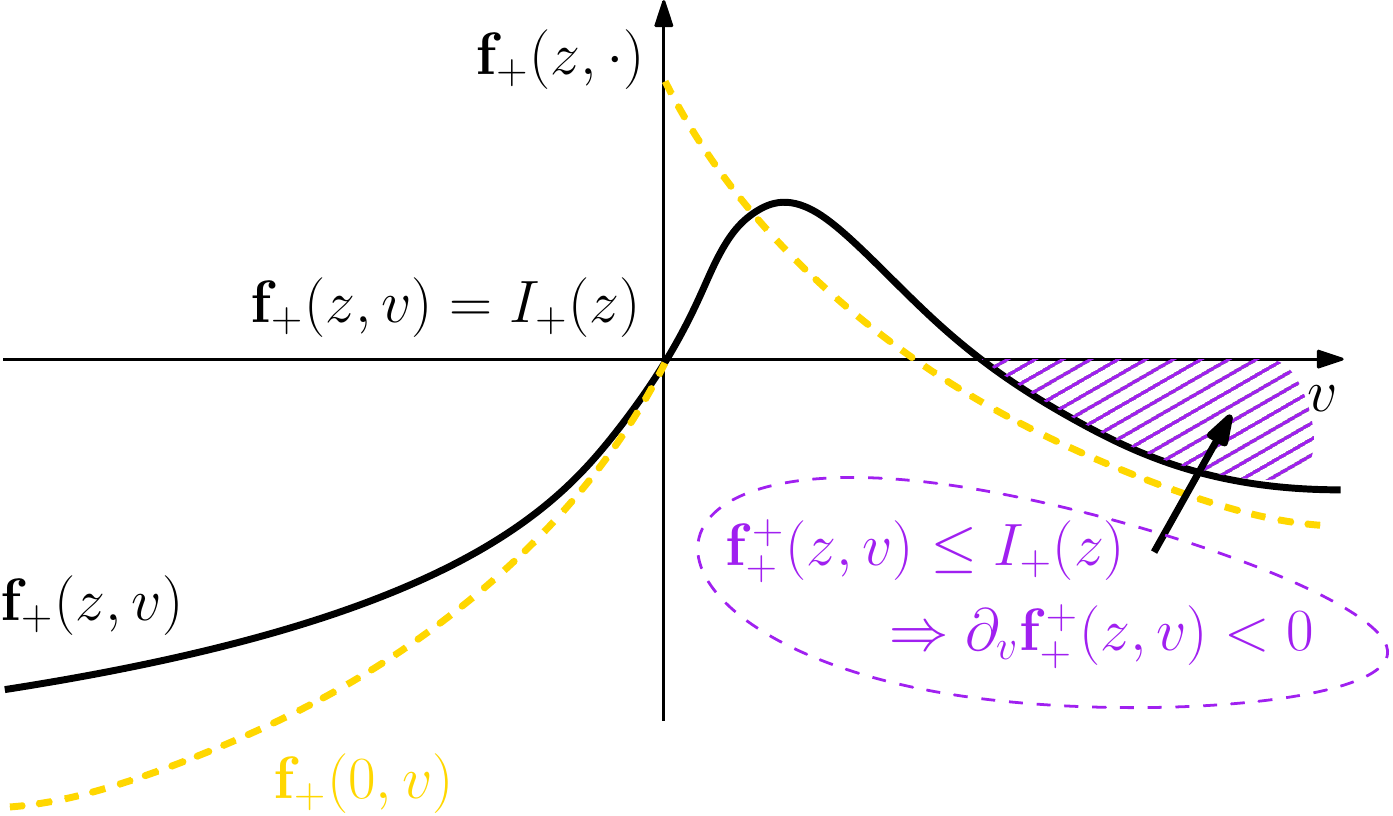}
\caption{\small Typical velocity profile of $\bbf= Tf$ for $z>0$, close to the transition at $z = 0$ (plain line). 
For the sake of comparison, $\bbf(0^+,v)$ is depicted in dashed line (see also Figure \ref{fig:overshoot} for numerical simulations). According to this cartoon, the inappropriate  sign $\partial_z \bbf_+^+(z,v)>0$ may  arise for large velocity only. One key feature is that the velocity profile is decreasing  in the zone where  $\partial_z \bbf_+^+(z,v)>0$. 
}
\label{fig:velocityprofile}
\end{center}
\end{figure}

The case of positive relative velocities $v>c$ requires more work. Indeed, the velocity profile $v\mapsto f_+^+(z,v)$ is not necessarily monotonic for $v>c$, on the contrary to $v<c$ (see Figures \ref{fig:overshoot} and \ref{fig:velocityprofile}). However, we are able to establish the following properties, which are sufficient for our purpose. 

\begin{lemma}\label{lem:qualitative profile}
The function $\bbf_+^+$ is decreasing with respect to $v$ on the set $\{ \bbf_+^+ \leq I_+ \}\cap\{v>c\}$. As a consequence, for all $z>0$ there exists $v_*(z)>c$ such that 
$\bbf_+^+(z,v)> I_+(z)$ for $z\in (0,v_*(z))\cap V$, and $\bbf_+^+(z,v)< I_+(z)$  for $z\in (v_*(z),\v0)\cap V$. 
\end{lemma}

Note that we allow for $v_*(z) \geq \v0$. In the latter case, $\bbf_+^+(z,v)\geq I_+(z)$ for all $v>c$.

\begin{proof}
The function $\bbf_+^+ - I_+$ satisfies the following damped transport equation with a non-negative source term,
\[    (v-c) \partial_z \left( \bbf_+^+(z,v) - I_+(z)\right)  + T_+^+\left( \bbf_+^+(z,v) - I_+(z) \right) = -(v-c)\dfrac{dI_+}{dz}(z) \geq 0 \,  . \]
By solving this transport equation along characteristics lines, it is immediate to establish that, if $\bbf_+^+(z_0,v)>  I_+(z_0)$, then $\bbf_+^+(z,v)> I_+(z)$ for all $z\geq z_0$. On the contrary, if $\bbf_+^+(z_0,v)\leq  I_+(z_0)$, then $\bbf_+^+(z,v) \leq  I_+(z)$ for all $0\leq z\leq z_0$. 

Take $c<v_1<v_2$. We have 
\[ (v_2-c)\partial_z \left( \bbf_+^+(z,v_2) - \bbf_+^+(z,v_1) \right) +  (v_2-v_1) \partial_z  \bbf_+^+(z,v_1)    =  - T_+^+ \left(  \bbf^+_+(z,v_2) - \bbf^+_+(z,v_1)\right)  \,  . \]
Introduce $h(z) = ( \bbf_+^+(z,v_2) - \bbf_+^+(z,v_1) )/(v_2-v_1)$. It satisfies the following damped transport equation with source term,
\beq \label{eq:transport h}  (v_2-c) \partial_z h(z) + T_+^+ h(z) =  - \partial_z  \bbf_+^+(z,v_1) = \dfrac{T_+^+}{v_1-c}\left( \bbf_+^+(z,v_1) - I_+(z) \right) \, . \eeq

If $z_0>0 $ is such that $\bbf_+^+(z_0,v_1) - I(z_0)\leq 0$, then $\bbf_+^+(z,v_1) - I_+(z)\leq 0$ for all $0\leq z\leq z_0$. On the other hand, we clearly have $h(0)<0$ since 
\beq\label{eq:h(0)} h(0) = \dfrac{\bbf_+^+(0,v_2) - \bbf_+^+(0,v_1)}{v_2 - v_1} = T_+^+\dfrac{f_+^+(0,v_2) - f_+^+(0,v_1)}{v_2 - v_1} = T_+^+\dfrac{f_-^+(0,v_2) - f_-^+(0,v_1)}{v_2 - v_1} < 0\, , \eeq
(recall that $f_-^+$ is decreasing with respect to $v$, for $z<0$, and $v>c$).

We deduce from \eqref{eq:transport h} and \eqref{eq:h(0)} that $h(z)< 0$ for all $0\leq z\leq z_0$ as soon as $\bbf_+^+(z_0,v_1) - I(z_0)\leq 0$. Therefore, for all $v_2>v_1$, $\bbf_+^+(z_0,v_2)< \bbf_+^+(z_0,v_1)$. As a consequence, the function $\bbf_+^+$ is decreasing with respect to $v$ on $\{ \bbf_+^+ \leq I \}\cap\{v>c\}$. 

The velocity threshold $v_*(z)$ is constructed by choosing the smallest velocity such that $\bbf_+^+(z,v)<I_+(z)$. 
\end{proof}

To conclude the proof of Proposition \ref{lem:monotonicity}, we establish that $\rho_+^+$ is decreasing as well, using compensations between lower and higher velocities. The derivative writes
\beq \dfrac {d \rho_+^+}{dz}(z) = \int_{\{v>c\}} \dfrac1{v-c}\left( I_+(z) - \bbf_+^+(z,v) \right)\, d\nu(v)\, . \label{eq:ddx rho} \eeq
The key observation is that, were we ignoring the factor $1/(v-c)$ inside the integral, the r.h.s. in \eqref{eq:ddx rho} would be
\beq \int_{\{v>c\}}  \left( I_+(z) - \bbf_+^+(z,v) \right)\, d\nu(v) =  \int_{\{v<c\}}  \left(  \bbf_+^-(z,v) - I_+(z)\right)\, d\nu(v) < 0 \, , \label{eq:barycenter} \eeq
by the very definition of $I = \int \bbf\, d\nu$. Indeed, it is negative since $\bbf_+^-(z,v) < I_+(z)$ for all $v<c$ \eqref{eq:g- I}.

It is sufficient to show that the r.h.s. of \eqref{eq:ddx rho} can be handled in the same way, because lower velocities contribute more to the integral, and because the inappropriate sign is contributed by higher velocities according to Lemma \ref{lem:qualitative profile}. After integration by parts, we obtain
\begin{align} 
\dfrac {d \rho_+^+}{dz}(z)  & = 
\left[\dfrac{1}{v - c} \int_c^{v} \left( I_+(z) - \bbf_+^+(z,v') \right)\, d\nu(v')\right]_{c}^{\v0} \nonumber\\
& \quad +   \int_{\{v<c\}} \dfrac1{(v-c)^2} \left( \int_c^{v} \left( I_+(z) - \bbf_+^+(z,v') \right)\, d\nu(v')\right)\, dv \label{eq:v = c non singulier}  \\
& = \dfrac{1}{\vm - c} \int_c^{\vm} \left( I_+(z) - \bbf_+^+(z,v') \right)\, d\nu(v') \nonumber\\
& \quad  +   \int_{\{v<c\}} \dfrac1{(v-c)^2} \left( \int_c^{v} \left( I_+(z) - \bbf_+^+(z,v') \right)\, d\nu(v')\right)\, dv\, .\label{eq:dz rho++}
\end{align}
Notice that integration by parts makes sense here since the possibly singular term at $v = c$ in \eqref{eq:v = c non singulier} vanishes because $ I_+(z) - \bbf_+^+(z,v) = \mathcal O(|v-c|^{1-1/p^n})$ pointwise for $z>0$. Hence, 
\begin{align*}
\left|\dfrac1{v-c}\int_c^{v} \left( I_+(z) - \bbf_+^+(z,v') \right)\, d\nu(v')\right| & \leq  C   
\dfrac1{v-c} \left( \int_c^{v} |v'-c|^{(1 - 1/p^n)p'}\, d v'\right)^{1/p'} \|\omega\|_p   \\
& \leq  C   
|v-c|^{ 1/p' - 1/p^n}   \|\omega\|_p \, .
\end{align*} 
As in the proof of Proposition \ref{prop:rho lip}, we find that taking $n$ large enough, the latter contribution converges to 0 as $v\to c$.  

We deduce from \eqref{eq:barycenter} that the first contribution in \eqref{eq:dz rho++} is negative. On the other hand, the cumulative function 
\[\int_c^{v} \left( I_+(z) - \bbf_+^+(z,v') \right)\, d\nu(v')\, ,\]
is everywhere negative. Indeed, according to Lemma \ref{lem:qualitative profile}, this cumulative function is decreasing, then increasing (and possibly only decreasing) on $\{v>c\}$. Therefore, it reaches its maximum value at $v = c$ or at $v = \v0$. 
Since it vanishes at $v = c$ and it is negative at $v = \vm$ \eqref{eq:barycenter}, it is everywhere non positive. 
%

As a conclusion, both contributions in \eqref{eq:dz rho++} are negative, so $\rho_+^+$ is decreasing. 
\end{proof}

\begin{remark}
As can be observed on Figure \ref{fig:overshoot}, the threshold velocity $v_*(z)$ is increasing with respect to $z$. 
We present briefly a formal proof of this secondary statement.
For the sake of simplicity, let assume that all quantities, including $v_*(z)$ have regular variations with respect to $z,v$. We differentiate the relation $I_+(z) = \bbf_+^+(z,v_*(z))$  with respect to $z$, 
\begin{equation*}
 \dfrac{d  I_+}{dz}(z)  = \partial_z \bbf_+^+(z,v_*(z)) +  \partial_v \bbf_+^+(z,v_*(z)) \dfrac{d v_*}{dz}(z)\, .  
\end{equation*}
However, $\partial_z \bbf^+(z,v_*(z)) = 0$ by definition of $v_*(z)$, and equation \eqref{eq:stat-cluster}. Therefore, 
\[ \dfrac{dv_*}{dz}(z) = \left(\partial_v \bbf_+^+(z,v_*(z))\right)^{-1} \dfrac{d I_+}{dz} (z) > 0\, , \]
since $\bbf_+^+$ is decreasing at $v = v_*(z)$.

This illustrates the fact that the monotonicity of $f$ is getting better and better as $z$ increases, since the zone with inappropriate monotonicity shrinks (and may eventually disappears as in Figure \ref{fig:overshoot}). Hence, the lack of monotonicity of $f$ is clearly a consequence of overshoot phenomena around the origin where all signs are changing in the tumbling rate $T$. 
\end{remark}

\noindent\textbf{Step  \#3. Propagation of monotonicity.} 
In the last step, we propagate the monotonicity of $\rho_\pm^\pm$ along a continuous deformation of the measure $\nu$. 
Let $(\nu_\tau)_{\tau\geq 0}$ be a path of probability measures satisfying Assumption $\Ac$, continuous for the strong topology of $L^p$, and  satisfying the following properties:
\[ \begin{cases}
\text{$\nu_0$ is the uniform measure on $V_0\cup(-V_0)$ , where $V_0$ is defined in Proposition \ref{prop:2 vel case}}\medskip\\
(\forall \tau)\quad \supp\nu_\tau = [\vt, \v0]\cup [-\v0,-\vt]
\medskip\\
(\forall T)\;(\exists \omega_0(T)>0)\; (\forall \tau \leq T)\quad  \omega_0(T) \leq \omega_\tau\leq \omega_0(T)^{-1} \text{ on $\supp\nu_\tau$ }  \medskip\\
\displaystyle\lim_{\tau\to +\infty}\nu_\tau =  \nu \quad L^{p}\, .
\end{cases}\]
Here, $\tau\mapsto \vt$ is a non-increasing, continuous function such that $\vt|_{\tau=0} = \v0-\delta_0$, where $\delta_0$ is chosen according to Proposition \ref{prop:2 vel case}. Note that $c\notin \supp \nu_0$ by definition of $V_0$. However, it may happen that $\vt$ crosses $c$. At this point, the regularity of $f$ degenerates. We shall analyse  carefully  this point (see Lemma \ref{lem:zero monotonicity}). Besides, this careful analysis brings some useful quantitative information. 

We shall prove that for all $\tau\geq 0$, $I_\tau$ is decreasing on $\R_+^*$, and increasing on $\R_-^*$. First, we can deduce from the refined asymptotic of $f$ obtained in Theorem \ref{theo:cluster} that this is true for large $z$, with some quantitative estimates which are locally uniform with respect to $\tau$. 


\begin{lemma}[Asymptotic monotonicity]\label{lem:asymptotic monotonicity}
Let $T>0$. 
There exists $C>0$, such that 
\[ (\forall \tau\leq T)\quad 
\begin{cases} 
 (\forall z>1) & \dfrac{d I_{\tau+}}{dz}(z) \leq -\dfrac1C e^{-\lambda_+ z} + C e^{-\lambda_+ z-\vartheta_+' z } \medskip\\ 
 (\forall z<-1) & \dfrac{d I_{\tau-}}{dz}(z) \geq \dfrac1C e^{\lambda_- z} - Ce^{\lambda_- z+\vartheta_-' z } \end{cases}
\]
where $\vartheta_\pm' = \frac2{3}\vartheta_+$.
\end{lemma}
\begin{proof}
We restrict to $z>0$. The procedure for $z<0$ is the same.  We omit the dependency upon $\tau$ in the notation for the sake of clarity.
 
We control explicitly the behaviour of $f$ as $z\to +\infty$ \eqref{eq:spectral gap 1}. 
Recall the definitions of $\kappa_\pm$ \eqref{eq:def mu}, 
\beq \label{eq:mu+-}
\begin{cases}
\kappa_+ = \dfrac{\int f(0,v) (v-c)^2 F_+(v)\, d\nu(v)}{\int (v-c)^2 F_+(v)^2\, d\nu(v)} = e^{\lambda_+ z}\dfrac{\int f_+(z,v) (v-c)^2 F_+(v)\, d\nu(v)}{\int (v-c)^2 F_+(v)^2\, d\nu(v)}\,,  \medskip\\
\kappa_- = \dfrac{\int f(0,v) (v-c)^2 F_-(v)\, d\nu(v)}{\int (v-c)^2 F_-(v)^2\, d\nu(v)} = e^{-\lambda_- z}\dfrac{\int f_-(z,v) (v-c)^2 F_-(v)\, d\nu(v)}{\int (v-c)^2 F_-(v)^2\, d\nu(v)} \, .
\end{cases}
\eeq
Let introduce the corrector $h$ such that 
\beq\label{eq:taylor} f(z,v) - \kappa_+ F_+(v) e^{-\lambda_+ z} = h(z,v) \, , \quad  \left(\int (v-c)^2 h(z,v)^2 d\nu(v)\right)^{1/2} \leq C e^{-\lambda_+ z-\vartheta_+ z }\, .  \eeq
The latter bound requires some improvement, as it controls the $L^2$ norm of $h$ with a weight, whereas some unweighted $L^1$ bound would be more appropriate for our purpose. This is done as follows (recall the pointwise decay bound \eqref{eq:589808}),
\begin{align}
\int |h|\, d\nu & = \int_{|v-c|<\delta} |h(z,v)|\, d\nu(v) + \int_{|v-c|>\delta} |h(z,v)|\, d\nu(v) \nonumber \\
& \leq C e^{-\lambda_+ z} \delta \omega_0(T)^{-1} + \left( \int_{|v-c|>\delta} (v-c)^2|h(z,v)|^2\, d\nu(v)\right)^{1/2}\left( \int_{|v-c|>\delta} \dfrac1{(v-c)^2} \, d\nu(v)\right)^{1/2} \nonumber \\
& \leq C e^{-\lambda_+ z} \delta   + C e^{-\lambda_+ z - \vartheta_+ z} \delta^{- 1/2}\omega_0(T)^{-1/2}  \nonumber \\
& \leq C e^{-\lambda_+ z} e^{-\vartheta_+' z }\, , \label{eq:L1 bound h}
\end{align}
where $\vartheta_+' = \frac2{3}\vartheta_+$.

Recall the calculation of the derivative of $I_+$,
\begin{equation}
\dfrac {d I_+}{dz} (z) = \int \dfrac {T_+(v-c)}{v-c} \left( I_+(z)  - T_+(v-c) f(z,v)  \right)\, d\nu(v)\,. \label{eq:I decreasing}
\end{equation}
Integrating \eqref{eq:taylor} against $T\nu$ with respect to velocity, we obtain
\[
\left| I_+(z) - \kappa_+ e^{-\lambda_+z}\right| \leq C e^{-\lambda_+ z} e^{-\vartheta_+' z }\, ,  
\]
(recall that $\int T_+ F_+\, d\nu = 1$ by definition). On the other hand, we have the pointwise identity,
\[T_+(v-c) f(z,v) - \kappa_+ T_+(v-c) F_+(v)  e^{-\lambda_+z} = T_+(v-c) h(z,v)\, .\] 
By construction, we have 
\[T_+(v-c) F_+(v) - 1 = \frac{\lambda_+(v-c)}{T_+(v-c) - \lambda_+(v-c)} = \lambda_+ (v-c) F_+(v)\, .\]
The latter has the sign of $v-c$. Thus, we would be in position to conclude from \eqref{eq:I decreasing}, if we could abusively neglect the corrector terms, as follows, 
\begin{align*}
\nonumber\dfrac {d I_+}{dz} (z) &\approx \int \dfrac {T_+(v-c)}{v-c} \left( \kappa_+ e^{-\lambda_+z} - \kappa_+ T_+(v-c) F_+(v)  e^{-\lambda_+z}  \right)\, d\nu(v)\\
& \approx \kappa_+ e^{-\lambda_+z} \int  T_+(v-c)  F_+(v) \, d\nu(v) = \kappa_+ e^{-\lambda_+z} \, .
\end{align*}
However, the pre-factor $\frac1{v-c}$ in \eqref{eq:I decreasing} makes the estimate involving the corrector $h$ more difficult to handle with.

We proceed as follows to circumvent this issue. For $z>1$, we have 
\begin{align*}
\dfrac {d I_+}{dz} (z) 
&=  \left (  \int_{|v-c|<\delta} +  \int_{|v-c|>\delta}\right ) \dfrac {T_+(v-c)}{v-c} \left( I_+(z)  - T_+(v-c) f_+(z,v)  \right)\, d\nu(v) \\
& =  \int_{|v-c|<\delta} \dfrac {T_+(v-c)}{v-c} T_+(v-c) \left( f_+(z,c)  -  f_+(z,v)  \right)\, d\nu(v)  \\
&\quad+  \int_{|v-c|>\delta} \dfrac {T_+(v-c)}{v-c}\left( \kappa_+ e^{-\lambda_+ z}  - \kappa_+  T_+(v-c) F_+(v) e^{-\lambda_+ z}    \right)\, d\nu(v) \\ 
&\quad +  \int_{|v-c|>\delta} \dfrac {T_+(v-c)}{v-c} \left(\int  T_+(v'-c) h(z,v') \,d\nu(v') - T_+(v-c) h(z,v)\right)  \, d\nu(v) \\
& = \mathcal{I} +\mathcal{II}+\mathcal{III} \, .
\end{align*}
The three contributions are estimated as follows. From Corollary \ref{cor:exp lip}, we deduce 
\begin{align*}
 \mathcal{I}  
 &\leq
C e^{-\lambda_+ z}  \int_{|v-c|<\delta} \dfrac{T_+(v-c)^2}{|v-c|}|v-c|^{1/p'} \, d\nu(v) \\
& \leq  C e^{-\lambda_+ z} \omega_0(T)^{-1} \delta^{1/p'}
\end{align*}
On the other hand, we have
\begin{align*}
\mathcal{II} & =  \kappa_+ e^{-\lambda_+ z} \int_{|v-c|>\delta} \dfrac {T_+(v-c)}{v-c}  \frac{-\lambda_+(v-c)}{T_+(v-c) - \lambda_+(v-c)}\, d\nu(v) \\
& = - \lambda_+ \kappa_+ e^{-\lambda_+ z} \int_{|v-c|>\delta}     T_+(v-c)F_+(v)\, d\nu(v) \\
& =  - \lambda_+ \kappa_+ e^{-\lambda_+ z} \left(1 -  \int_{|v-c|<\delta}   T_+(v-c)F_+(v)\, d\nu(v)\right) \\
& \leq - \lambda_+ \kappa_+ e^{-\lambda_+ z} \left(1 -  \sup_{v\in V}\left(   T_+(v-c) F_+(v)\right) \delta \omega_0(T)^{-1} \right)
\end{align*}
Finally, we deduce from the $L^1$ bound \eqref{eq:L1 bound h} that 
\begin{equation*}
\mathcal{III}  \leq   C\delta^{-1} \int |h|\, d\nu  \leq  C\delta^{-1}e^{-\lambda_+ z -\vartheta_+' z } \,.
\end{equation*}
Collecting $\mathcal{I},\mathcal{II},\mathcal{III}$, we get eventually an estimate of the form ,
\beq\label{jguzgimv}
\dfrac {d I_+}{dz} (z) \leq  \left( - \lambda_+ \kappa_+ +  C \delta^{1/p'}   + C  \delta^{-1} e^{-\vartheta_+' z } \right)e^{-\lambda_+ z}\, .
\eeq
It remains to bound $\kappa_+$ uniformly from below in order to conclude. This is the purpose of the next auxiliary lemma. 

\begin{lemma}\label{eq:bound mu rho0}
Under the unit mass normalization  \eqref{eq:f unit mass}, the positive numbers $\kappa_\pm$ and  $\rho(0)$ are bounded below by some positive constant.  
\end{lemma}

\begin{proof}
We begin with the bound on $\rho(0)$, by using the conservation law \eqref{eq:def mu}, 
\begin{align*} (\forall z>0)\quad \int f_+(z,v) (v-c)^2 F_+(v)\, d\nu(v) &= e^{-\lambda_+ z}\int f(0,v) (v-c)^2 F_+(v)\, d\nu(v) \\ 
& \leq e^{-\lambda_+ z} \sup_{v\in V} \left((v-c)^2 F_+(v) \right) \rho(0)  
\, .
\end{align*}
Integrating this inequality with respect to $z>0$, we obtain
\begin{align}
\int_0^{+\infty}\int_{|v-c|>\delta} f_+(z,v) (v-c)^2 F_+(v)\, d\nu(v) dz & \leq \dfrac C{\lambda_+} \rho(0) \nonumber \\
\delta^2 \left( \inf_{v\in V} F_+(v)\right) \int_0^{+\infty}\int_{|v-c|>\delta} f_+(z,v)  \, d\nu(v)dz & \leq \dfrac C{\lambda_+} \rho(0)\nonumber \\
\delta^2 \left( \inf_{v\in V} F_+(v)\right) \left( \int_0^{+\infty}  \int  f_+(z,v)  \, d\nu(v)dz -  \int_0^{+\infty}  \int_{|v-c|<\delta} f_+(z,v)  \, d\nu(v)dz\right) & \leq \dfrac C{\lambda_+} \rho(0)\nonumber \\ 
\delta^2 \left( \inf_{v\in V} F_+(v)\right) \left( M_+ -  C \delta\omega_0(T)^{-1} \dfrac1{\lambda_+}  \right) & \leq \dfrac C{\lambda_+} \rho(0)\, ,\label{eq:bound below M+}\end{align}
where $M_+$ denotes the mass which is contained on the positive half space, namely
\[ M_+ = \int_0^{+\infty} \!\!\!\! \int  f_+(z,v)  \, d\nu(v)dz\, . \]
Note that in the last estimate, we have used the uniform exponential decay of $f$ \eqref{eq:589808}. The same estimate holds true on the negative half space,
\beq\label{eq:bound below M-}
\delta^2 \left( \inf_{v\in V} F_-(v)\right) \left( M_- -  C \delta\omega_0(T)^{-1}\dfrac1{\lambda_-}  \right) \leq \dfrac C{\lambda_-} \rho(0)\, .
\eeq
Combining \eqref{eq:bound below M+} and  \eqref{eq:bound below M-}, we obtain using the notation $M = M_+ + M_-$,
\begin{align}
\delta^2   M - C \delta^{3} &\leq C\rho(0) \nonumber\\
M^{3}  &\leq C\rho(0)\, . \label{eq:bound below rho}
\end{align}
Eventually, we get that $\rho(0)$ is bounded from below under the normalization  $ M = 1$ \eqref{eq:f unit mass}, and the conditions $\omega_0>0$, $c\in (c_*,c^*)$.

The estimate on $\kappa_+$ from below is obtained with the same lines of proof,
\begin{align*}
\kappa_+  & = \left(\int (v-c)^2 F_+(v)^2\, d\nu(v) \right)^{-1}  \int  f(0,v) (v-c)^2 F_+(v) \, d\nu(v)\\
\kappa_+  & \geq \left(\int (v-c)^2 F_+(v)^2\, d\nu(v) \right)^{-1}  \int_{|v-c|>\delta} f(0,v) (v-c)^2 F_+(v) \, d\nu(v)   \\
& \geq \frac1C  \delta^2 \left( \inf_{v\in V} F_+(v)\right)\left( \rho(0) - \int_{|v-c|<\delta} f(0,v)\, d\nu(v) \right) \\
& \geq \frac1C  \delta^2 \left( \inf F_+(v)\right)\left( \rho(0) -  C \rho(0) \delta\omega_0(T)^{-1}    \right) \\
& \geq \dfrac1C \rho(0)\, .
\end{align*}
This concludes the proof of Lemma \ref{eq:bound mu rho0}. 
\end{proof}
Back to \eqref{jguzgimv}, we can choose $\delta$ small enough, as compared to $\lambda_+\kappa_+$, so as to obtain
\begin{equation*}
\dfrac {d I_+}{dz} (z) \leq  \left( - \frac12\lambda_+ \kappa_+ + C   e^{-\vartheta_+' z } \right)e^{-\lambda_+ z}\, .
\end{equation*}
This concludes the proof of Lemma \ref{lem:asymptotic monotonicity}.
\end{proof}

Lemma \ref{lem:asymptotic monotonicity} has a counterpart for small $z$. Indeed, we expect from Section \ref{sec:Further regularity} that regularity of $I$ cannot be guaranteed up to $z = 0$ (see Proposition \ref{prop:rho lip}). The case of small $z$ requires particular attention. In the next lemma, we investigate quantitatively the behaviour of macroscopic quantities for small $z$. In fact, we prove that it may happen that the derivative of $\rho_+^+$ diverge as $z\searrow 0$, but with the appropriate sign. 

We face the following trivial alternative: either the derivative of  $\rho_+^+$ diverges, or it has a finite value, and we can push regularity up to the origin. This depends whether $c\in \supp \nu_\tau$ ($c\geq \vt$), or not. In order to make it quantitative, so as to manage the transition between the two alternatives, we distinguish between the three following cases: 
\begin{enumerate}
\item[(i)] $c\in \supp\nu_\tau$, \ie $\vt\leq c$,
\item[(ii)] $c\in (\vt - \delta_1,\vt)$, 
\item[(iii)] None of (i)-(ii), \ie $c\leq \vt -\delta_1$.
\end{enumerate}
Here, $\delta_1$ is a positive real number to be defined below.

\begin{lemma}\label{lem:zero monotonicity}
Let $T>0$. Assume (i) or (ii). There exist $C>0$ and $\eta_0>0$ such that 
\[ (\forall \tau\leq T)\quad 
\begin{cases} 
 (\forall z\in (0,\eta_0)) & \dfrac{d I_{\tau+}}{dz}(z) \leq -\dfrac1C \medskip\\ 
 (\forall z\in (-\eta_0,0)) & \dfrac{d I_{\tau-}}{dz}(z) \geq  \dfrac1C \end{cases}
\]
On the contrary, assume (iii). Then, $\frac{d I_{\tau+}}{dz}$ (resp.   $\frac{d I_{\tau-}}{dz}$) is Lipschitz continuous on $\R_+$ (resp. on $\R_-$), with quantitative bounds depending on $\delta_1$. 
\end{lemma}
\begin{proof}
The computation of the derivative of $\rho_+^+$ goes as follows,
\begin{align*}
\dfrac{d\rho_+^+}{dz}(z)
&= \int_{\{v>c\}} \dfrac 1{v-c}\left( I_+(z) - g_+^+(z,v)\right)\, d\nu(v) \nonumber \\
& = \int_{\{v>c\}}\dfrac 1{v-c} \left[ I_+(z) - T_+^+ \left( \int_0^{+\infty} I_-(-s(v-c)) \exp(-T_-^+ s)\, ds\right)\exp\left(-T_+^+\frac z{v-c}\right) \right. \nonumber\\
& \qquad \qquad\qquad\qquad\qquad\qquad \quad \left. - \int_0^{\frac z{v-c}}I_+(z - s(v-c)) T_+^+\exp(-T_+^+ s)\, ds\right]\, d\nu(v) 
\end{align*}
We deduce from global H\"older regularity of both $I_+$ and $I_-$ that, 
\begin{align*}
\dfrac{d\rho_+^+}{dz}(z)
& \leq  \int_{\{v>c\}}\dfrac 1{v-c} \left[ I_+(z)\exp\left(-T_+^+\frac z{v-c}\right)  - \dfrac{T_+^+}{T_-^+} I_-(0)\exp\left(-T_+^+\frac z{v-c}\right) \right.\\
& \quad \left . + C \left( \int_0^{+\infty} s^{1/p'}|v-c|^{1/p'} \exp(-T_-^+ s)\, ds\right)\exp\left(-T_+^+\frac z{v-c}\right) \right.\\
& \quad \left.  + C \left( \int_0^{\frac{z}{v-c}} s^{1/p'}|v-c|^{1/p'}  \exp(-T_+^+ s)\, ds\right)  \right]\, d\nu(v)\\
& \leq \int_{\{v>c\}}\dfrac 1{v-c} \left( I_+(0) - \dfrac{T_+^+}{T_-^+} I_-(0) + C|z|^{1/p'} \right)\exp\left(-T_+^+\frac z{v-c}\right)\, d\nu(v)  \\
& \quad  + C \int_{\{v>c\}}|v-c|^{1/p'-1}  \, d\nu(v)
\end{align*}
The constant that pops up in the leading order contribution has a sign. Interestingly, this sign is the signature of the confinement effect at the origin $z = 0$. 

\begin{lemma}\label{ugpmig}
The quantity $T_+^+I_-(0) - T_-^+ I_+(0)$ is  bounded from below by a positive constant, uniformly for $\tau \leq T$. 
\end{lemma}

\begin{proof}
First, we compute explicitly the value of this quantity, by means of $\rho$. From \eqref{bumivùozibipb} evaluated at $z = 0$, we deduce 
\begin{equation}
T_+^+I_-(0) - T_-^+ I_+(0) = \left( T_+^+ T_-^- - T_-^+T_+^- \right) \rho^- (0)
 =  4\chi_S \rho^- (0)\, .\label{eq:Delta I}
\end{equation}
Then, we seek some quantitative estimate of $\rho^- (0)$ from below. Recall the H\"older regularity of $I_+$. By definition, we have
\begin{align*}
\rho^-(0) &= \int_{\{v<c\}} \int_{0}^{+\infty} I_+(-s(v-c)) e^{-T_+^- s}\, ds d\nu(v) \nonumber\\
& \geq \int_{\{v<c\}} \int_{0}^{+\infty} \left( \left( I_+(0) - C s^{1/p'}|v-c|^{1/p'}\right)\vee 0\right)   e^{-T_+^- s}\, ds d\nu(v) \, , \quad I_+(0) \geq (T_+^- \wedge T_+^+)\rho(0)\nonumber\\
& \geq C \int_{\{v<c\}} \int_{0}^{+\infty} \left( \left( \rho(0) - C \dfrac{s^{1/p'}|v-c|^{1/p'}}{T_+^- \wedge T_+^+}\right)\vee 0\right)   e^{-T_+^- s}\, ds d\nu(v)\, . 
\end{align*}
Using the bound from below on $\rho(0)$ obtained via other arguments \eqref{eq:bound below rho}, we deduce the statement of Lemma \ref{ugpmig}. 
\end{proof}

We deduce that for $z$ small enough, 
we have 
\begin{align*}
\dfrac{d\rho_+^+}{dz}(z)& \leq -\frac{1}{C} \int_{\{v>c\}}\dfrac 1{v-c}  \exp\left(-T_+^+\frac z{v-c}\right)\, d\nu(v) + C \omega_0(T)^{-1}\\
&  \leq -\frac{\omega_0(T)}{C} \int_{v=c\vee \vt}^{\v0}\dfrac 1{v-c}  \exp\left(-T_+^+\frac z{v-c}\right)\, dv + C   \\
& \leq  -\frac{1}{C} \int_{T_+^+ \frac{z}{\v0-c}}^{T_+^+\frac{z}{(\vt-c)\vee 0}}\dfrac {e^{-u}}{u} \, du + C  
\end{align*}
Assume condition (i) \ie $(\vt-c)\vee 0 = 0$. Then, the last integral is equivalent to $C^{-1}\log z$ as $z\to 0$. Therefore, there exists $\eta_0$, and $C>0$, such that $\frac{d\rho_+^+}{dz}(z) \leq -C^{-1}$ for $z<\eta_0$.

Assume condition (ii) \ie  $(\vt-c)\vee 0 \leq \delta_1$. Then, for $z \leq (T^+_ +)^{-1}\delta_1$, we have
\begin{equation*}
\dfrac{d\rho_+^+}{dz}(z)  \leq  \frac{e^{-1}}C \log\left (\dfrac{\delta_1}{\v0-c} \right )  + C 
\end{equation*}
The latter inequality determines the choice of $\delta_1$, so that the contribution on the r.h.s. is controlled by some negative constant $-C^{-1}$. 

Assume condition (iii). Then, the {\em a priori} control $c\leq \vt -\delta_1$ enables to revisit regularity of $f$. Indeed, we get immediately from \eqref{eq:f I} that $f$ is Lipschitz continuous with respect to $z$, uniformly with respect to $\delta_1$. Differentiating with respect to $z$, for $z>0$ only (since $T$ is discontinuous with respect to $z$), we deduce that $\partial_z f_+$ is also Lipschitz continuous with respect to $z$. So is $\frac{dI_+}{dz}$. 

This concludes the proof of Lemma \ref{lem:zero monotonicity}.
\end{proof}

After the quantitative control of $I_+$ at both ends $0$ and $+\infty$, we can initiate the homotopy like argument.
Let introduce the set  
\begin{equation}\label{eq:T set}
\T = \left\{ \tau\geq 0 \;|\; (\forall z\neq0)\;  (\sign z)\dfrac{dI_\tau}{dz}(z) \leq 0  \right\} \, .
\end{equation} 
We shall prove that $\T$ is both open and closed in $\R_+$. As it is non-empty by Proposition \ref{prop:2 vel case}, it coincides with $\R_+$.

\begin{lemma}\label{lem:closed}
$\T$ is closed in $\R_+$. 
\end{lemma}
\begin{proof}
The stationary distribution $f_\tau(z,v)$ associated with the probability measure $\nu_\tau$  is unique in the space $L^2\left (a_\tau(z)d\nu_\tau(v)dz\right )$ (up to a constant multiplicative factor), where the weight $a_\tau(z)$ is defined by
\[\begin{cases}
a_\tau(z) = e^{\lambda_{\tau+} z} & \text{for $z>0$}\medskip\\
a_\tau(z) = e^{-\lambda_{\tau-} z} & \text{for $z<0$}
\end{cases}\]
This is a consequence of the linear and irreducible structure of the kinetic equation(\footnote{A much more refined analysis based on hypocoercive estimates after \cite{dolbeault_hypocoercivity_2015} is developed in \cite{calvez_confinement_2015} for the Cauchy problem, in the symmetric case $\chi_N = 0$ and $c = 0$. There, the goal is to capture exponential relaxation in time towards the stationary state. Of course, this ensures uniqueness as a  weak corollary.}). Let $(f_1,f_2)$ be two stationary distributions belonging to $L^2(a_\tau(z)d\nu_\tau(v)dz)$, and denote $h = f_2 - f_1$. By integration of the difference equation against $h/f_1$, we obtain successively,
\begin{multline*}
\iint (v-c) \partial_z h(z,v) \frac{h(z,v)}{f_1(z,v)}\,d\nu(v) dz  
\\ = \iiint \left (T(v'-c) h(z,v') - T(v-c) h(z,v)\right ) \frac{h(z,v)}{f_1(z,v)}\,d\nu(v')d\nu(v) dz \, ,
\end{multline*}
\begin{multline*}
- \frac12 \iint \left (\frac{h(z,v)}{f_1(z,v)}\right )^2 (v-c) \partial_z f_1(z,v) \\ = \iiint  T(v'-c) \frac{h(z,v')}{f_1(z,v')}\frac{h(z,v)}{f_1(z,v)}f_1(z,v') \,d\nu(v')d\nu(v) dz \\
- \iiint T(v-c)  \left(\frac{h(z,v)}{f_1(z,v)}\right)^2 f_1(z,v)\,d\nu(v')d\nu(v) dz \, .
\end{multline*}
Using equation \eqref{eq:f I} satisfied by the function $f_1$, we find,
\[ \frac12\iiint T(v'-c) \left( \dfrac{h(z,v')}{f_1(z,v')} - \dfrac{h(z,v)}{f_1(z,v)} \right)^2  f_1(z,v')\, d\nu(v')d\nu(v) dz = 0\, . \]

Hence, there exists a function of the variable $z$ only, say $b(z)$, such that $(\forall z\in \R)\; f_2(z,v) = b(z)f_1(z,v)\; \nu \; \as$. Plugging this identity into the stationary equation, we find $b'(z) = 0$. The unit mass normalization \eqref{eq:f unit mass} enables to fix this multiplicative factor to 1.

Continuity of $I_\tau$ with respect to $\tau$ is a consequence of this uniqueness property. Let $\tau_n\to \tau$. The sequence $f_{\tau_n}$ is uniformly bounded, uniformly H\"older continuous with respect to $z$, and H\"older continuous with respect to $v$, locally uniformly with respect to $z$. Moreover, it is uniformly exponentially small for large $|z|$ (see Proposition \ref{prop:holder}). By a diagonal extraction argument, there exists a subsequence $f_{\tau_{n'}}$ that converges towards some  function $f'$. Convergence is uniform  on compact sets of $\R^*\times V$. 

On the other hand, $\nu_{\tau_n}\to \nu_{\tau}$  in $L^p$ by the continuity assumption. Therefore, we can pass to the limit in the weak formulation \eqref{eq:weak formulation} 

We deduce from the uniqueness of the stationary distribution $f_\tau$ that $f' = f_\tau$. 

To conclude, it is sufficient to notice that $I_{\tau_{n'}+}$ (resp. $I_{\tau_{n'}-}$) is a sequence of non-increasing  functions (resp. non-decreasing functions) which converges pointwise towards $I_{\tau+}$ (resp. $I_{\tau-}$). Therefore, $I_{\tau+}$ is non-decreasing (resp. $I_{\tau-}$ is non-increasing).
\end{proof}

\begin{lemma}\label{lem:open}
$\T$ is open in $\R_+$. 
\end{lemma}
\begin{proof}
Some stability result in Lipschitz regularity is required to ensure that $\frac{dI_+}{dz}$ certainly remains uniformly negative, on compact sets of $\R_+^*$ (recall that we control the asymptotic behaviour as $z\to + \infty$).
 
Let $\tau_0\in \T$. Proposition \ref{lem:monotonicity} ensures that $\frac{dI_{\tau_0+}}{dz}<0$ for all $z> 0$, and $\frac{dI_{\tau_0-}}{dz}>0$ for all $z< 0$. Let $L>0$ large enough so that the quantities arising in Lemma \ref{lem:asymptotic monotonicity} have the appropriate sign for $\tau$ close to $\tau_0$ and $|z|>L$, \ie both 
\[-\dfrac1C e^{-\lambda_+ z} + C e^{-\lambda_+ z-\vartheta_+' z } < 0\quad \text{and}\quad   \dfrac1C e^{\lambda_- z} - Ce^{\lambda_- z+\vartheta_-' z} >0 \, . \]
Our claim is the following: there exists  $\eps_0>0$ and a neighbourhood $\V_0$ of $\tau_0$ such that 
\beq\label{eq:uniform}\begin{cases} 
(\forall \tau\in \V_0)\; (\forall z\in [0,L]) \quad &  \dfrac{dI_{\tau+}}{dz}(z)<-\eps_0\medskip\\
(\forall \tau\in \V_0)\; (\forall z\in [-L,0]) \quad   & \dfrac{dI_{\tau-}}{dz}(z)>\eps_0
\end{cases}\eeq
Otherwise, there would exist a sequence $\tau_n$ converging towards $\tau_0$ such that the maximum value  of $\frac{dI_{\tau_n+}}{dz}$ (resp. minimum value of $\frac{dI_{\tau_n-}}{dz}$) converges to a non-negative value. We restrict to $z>0$ for the sake of clarity. Let $z_n$ be a maximum point. 

To control the possible singularity at the origin, we distinguish between the two alternatives (i) or (ii), and (iii), as discussed in Lemma \ref{lem:zero monotonicity}. 

In cases (i) or (ii), we can separate $(z_n)$ from the origin, say $z_n>\eta_0/2$, as the derivative $\frac{dI_{\tau_n+}}{dz}$ is uniformly negative close to the origin. Note that the continuity of the support of $\nu_\tau$ is implicitly used in this argument. 

On the other hand, we learn from Corollary~\ref{cor:vyotflu} that $\frac{dI_{\tau_n+}}{dz}$ is uniformly Lipschitz continuous on $(\eta_0/2,L)$. By compactness of $(z_n)$, and equi-continuity of $\left (\frac{dI_{\tau_n+}}{dz}\right )$, we can pass to the limit along some subsequence, so as to obtain 
\begin{equation*}
\frac{dI_{\tau_0+}}{dz}(z^*) \geq  0\, , 
\end{equation*}
for some $z^*\in [\eta_0/2,L]$ which is a limit point of $(z_n)$. This is a contradiction. The same holds for $z<0$. 

In case (iii), regularity is fine up to $z=0$, so we can apply the same argument with $\eta_0 = 0$, including the origin.

Therefore, \eqref{eq:uniform} holds true, and $\T$ is  open.  \end{proof}

We are in position to conclude the proof of Theorem \ref{theo:mono}. Indeed, for all $\tau\geq 0$, the macroscopic  quantities $I_{\tau+}$ and $I_{\tau-}$ satisfy the appropriate monotonicity condition \eqref{eq:T set}. This monotonicity condition remains true in the limit $\tau\to +\infty$, as in the proof of Lemma \ref{lem:closed}. By approximation, it is true for all measure $\nu$ satisfying hypothesis $\Ac$.

\subsection{Improved monotonicity and quantitative regularity}\label{sec:improved monotonicity}

Several corollaries can be deduced from the previous monotonicity analysis. A first result improves monotonicity properties of the macroscopic densities $\rho_\pm^\pm$.

\begin{corollary}\label{cor:monotonicity++}
Let $g$ be the normalized density defined as in \eqref{eq:renormalized g},
\beq\label{eq:renormalized g bis}
\begin{cases}
(\forall z<0) & g_-(z,v) = f_-(z,v) e^{-\lambda_-z} \medskip\\
(\forall z>0) & g_+(z,v) = f_+(z,v) e^{\lambda_+z} \end{cases}
\eeq
Let $\varrho_\pm^-(z) = \int_{\{v<c\}} g_\pm(z,v)\, d\nu(v)$ and $\varrho_\pm^+(z) = \int_{\{v>c\}} g_\pm(z,v)\, d\nu(v)$. Then $\varrho_\pm^\pm$ has the same monotonicity as $\rho_\pm^\pm$:
\[ \begin{cases}  
 \dfrac{d\varrho_+^+}{dz}(z) < 0\quad \text{and}\quad \dfrac{d\varrho_+^-}{dz}(z) < 0 & \text{for $z\geq 0$}  \medskip\\
\dfrac{d\varrho_-^+}{dz}(z) > 0 \quad \text{and}\quad\dfrac{d\varrho_-^-}{dz}(z) > 0 & \text{for $z\leq0$} 
\end{cases}   \] 
\end{corollary}

\begin{proof}
We introduce a pair of exponents $\lambda_l\in (0,\lambda_-)$, and $\lambda_r\in (0,\lambda_+)$. We shall prove that the monotonicity property holds true for all such pair  $(\lambda_l,\lambda_r)$
By increasing $(\lambda_l,\lambda_r)$ up  to $(\lambda_-,\lambda_+)$ we will obtain Corollary \ref{cor:monotonicity++} as a limiting result.

With some abuse of notation, we still denote by $g$ the density normalized with exponents $\lambda_l,\lambda_r$ as in \eqref{eq:renormalized g bis}. It satisfies the following equation,
\[
\begin{cases}
(\forall z<0) & (v-c) \partial_z g_-(z,v) = J_-(z) - (T_-(v-c) + \lambda_l(v-c))g_-(z,v) \medskip\\
(\forall z>0) & (v-c) \partial_z g_+(z,v) = J_+(z) - (T_+(v-c) - \lambda_r(v-c))g_+(z,v)  
\end{cases}
\]
where $J_-$ satisfies the following identity,
\begin{align*} 
J_-(z) & = \int T_-(v-c) g_-(z,v)\, d\nu(v) \\
& = \int \left( T_-(v-c) + \lambda_l (v-c)\right) g_-(z,v)\, d\nu(v)\\
& = \int \bg_-(z,v)\, ,\quad \text{where}\quad \bg_-(z,v) = \left( T_-(v-c) + \lambda_l (v-c)\right) g_-(z,v)\,,
\end{align*}
There, we have used the zero-flux property of the stationary state \eqref{eq:fluxes}. The same holds for $J_+$, and $\bg_+(z,v) = \left( T_+(v-c) - \lambda_r (v-c)\right) g_+(z,v)$. 
We introduce the new tumbling rates
\[ R_-(v-c) = T_-(v-c) + \lambda_l (v-c)\, , \quad
R_+(v-c) = T_+(v-c) - \lambda_r (v-c)\, .  \]

We sketch briefly how to adapt the three steps in the proof of Theorem \ref{theo:mono}. 
\medskip

\noindent\textbf{Step  \#1. Initialization (the easy case).} Notice that $J_+(z) = I_+(z) e^{\lambda_rz}$.  We repeat computation \eqref{eq:I+ decreasing}:
\begin{align*}
 \dfrac {d J_+}{dz}(z)  & = \dfrac {d I_+}{dz}(z) e^{\lambda_rz} + \lambda_r I_+(z) e^{\lambda_rz} \\
& = e^{\lambda_rz} \int \left(\la \frac {T_+(v-c)}{v-c}\ra T_+(v-c) - \dfrac{\left (T_+(v-c)\right )^2}{v-c} + \zeta_+ (v-c) + \lambda_r T_+(v-c) \right) f_+(z,v)\, d\nu(v)\, , 
\end{align*}
To conclude as in Section \ref{ssec:monotonicity}, we seek $\zeta_+$ satisfying the following pair of conditions, instead of \eqref{eq:bmupigmli}:
\[
\begin{cases}
\displaystyle -\frac12 \frac{(T_+^+)^2}{\v0 - c}-\frac12 \frac{T_+^- T_+^+}{\v0 + c} + \zeta_+(\v0-c) + \lambda_r T_+^+ < 0 \medskip\\
\displaystyle \frac12 \frac{T_+^+ T_+^-}{\v0 - c}+\frac12 \frac{(T_+^-) ^2}{\v0 + c} - \zeta_+(\v0+c) + \lambda_r T_+^-  < 0
\end{cases}
\]
Existence of $\zeta_+\in \R$  is equivalent to the following inequality,
\begin{align} 
\frac12 \frac{T_+^+ T_+^-}{(\v0 - c)(\v0+c)}+\frac12 \frac{(T_+^-) ^2}{(\v0 + c) ^2} + \lambda_r \dfrac{T_+^-}{\v0+c} & < \frac12 \frac{(T_+^+)^2}{(\v0 - c)^2}+\frac12 \frac{T_+^- T_+^+}{(\v0 + c)(\v0 - c)} - \lambda_r \dfrac{T_+^+}{\v0-c} \nonumber\\
\Leftrightarrow\quad  \frac12 \frac{T_+^-}{\v0 + c}\left( \frac{T_+^-}{\v0 + c} + 2\lambda_r \right ) 
&<  
 \frac12 \frac{T_+^+}{\v0 - c} \left ( \frac{T_+^+}{\v0 - c} - 2\lambda_r \right )
\, . \label{uigmiuvmiug} 
\end{align} 
Recall the equation for the right side exponent $\lambda_+$ \eqref{eq:dispersion lambda+}, when the measure $\nu$ is a symmetric combination of two Dirac masses \eqref{eq:2vel}: 
\begin{equation*}
\frac12 \dfrac{\v0-c}{T_+^+ - \lambda_+ (\v0 -c)} + \frac12 \dfrac{-\v0-c}{T_+^- - \lambda_+ (-\v0 -c)} = 0\, .
\end{equation*}
We deduce that 
\begin{equation*}
2\lambda_+ = \dfrac{T_+^+}{\v0 - c} - \dfrac{T_+^-}{\v0 + c}\, .
\end{equation*} 
Plugging this identity into \eqref{uigmiuvmiug}, we obtain the equivalent formulation,
\begin{align*} 
\frac12 \frac{T_+^-}{\v0 + c}\left( \dfrac{T_+^+}{\v0 - c} + 2(\lambda_r- \lambda_+) \right ) 
&<  
 \frac12 \frac{T_+^+}{\v0 - c} \left ( \frac{T_+^-}{\v0 + c} + 2(\lambda_+ - \lambda_r) \right )
\nonumber\\
\Leftrightarrow\quad  \frac{T_+^-}{\v0 + c}\left( \lambda_r- \lambda_+ \right ) 
&<  
   \frac{T_+^+}{\v0 - c} \left ( \lambda_+ - \lambda_r \right )
\, . 
\end{align*}
The last condition clearly holds true, as $\lambda_r < \lambda_+$ by definition. \medskip 

\noindent\textbf{Step  \#2. Enhancement of monotonicity.}
Now, assume {\em a priori} that $J_+$ is non-increasing, and that $J_-$ is non-decreasing.  

The shapes of the velocity profiles of $\bg$ are the same as for $\bbf$.
\begin{enumerate}[(i)]
\item For $z>0$ and $v<c$, $\bg_+^-$ is non-decreasing with respect to $v$. Moreover, we have $\bg_+^-(z,v) \leq J_+(z)$. Indeed, we have
\beq\label{eq:duhamel bg+} \bg_+^-(z,v) = \int_0^\infty J_+(z - s(v-c)) R_+^-(v-c) \exp\left( - R_+^-(v-c) s\right) \leq J_+(z)\, .  \eeq
Thus, $\varrho_+^-$ is decreasing. 
A similar statement holds true for $\bg_-^+$ ($z<0,v>c$), and $\varrho_-^+$. 
\item For $z>0$ and $v>c$, $\bg^+_+$ is decreasing with respect to $v$ on the set $\{ \bg_+^+ \leq J_+ \}\cap\{v>c\}$. Indeed we adapt the two ingredients of the proof of Lemma \ref{lem:qualitative profile}. Firstly, $\bg_+^+ - J_+$ satisfies a damped transport equation with a non-negative source term,
\[  (v-c) \partial_z \left( \bg_+^+ - J_+\right)  + R_+^+\left( \bg_+^+ - J_+ \right) = -(v-c)\dfrac{dJ_+}{dz} \geq 0 \,  . \]
Secondly, introduce the auxiliary function $h(z) = \partial_v \bg_+^+(z,v)$. It satisfies the following damped transport equation with source term,
\begin{align*} 
(v-c) \partial_z h(z) + R_+^+(v-c) h(z) 
&=  - \partial_z  \bg_+^+(z,v) - \lambda_r \left( J_+(z) - \bg_+^+(z,v)\right) \\
&= - \left(  \dfrac1{v - c} R_+^+ \left( J_+(z) - \bg_+^+(z,v)\right) + \lambda_r \left( J_+(z) - \bg_+^+(z,v)\right) \right)\\
& = \dfrac{T_+^+}{v - c} \left( \bg_+^+(z,v) - J_+(z) \right)\, .
\end{align*}
The same conclusion as in Lemma \ref{lem:qualitative profile} holds true. Therefore, the same compensations as in the proof of Proposition \ref{lem:monotonicity} are working to make $\varrho_+^+$ decreasing too.  
\end{enumerate}
\medskip

\noindent\textbf{Step  \#3. Propagation of monotonicity.} 
In order to conclude, it is sufficient to check that $J_+$ is uniformly decreasing as for large $z$, as in Lemma \ref{lem:asymptotic monotonicity}. Again, the asymptotic behaviour of $g_+$ is involved. Roughly speaking, we have for $z>0$,
\begin{align*}
(v-c)\partial_z g_+(z,v) = J_+(z) - g_+(z,v) &\sim \kappa_+ e^{-(\lambda_+ - \lambda_r)z} \left( 1 - F_+(v)^{-1} \left(T_+(v-c) - \lambda_r (v-c)\right)\right)\\
& \sim \kappa_+ e^{-(\lambda_+ - \lambda_r)z} \dfrac{(\lambda_r - \lambda_+)(v-c)}{F_+(v)}\, .
\end{align*}
So, $\partial_z g_+$ is decreasing with respect to $z$ for large $z$, provided $\lambda_r < \lambda_+$. This asymptotic result can be made quantitative and rigorous, as in Lemma \ref{lem:asymptotic monotonicity}.

On the other hand, the possible singular behaviour close to the origin, as in Lemma \ref{lem:zero monotonicity}, is unaffected by the exponential normalization.

With similar results as Proposition \ref{lem:monotonicity}, Lemma \ref{lem:asymptotic monotonicity} and \ref{lem:zero monotonicity} at hand, we can propagate monotonicity as in Lemma \ref{lem:closed} and \ref{lem:open}. 
\end{proof}

The following statement is an immediate consequence of Corollary \ref{cor:monotonicity++}. 

\begin{corollary}\label{cor:bounds rho++}
For $z>0$, we have
\[
\begin{cases} 
& \kappa_+ e^{-\lambda_+ z} \leq I_+(z) \leq I_+(0)e^{-\lambda_+ z}\medskip\\
&\displaystyle \kappa_+\left(\int F_+(v)\, d\nu(v)\right)e^{-\lambda_+ z} \leq \rho_+(z) \leq \rho(0) e^{-\lambda_+ z} 
\end{cases}
\]
A similar statement holds true for $z<0$.
\end{corollary}

Corollary \ref{cor:monotonicity++} also enables to derive quantitative regularity estimates on macroscopic quantities. 
\[ J_\pm(z) = I_\pm(z) e^{\pm \lambda_\pm z}\, . \]

\begin{corollary}\label{cor:improved holder regularity}
The averaged quantity $J_+$ is $1/p'-$H\"older continuous with some explicit constant $L(R,\omega,p)$ which is linear with respect to $\rho$ (see \eqref{eq:holder constant} below). 
\end{corollary}

\begin{proof}
We restrict to the case $z>0$. We argue as in the proof of Proposition \ref{prop:holder} \eqref{eq:I_+ holder continuous}. Let $0<z_1<z_2$. We have
\begin{align*} 
0\leq J_+(z_1) - J_+(z_2) 
&= \int R_+(v-c)\left( g_+(z_1,v) -g_+(z_2,v)\right)  \, d\nu(v) \\
& \leq \int_{\{|v-c|<\delta\}} R_+(v-c)\left( g_+(z_1,v) -g_+(z_2,v)\right) \, d\nu(v) \\
& \quad  + \int_{\{|v-c|>\delta\}} \dfrac{R_+(v-c)}{|v-c|} \sup_{z'\in(z_1,z_2)} \left| J_+(z') - R_+(v-c)g_+(z',v)\right||z_1-z_2|\, d\nu(v) \\
&\leq \int_{\{0<c-v<\delta\}} R^-_+(v-c)   g_+^-(z_1,v)  \, d\nu(v) + \int_{\{0<v-c<\delta\}} R^+_+(v-c)   g_+^+(z_1,v)  \, d\nu(v) 
\\& \quad +\int_{\{c-v>\delta\}} \dfrac{R_+^-(v-c) }{|v-c|}  \left(\sup_{z'\in(z_1,z_2)}  J_+(z') \right)|z_1-z_2|\, d\nu(v) \\
& \quad + \int_{\{v-c>\delta\}} \dfrac{R_+^+(v-c) }{|v-c|} \left( \sup_{z'\in(z_1,z_2)} J_+(z')\vee R^+_+(v-c)  g_+(z',v)\right) |z_1-z_2| \, d\nu(v)\,.
\end{align*}
We deduce from the Duhamel representation formula \eqref{eq:duhamel bg+} that
\[ (\forall v<c)\quad R_+^-(v-c) g_+^-(z,v) \leq J_+(z) \leq J_+(0) \, .\]
On the other hand, for $v>c$, we have 
\begin{align*}  
R_+^+(v-c) g_+^+(z,v) & = R_+^+(v-c) \left( \int_0^\infty J_-(-s(v-c)) \exp\left( - R_-^+(v-c) s \right)\, ds \right)\exp\left( - \dfrac {R_+^+(v-c)z}{v-c} \right) \\
& \quad + R_+^+(v-c) \int_0^{\frac z{v-c}} J_+(z-s(v-c))\exp\left( - R_+^+(v-c) s \right)\, ds \\
& \leq \dfrac{R_+^+(v-c)}{R_-^+(v-c)}J_-(0)\exp\left( - \dfrac{R_+^+(v-c) z}{v-c} \right) + J_+(0) \left( 1 - \exp\left( - \dfrac{R_+^+(v-c) z}{v-c} \right)\right)\\
& \leq  \left(\dfrac{R_+^+(v-c)}{R_-^+(v-c)}J_-(0)\right)\vee J_+(0) \, .
\end{align*}
We deduce:
\begin{align*} 
0\leq J_+(z_1) - J_+(z_2)&\leq \left(  J_+(0)   +   \left(\left(\sup\dfrac{R_+^+}{R_-^+}\right)J_-(0)\right)\vee J_+(0)\right) \delta^{1/p'}\|\omega\|_p \\ 
& \quad + \left(\sup R_+^-  \right) J_+(0)  A(p) \delta^{-1/p}\|\omega\|_p  \\ 
& \quad +  \left(\sup R_+^+\right)  \left(\left(\sup \dfrac{ R_+^+ }{R_-^+}\right)J_-(0)\vee  J_+(0)  \right) A(p) \delta^{-1/p}\|\omega\|_p |z_1-z_2| \, .
\end{align*}
where the constant $A$ is defined as $A(p) = (p-1)^{-1/p'}$. Finally, we use the trivial relations $J_+(0)   \leq (\sup R_+) \rho(0)$ and $J_-(0) \leq (\sup R_-) \rho(0)$, in order to obtain the following estimate which is linear with respect to $\rho(0)$,
\begin{multline*} 
0\leq J_+(z_1) - J_+(z_2) \leq    \left(  (\sup R_+)   +   \left(\left(\sup\dfrac{R_+^+}{R_-^+}\right)(\sup R_-)\right)\vee (\sup R_+)\right) \delta^{1/p'}\|\omega\|_p\rho(0) \\ 
\quad + \left( \left(\sup R_+^-  \right)(\sup R_+)   +   \left(\sup R_+^+\right)  \left(\left(\sup \dfrac{ R_+^+ }{R_-^+}\right)(\sup R_-)\vee (\sup R_+)  \right)\right) A(p) \delta^{-1/p}\|\omega\|_p |z_1-z_2|\rho(0)  \, .
\end{multline*}
By optimizing with respect to $\delta$, we obtain the following quantitative H\"older estimate,
\beq\label{eq:improved reg}
|J_+(z_1) - J_+(z_2)| \leq L(R,\omega,p)\rho(0) |z_1 - z_2|^{1/p'}\, ,
\eeq 
where
\begin{multline}\label{eq:holder constant}
L(R,\omega,p) = p (p-1)^{-1/{(p'p)}}\left(  (\sup R_+)   +   \left(\left(\sup\dfrac{R_+^+}{R_-^+}\right)(\sup R_-)\right)\vee (\sup R_+)\right)^{1/p}\times \\
\left( \left(\sup R_+^-  \right)(\sup R_+)   +   \left(\sup R_+^+\right)  \left(\left(\sup \dfrac{ R_+^+ }{R_-^+}\right)(\sup R_-)\vee (\sup R_+)  \right)\right)^{1/p'}
 \|\omega\|_p \, .
\end{multline}
Note that the pre-factor is increasing, and satisfies $p (p-1)^{-1/{(p'p)}} \sim p$ for large $p$. This rules out the Lipschitz case $p = \infty$ ($p'=1$), as expected from \cite{golse_regularity_1988}.
\end{proof}

\section{The case without nutrient $\chi_N=0$ (stationary cluster)}

The existence of a stationary state $(\rho,S)$ is an immediate consequence of the monotonicity of $\rho$ established in Section \ref{sec:monotonicity}, and the following general statement about solutions of the elliptic problem on $S$. The purpose of this section is to deal with $c = 0$ only. However, we formulate our statement for any $c$ to anticipate the coupling with the nutrient $N$ in the forthcoming section.

Here, we set $\beta =  1$, without loss of generality.

\begin{proposition}\label{prop:monotonicity S}
Assume that the function $\rho\in L^1$ is locally Lipschitz continuous on $\R^*$, and that it is increasing for $z<0$, and decreasing for $z>0$. Let $S$ be the unique solution of the following elliptic problem
\[ (\forall z\in \R) \quad -c\partial_z S(z) - D_S \partial^2_z S(z) + \alpha S(z) = \rho(z)\, . \]
Then $\partial_z S$ vanishes once, and only once.
\end{proposition}

\begin{proof}
Let denote by $P$ the derivative of $S$: $P = \partial_z S$. It belongs to $W^{2,\infty}$, locally uniformly on $\R^*$. It satisfies the following equation,
\[ -c\partial_z P(z) - D_S\partial^2_z P(z) + \alpha P(z) = \partial_z \rho(z)\, , \]
together with the following asymptotic limits,
\beq \lim_{z\to -\infty} P(z) = 0^+ \, , \quad  \lim_{z\to +\infty} P(z) = 0^-\, , \label{eq:P asymptotic} \eeq
where the notation $0^+$ means that we approach $0$ from above, and $0^-$ means that we approach $0$ from below.  

Assume by contradiction that $P$ vanishes at least twice. Then, we deduce from \eqref{eq:P asymptotic} that it must vanish at least three times, including a possible double root, where $\partial_z P$ vanishes also. Two vanishing points are necessarily on the same side, say for $z\leq 0$. Accordingly, there must exist a locally minimal point $z_0\leq 0$ with non positive value, $P(z_0) \leq 0$,  $\partial_z P(z_0) = 0$, and $\partial^2_z P(z_0)\geq 0$. If $z_0<0$, or $P(z_0)<0$, we have
\[ \alpha P(z_0) = c \partial_z P(z_0) + D_S \partial^2_z P(z_0) + \partial_z \rho(z_0) > c \partial_z P(z_0) + D_S \partial^2_z P(z_0) \geq 0\, . \]
This is a contradiction. If $z_0 = 0$, and $P(z_0) = 0$ then $0 $ is necessary a double root. We can repeat the same reasoning on the right side: there must exist $z_0'>0$ such that $P(z_0') \geq 0$,  $\partial_z P(z_0') = 0$, and $\partial^2_z P(z_0')\leq 0$. At this point, we have 
\[ \alpha P(z_0') = c \partial_z P(z_0') + D_S \partial^2_z P(z_0') + \partial_z \rho(z_0') < c \partial_z P(z_0') + D_S \partial^2_z P(z_0') \leq 0\, . \]
Again, this is a contradiction.
We argue similarly when we are to choose $z_0$ initially on the right side. 
\end{proof}

In the case $\chi_N = 0$, and $c = 0$, the derivative $\partial_z S(z)$ can only vanish at $z = 0$ for symmetry reason. We deduce that the {\em a priori} hypothesis \eqref{eq:ansatz} is indeed correct {\em a posteriori}. Therefore there exists a stationary state $(f,S)$. This concludes the proof of Theorem \ref{theo:TW stat}.

\section{Coupling with chemical concentrations $S,N$ (travelling wave)}


\subsection{Matching the condition on $S$}
\label{sec:matching S}

As in the case without nutrient, we deduce from Proposition \ref{prop:monotonicity S} that $\partial_zS$ changes sign only once. However it might not happen at $z=0$, as required in \eqref{eq:ansatz}. 

The objective of this section consists in varying $c$ such as to satisfy $\partial_z S(0) = 0$ (similarly as in the proof of Theorem \ref{theo:TW macro} in the macroscopic case). 

This requires two intermediate results: continuity of $\partial_z S(0)$ as a function of $c$, and analysis of the extremal behaviours as $c\to c_*$, resp. as $c\to c^*$. 

We introduce the following notation for  $\partial_z S(0)$, as a function of parameter $c$,
\[ \Upsilon(c) =  \partial_z S(0)\, . \]
Recall that $\Upsilon$ is given by the following integral representation formula, 
\beq \Upsilon(c)  =  - \int_{-\infty}^0 \mu_+  e^{\mu_+  z} \rho_-(z)\, dz + \int_0^{+\infty} \mu_-  e^{-\mu_-  z} \rho_+(z)\, dz\, , \label{eq:dxS0 pre}\eeq
where the reaction-diffusion exponents $\mu_{\pm}(c)$ are defined as
\beq\label{eq:exposant mu}\begin{cases} \mu_-(c) = \dfrac{-c + \sqrt{c^2 + 4\alpha D_S}}{2D_S} >0 \medskip\\
\mu_+(c) = \dfrac{c + \sqrt{c^2 + 4\alpha D_S}}{2D_S} >0\end{cases}  \eeq

\begin{proposition}
Under assumption $\Ac$, the function $\Upsilon$ is continuous on $(c_*,c^*)$. 
\end{proposition}

\begin{proof}
Let $c_n\to c\in (c_*,c^*)$. As in the proof of Lemma \ref{lem:closed}, we can extract a subsequence $(f_{n'})_{n'}$ converging towards $f$ uniformly on compact sets of $\R^* \times V$. First, we pass to the limit in the weak formulation
\begin{multline}\label{eq:weak limit n'} - \iint (v - c_{n'}) f_{n'}(z,v) \partial_z \varphi\, d\nu(v) dx \\ = \iiint T(z,v'-c_{n'}) f_{n'}(z,v') \left( \varphi(z,v) - \varphi(z,v')\right)\, d\nu(v')d\nu(v)dz \, .\end{multline}
The l.h.s. can be handled easily as it can be written
\begin{multline*} - \iint (v - c_{n'}) f_{n'}(z,v) \partial_z \varphi\, d\nu(v) dx \\ = - \iint (v - c) f_{n'}(z,v) \partial_z \varphi\, d\nu(v) dx + (c_{n'}-c) \iint  f_{n'}(z,v) \partial_z \varphi\, d\nu(v) dx\, . \end{multline*}
The last contribution is bounded by $|c-c_{n'}|\|f_{n'}\|_{L^1}\|\partial_x \varphi\|_\infty$. The r.h.s. in \eqref{eq:weak limit n'} must be treated with caution because $T$ is not continuous with respect to $c$. However, we can split the integral into $\{ |v'-c|<\delta \}$, and $\{ |v'-c|>\delta \}$. The former is controlled by $\delta^{1/p'}$, the $L^\infty$ bound of $f_{n'}$, and the $L^p$ bound of $\nu$. We can pass to the limit in the latter since $T(z,v'-c_{n'}) = T(z,v'-c)$ provided $|c - c_{n'}|<\delta$.

By uniqueness of the limit problem, we have $f_{n'} \to f$, where $f$ is the density profile associated with $c$. 

Finally, we can pass to the weak limit in the integral formula for $\Upsilon$ \eqref{eq:dxS0 pre}. 
\end{proof}

\begin{proposition}
Under assumption $\Ac$, there are explicit conditions on the reaction diffusion parameters $(\alpha, D_S)$ such that the following extremal behaviours hold true,
\[
\begin{cases}
&\displaystyle\lim_{c\searrow c_*} \Upsilon(c) >0\medskip\\
&\displaystyle\lim_{c\nearrow c^*} \Upsilon(c) <0
\end{cases}
\]
\end{proposition}

The idea is quite simple: as $c\to c_*$, the macroscopic profile $\rho$ becomes flat on the right side (see Figure \ref{fig:flatness}). As a result, we expect that the maximum of $S$ is pushed to the right of the origin, \ie $\Upsilon(c) >0$. Similar behaviour can be expected as $c\to c^*$, but on the right side. This is made as  quantitative  as possible in the following proof. 

\begin{proof}
We will use crucially the improved monotonicity obtained in Section \ref{sec:improved monotonicity}, and particularly Corollary \ref{cor:bounds rho++}. 

Firstly, we consider the case $c\to c_*$. We have 
\begin{align}
\Upsilon(c) &\geq - \rho(0) \int_{-\infty}^0 \mu_+(c) e^{\mu_+(c) z} e^{\lambda_-(c)}\, dz + \kappa_+ \la F_+\ra \int_0^{+\infty} \mu_-(c) e^{-\mu_-(c) z} e^{-\lambda_+(c)z}\, dz \nonumber\\
& \geq - \rho(0) \dfrac{\mu_+(c)}{\mu_+(c) + \lambda_-(c)} + \kappa_+ \la F_+\ra\dfrac{\mu_-(c)}{\mu_-(c) + \lambda_+(c)}\, , \label{eq:bound below upsilon}
\end{align}
where $\kappa_+$ is defined as in \eqref{eq:mu+-}. In order to estimate $\kappa_+$ from below, we use the improved regularity obtained in Corollary \ref{cor:improved holder regularity} \eqref{eq:improved reg}. For all $v>c$, we have
\begin{align*}
f(0,v) &= \int_0^{+\infty} I_-(-s(v-c)) e^{-T_-^+s} \, ds\\
&  = \int_0^{+\infty} J_-(-s(v-c)) e^{-(T_-^+ + \lambda_-(v-c))s} \, ds \\
& \geq \int_0^{+\infty} \left(\left(J_-(0) - \rho(0) \left[J_-\right ]_{1/p'} s^{1/p'} |v-c|^{1/p'}  \right)\vee 0\right) e^{-(T_-^+ + \lambda_-(v-c))s} \, ds \\
& \geq \rho(0) \left[J_-\right ]_{1/p'} |v-c|^{1/p'} \int_0^{\tau_-}\left( \tau_-(v-c)^{1/p'} - s^{1/p'}\right) e^{-(T_-^+ + \lambda_-(v-c))s} \, ds \, , 
\end{align*}
where the stopping time $\tau_-$ is given by
\begin{equation*}
\tau_-(v-c) = \dfrac{\left( \inf R_-\right)^{p'}}{\left[J_-\right ]_{1/p'}^{p'}|v-c|}\, .
\end{equation*}
Similarly, for $(v<c)$ we find
\begin{align*}
f(0,v) &= \int_0^{+\infty} I_+(-s(v-c)) e^{-T_+^-s} \, ds\\
&  = \int_0^{+\infty} J_+(-s(v-c)) e^{-(T_+^- - \lambda_+(v-c))s} \, ds \\
& \geq \int_0^{+\infty} \left(\left(J_+(0) - \rho(0)\left[J_+\right ]_{1/p'} s^{1/p'} |v-c|^{1/p'}  \right)\vee 0\right) e^{-(T_+^- - \lambda_+(v-c))s} \, ds \\
& \geq \rho(0) \left[J_+\right ]_{1/p'} |v-c|^{1/p'} \int_0^{\tau_+(v-c)}\left( \tau_+(v-c)^{1/p'} - s^{1/p'}\right) e^{-(T_+^- - \lambda_+(v-c))s} \, ds \, , 
\end{align*}
where the stopping time $\tau_+$ is given by
\begin{equation*}
\tau_+(v-c) = \dfrac{\left( \inf R_+\right)^{p'}}{\left[J_+\right ]_{1/p'}^{p'}|v-c|}\, .
\end{equation*}
Combining these two estimates, we get the following bound from below,
\begin{align*}
\dfrac{\kappa_+}{\rho(0)} & = \dfrac1{\rho(0)}\dfrac{\int f(0,v) (v-c)^2 F_+(v)\, d\nu(v)}{\int (v-c)^2 F_+(v)^2\, d\nu(v)} \geq  \left( \int (v-c)^2 F_+(v)^2\, d\nu(v)\right)^{-1} \times \\ 
& \quad \left[  \int_{\{v>c\}} \left[J_-\right]_{1/p'} \left(\int_0^{\tau_-}\left( \tau_-^{1/p'} - s^{1/p'}\right) e^{-(T_-^+ + \lambda_-(v-c))s} \, ds\right)  |v-c|^{2+1/p'} F_+(v)\, d\nu(v) \right. \\
& \quad \left. + \int_{\{v<c\}} \left[J_+\right ]_{1/p'} \left(\int_0^{\tau_+}\left( \tau_+^{1/p'} - s^{1/p'}\right) e^{-(T_+^- - \lambda_+(v-c))s} \, ds\right)  |v-c|^{2+1/p'} F_+(v)\, d\nu(v)    \right]
\end{align*}
Plugging this estimate into \eqref{eq:bound below upsilon}, we deduce that there is an explicit constant $A_+$ such that 
\[ \Upsilon(c) \geq \rho(0) \left( -  \dfrac{\mu_+(c)}{\mu_+(c) + \lambda_-(c)} +  A_+ \la F_+\ra \dfrac{\mu_-(c)}{\mu_-(c) + \lambda_+(c)}  \right)\, . \]
As $c\to c_*$, we have $\lambda_+(c)\to 0$, and $F_+\to 1/T_+$, therefore, the condition $\Upsilon(c^*)>0$ is a consequence of the following condition:
\begin{multline}\label{eq:condition c_*}
\dfrac{c_* + \sqrt{c_*^2 + 4\alpha D_S}}{c_* + \sqrt{c_*^2 + 4\alpha D_S} + 2D_S \lambda_-(c_*)} \leq \la\left(\dfrac{v-c_*}{T_+(v-c_*)}\right)^2\ra^{-1} \la \dfrac1{T_+(v-c_*)}\ra \times\\ \left[  \int_{\{v>c_*\}} \left[J_-\right]_{1/p'} \left(\int_0^{\tau_-}\left( \tau_-^{1/p'} - s^{1/p'}\right) e^{-(T_-^+ + \lambda_-(c_*)(v-c_*))s} \, ds\right)  |v-c_*|^{2+1/p'} \dfrac1{T_+^+ }\, d\nu(v) \right. \\
  \left. + \int_{\{v<c_*\}} \left[J_+\right ]_{1/p'} \left(\int_0^{\tau_+}\left( \tau_+^{1/p'} - s^{1/p'}\right) e^{-T_+^- s} \, ds\right)  |v-c_*|^{2+1/p'} \dfrac1{T_+^- }\, d\nu(v)    \right] \, . 
\end{multline}
As the r.h.s. does not depend on the reaction-diffusion parameters $(\alpha,D_S)$, whereas the l.h.s. vanishes as $D_S\to +\infty$, condition \eqref{eq:condition c_*} is clearly not empty. 

Secondly, we consider the case $c\to c^*$. There, we perform similar estimates, but the other way around,
\[ \Upsilon(c) \leq - \kappa_- \la F_-\ra \dfrac{\mu_+(c)}{\mu_+(c) + \lambda_-(c)} +  \rho(0) \dfrac{\mu_-(c)}{\mu_-(c) + \lambda_+(c)}\, , \]
where $\kappa_-$ is defined as in \eqref{eq:mu+-}. As previously, the ratio $\kappa_-/\rho(0)$ can be estimated from below by using the regularity of $f(0,v)$ with respect to velocity, through the regularity of $J_\pm$ with respect to space. As a consequence, we get for $c\to c^*$,
\[ \Upsilon(c^*) \leq \rho(0) \left( - A_- \la F_-\ra  +   \dfrac{\mu_-(c^*)}{\mu_-(c^*) + \lambda_+(c^*)}\right)\, . \]
The latter is negative under some condition which reads,
\begin{multline}\label{eq:condition c^*}
\dfrac{-c^* + \sqrt{(c^*)^2 + 4\alpha D_S}}{-c^* + \sqrt{(c^*)^2 + 4\alpha D_S} + 2D_S \lambda_+(c^*)}
\leq 
\la   \left(\dfrac{v-c^*}{T_-(v-c^*)}\right)^2\ra^{-1} \la \dfrac1{T_-(v-c^*)}\ra \times\\ \left[  \int_{\{v>c^*\}} \left[J_-\right]_{1/p'} \left(\int_0^{\tau_-}\left( \tau_-^{1/p'} - s^{1/p'}\right) e^{-T_-^+ s} \, ds\right)  |v-c^*|^{2+1/p'} \dfrac1{T_-^+}\, d\nu(v) \right. \\
  \left. + \int_{\{v<c^*\}} \left[J_+\right ]_{1/p'} \left(\int_0^{\tau_+}\left( \tau_+^{1/p'} - s^{1/p'}\right) e^{-( T_+^- - \lambda_+(c^*)(v-c^*)) s} \, ds\right)  |v-c^*|^{2+1/p'} \dfrac1{T_-^-}\, d\nu(v)    \right] \, . 
\end{multline}
We notice that the l.h.s. vanishes when $\alpha\to 0$, or $D_S\to +\infty$ (independently). This guarantees that the condition is not empty as well. 
\end{proof}

\subsection{Matching the condition on $N$}
\label{sec:matching N}

The monotonicity condition on $N$ to be satisfied \eqref{eq:ansatz} is not very restrictive. It only requires $c>0$ as established in Proposition \ref{prop:N} below. This is automatically satisfied when $c_* \geq 0$, which is equivalent to $\chi_N\geq \chi_S$ \eqref{knùioùoib}, conditioned on \eqref{eq:condition c_*}-\eqref{eq:condition c^*}. 

On the contrary, if $\chi_N< \chi_S$, then we should impose that $\Upsilon(0)>0$ in order to guarantee the existence of $c\in (0,c^*)$ such that $\Upsilon(c) = 0$, conditioned on \eqref{eq:condition c^*}. Arguing as previously, a sufficient condition writes
\[ 
A_+ \la F_+\ra \dfrac{\mu_-(0)}{\mu_-(0) + \lambda_+(0)} -  \dfrac{\mu_+(0)}{\mu_+(0) + \lambda_-(0)} \geq 0\, , 
\] 
which is equivalent to
\begin{multline}
\label{eq:condition c=0}
\dfrac{ \sqrt{\alpha/ D_S} +   \lambda_+(0)}{\sqrt{\alpha / D_S} +  \lambda_-(0)} \leq \la v^2  F_+(v)^2\ra^{-1} \la  F_+(v) \ra \times\\ \left[  \int_{\{v>0\}} \left[J_-\right]_{1/p'} \left(\int_0^{\tau_-}\left( \tau_-^{1/p'} - s^{1/p'}\right) e^{-(T_-^+ + \lambda_-(0)v)s} \, ds\right)  |v|^{2+1/p'} F_+(v) \, d\nu(v) \right. \\
  \left. + \int_{\{v<0\}} \left[J_+\right ]_{1/p'} \left(\int_0^{\tau_+}\left( \tau_+^{1/p'} - s^{1/p'}\right) e^{-(T_+^- - \lambda_+(0)v) s} \, ds\right)  |v|^{2+1/p'} F_+(v)\, d\nu(v)    \right] \, . 
\end{multline}
To see that this condition is not empty, it is enough to consider that $\lambda_+(0)\to 0$ as $\chi_N \nearrow \chi_S $. In a second step, the limit $\alpha/D_S\to 0$ enables to realize the condition. However, this is not entirely satisfactory, as it imposes some strong condition on the parameters of the run-and-tumble process, $(\chi_S,\chi_N)$. 

We conclude this section by checking that, indeed, if $c> 0$, then $N$ is increasing. 

\begin{proposition}\label{prop:N}
Assume that $c> 0$, and that $\rho$ is exponentially bounded on both sides, namely
\begin{equation*}
\begin{cases}
(\forall z<0) \quad & \rho_-(z) \leq C e^{\lambda_- z}\medskip\\
(\forall z>0) \quad & \rho_+(z) \leq C e^{-\lambda_+ z}
\end{cases}
\end{equation*}
Then, there exist two constants $N_-<N_+$, and a solution $N$ of the elliptic problem in \eqref{eq:TW}, such that 
\[
\begin{cases}
\displaystyle\lim_{z\to -\infty} N(z) = N_-\medskip\\
\displaystyle\lim_{z\to +\infty} N(z) = N_+
\end{cases}
\]
Moreover, we have $\partial_z N>0$.  
\end{proposition}

\begin{proof}
We introduce $u(z)=  \partial_z \log N(z)$. Elliptic equation for $N$ is equivalent to  the following first order equation for $u$,
\[-  c u(z) = D_N \left( \partial_z  u(z) + | u(z)|^2\right) - \gamma \rho(z)\, .\]
This rewrites as a non autonomous, non linear ODE in the $z$ variable, 
\beq \label{eq:ODE u} \partial_z u(z) = -\dfrac c{D_N}u(z) - | u(z)|^2 + \dfrac\gamma{D_N}\rho(z)\, , \eeq
together with the boundary conditions $\lim_{z\to \pm\infty} u(z) = 0$. This means that $u$ is the unique homoclinic orbit that leaves the origin $u = 0$ as $z\to -\infty$, and gets back to the origin $u=0$ as $z\to +\infty$. A way to construct this solution is to consider the family of Cauchy problems on $(a,+\infty)$,
\beq\label{eq:cauchy ua} \begin{cases}
&\partial_z u_a(z) = -\dfrac c{D_N}u_a(z) - | u_a(z)|^2 + \dfrac\gamma{D_N}\rho(z) \medskip\\
&u_a(a) = 0\, ,
\end{cases}
\eeq
and to take the limit as $a\to -\infty$. Indeed, we deduce immediately from the structure of \eqref{eq:cauchy ua} that $(\forall z>a)\;  u_a(z)>0$. In addition, we deduce from the Gronwall lemma that 
\[ u_a(z) \leq \dfrac\gamma{D_N}\int_{a}^z e^{c/D_N (y-z)} \rho(y)\, dy\, . \]
As a consequence, for $z<0$ we deduce from Corollary \ref{cor:bounds rho++} that 
\[ u_a(z) \leq  \dfrac\gamma{D_N\lambda_- + c}C e^{\lambda_- z}\, .\]
On the other hand, for $z>0$, we deduce similarly that
\begin{align*} 
u_a(z) & \leq \dfrac\gamma{D_N\lambda_- + c}C e^{-c/D_N z} + \dfrac\gamma{D_N\lambda_+ - c} C\left(  e^{-c/D_N z} - e^{-\lambda_+ z} \right)  \\
& \leq C e^{-c/D_N z} + \begin{cases}
C \max \left(e^{-c/D_N z},e^{-\lambda_+ z} \right) & \text{if}\quad \lambda_+\neq c/D_N\medskip \\
C z e^{-c/D_N z} & \text{if}\quad \lambda_+= c/D_N
\end{cases} 
\end{align*}
In any case, we can extract a subsequence $u_{a_n}$ such that $u_{a_n}$ converges uniformly towards $u$ solution of \eqref{eq:ODE u} with the boundary conditions $\lim_{z\to \pm\infty} u(z) = 0$. 

Previous estimates guarantee that $u$ is integrable on both sides, so it enables to define properly $N$ as in Proposition \ref{prop:N}. 
\end{proof}

This concludes the proof of Theorem \ref{theo:kin TW}.

\section{Focus on the discrete velocity case}
\label{sec:discrete velocity}

We revisit the travelling wave problem \eqref{eq:TW} in the discrete velocity case, under Assumption $\Ad$. In this case, the auxiliary function $\Upsilon(c)$ is not continuous with respect to $c$, but it has jumps each time $c$ crosses some discrete velocity $v_i$. By analysing carefully the sign of the jumps, we are able to exhibit some set of parameters for which there is no uniqueness of the travelling wave. Interestingly, we can also find some set of parameters for which there is no travelling wave satisfying the natural ansatz \eqref{eq:ansatz}. 

The careful analysis of the discrete velocity case is based on the spectral decomposition of kinetic operators in Case's elementary modes (see \cite{kaper_spectral_1982,gosse_computing_2013} and references therein). 

The extremal speeds $c_*, c^*$ are defined as in Section \ref{sec:decoupled}, see \eqref{eq:def c_*}-\eqref{eq:def c^*}. There, the $L^p$ boundedness of the velocity measure $\nu$ is used. However, the auxiliary function $R$ is still decreasing, and Lipschitz continuous in the discrete velocity case \Ad. The following computation resolves the case when $c$ crosses one of the discrete velocities $v_j$. There, we have for $\eps,\eps'>0$ sufficiently small,
\begin{equation*} 
R(v_j + \eps) - R(v_j-\eps') = - \sum_{i\neq j}  \omega_i \dfrac{\eps + \eps'}{T_+(v_i - v_j)}  + \omega_j \left( \dfrac{-\eps}{T_+^-} - \dfrac{\eps'}{T_+^+}\right) <0 \\
\end{equation*}
Notice that the second contribution does not vanish here, on the contrary to the continuous velocity case \Ac \eqref{jmbujvb}, but it has the good sign to come up with the same conclusion since $T_+>0$. 
 
From now on, we suppose that $c\in (c_*,c^*)$, and that $c\notin V = \{v_i\}_{1\leq i\leq N}$. 

\subsection{Spectral decomposition of kinetic transport operators}\label{sec:spectral}

We develop in this section quantitative spectral analysis of the discrete velocity case \Ad. Recall that we are seeking solutions of the linear equation
\beq\label{eq:kin discrete}
(\forall z)\quad (\forall i)\quad (v_i-c) \partial_z f^i(z) =  \sum_{j = 1}^N\omega_j  T(z,v_j-c)f^j(z)  - T(z,v_i-c)f^i(z)\, .
\eeq

We decompose the solution on each side (resp. $z<0$ and $z>0$) along special Case's solutions, which have the form 
\[ (z<0)\quad \mathfrak{f}^k_-(z,v) = e^{\lambda^k_{-} z}F^k_{-}(v)\,, \quad \text{or}\quad (z>0) \quad \mathfrak{f}^k_+(z,v) = e^{-\lambda^k_{+} z}F^k_{+}(v)\, .\] 
The total number of modes is $N$. They are distributed on each side according to the position of $c$ relatively to the discrete velocities $(v_i)_{1\leq i\leq N}$. 
\begin{definition}
Assume that $c$ is such that 
\beq \label{eq:c ordonne} v_1 < \dots < v_K < c < v_{K+1} < \dots < v_N\, . \eeq
\begin{enumerate}[(i)]
\item There are $K$ negative modes defined for $z<0$, $\mathfrak{f}^k_-(z,v)$, where $F^k_{-}$, and $\lambda^k_{-}$ are given by the following expressions,
\[
(\forall 1\leq k\leq K)\quad (\forall i)\quad F^k_-(v_i) = \dfrac{1}{T_-(v_i - c) + \lambda^k_-(v_i - c)}\, , 
\]
Each exponent $\lambda^k_->0$ is a root of the following dispersion relation, 
\beq
\sum_{i = 1}^N \omega_i \dfrac{ (v_i - c)}{T_-(v_i - c) + \lambda^k_-(v_i - c)} = 0\, . \label{eq:dispersion k-}
\eeq
\item Similarly, there are $N-K$ negative modes defined for $z>0$, $\mathfrak{f}^k_+(z,v)$, where $F^k_{+}$, and $\lambda^k_{+}$ are given by the following expressions,
\[ 
(\forall K+1\leq k\leq N)\quad (\forall i)\quad F^k_+(v_i) = \dfrac{1}{T_+(v_i - c) - \lambda^k_+(v_i - c)}\, . 
\]
Each exponent $\lambda^k_+>0$ is a root of the following dispersion relation, 
\beq
\sum_{i = 1}^N \omega_i \dfrac{(v_i - c)}{T_+(v_i - c) - \lambda^k_+(v_i - c)} = 0\, . \label{eq:dispersion k+}
\eeq
\end{enumerate}
\end{definition}

The following couple of  propositions investigates the relation between the number of roots of the dispersion relations \eqref{eq:dispersion k-}-\eqref{eq:dispersion k+}, and the degrees of freedom on the kinetic transport problems to be solved on each side. This discussion justifies {\em a posteriori} notations in the previous definition. 

\begin{proposition}\label{prop:number roots}
\begin{enumerate}[(i)]
\item There are exactly $K$ distinct positive roots $(\lambda_-^k)_{1\leq k\leq K}$ of the dispersion relation \eqref{eq:dispersion k-}. They are localized as follows,
\beq \label{eq:asymptot -}
0 < \lambda^{1}_- < \dfrac{T_-^-}{c-v_1} < \lambda^{2}_- <  \dfrac{T_-^-}{c-v_{2}} < \dots < \lambda^{K}_- < \dfrac{T_-^-}{c-v_K} \, .
\eeq
Any root $\lambda_- ^k$ $(1\leq k\leq K)$ is decreasing with respect to $c$ under condition \eqref{eq:c ordonne}.
\item There are exactly $N-K$ distinct positive roots $(\lambda_+^k)_{K+1\leq k\leq N}$ of the dispersion relation \eqref{eq:dispersion k+}. They are localized as follows,
\beq \label{eq:asymptot +}
0 < \lambda^{N}_+ < \dfrac{T_+^+}{v_{N}-c} < \lambda^{N-1}_+ <  \dfrac{T_+^+}{v_{N-1}-c} < \dots < \lambda^{K+1}_+ < \dfrac{T_+^+}{v_{K+1}-c} \, .
\eeq
Any root $\lambda_+ ^k$ $(K+1\leq k\leq N-K)$ is increasing with respect to $c$ under condition \eqref{eq:c ordonne}.
\end{enumerate}
\end{proposition}
\begin{proof}
Let $Q_-(\lambda)$ be defined as in \eqref{eq:dispersion k-} for $\lambda>0$,
\[
Q_-(\lambda) = \sum_{i = 1}^N \omega_i \dfrac{ 1}{\dfrac{T_-(v_i - c)}{(v_i - c)} + \lambda} = 0\, .
\]
It is decreasing with respect to $\lambda$ on each interval of definition. Moreover, we have 
$Q_-(0) > 0$ since $c<c^*$ and $Q_-(\lambda)\to 0$ as $\lambda\to+\infty$. We argue analogously for $z>0$, by introducing $Q_+$ as follows
\[
Q_+(\lambda) = \sum_{i = 1}^N \omega_i \dfrac{ 1}{\dfrac{T_+(v_i - c)}{(v_i - c)} - \lambda} = 0\, .
\]

Monotonicity of $\lambda_\pm$ is a simple consequence of the implicit dispersion relations \eqref{eq:dispersion k-}-\eqref{eq:dispersion k+}, namely we have
\beq\label{eq:monotonicity lambda} \begin{cases} 
\displaystyle \dfrac{d\lambda_-}{dc}(c) = \left( \sum_{i = 1}^N \omega_i \dfrac{- (v_i-c)^2}{(T_-(v_i-c) + \lambda_-(c)(v_i-c))^2}  \right)^{-1}\sum_{i = 1}^N \omega_i \dfrac{T_-(v_i-c)}{(T_-(v_i-c) + \lambda_-(c)(v_i-c))^2}  < 0 \, , \medskip\\
\displaystyle \dfrac{d\lambda_+}{dc}(c) = \left( \sum_{i = 1}^N \omega_i  \dfrac{(v_i-c)^2}{(T_+(v_i-c) - \lambda_+(c)(v_i-c))^2}  \right)^{-1}\sum_{i = 1}^N \omega_i \dfrac{T_+(v_i-c)}{(T_+(v_i-c) - \lambda_+(c)(v_i-c))^2}   >0 \, .
\end{cases}
\eeq
\end{proof}

\begin{proposition}
\begin{enumerate}[(i)]
\item Bounded solutions of the kinetic stationary problem \eqref{eq:kin discrete} restricted to $z<0$ have exactly  $K$ degrees of freedom. Any such solution can be decomposed uniquely as
\beq\label{eq:Case1}
(\forall z<0) \quad f_-(z,v) = \sum_{k = 1}^K a_k  \mathfrak{f}^k_-(z,v)\, .
\eeq
\item Bounded solutions of the kinetic stationary problem \eqref{eq:kin discrete} restricted to $z>0$ have exactly  $N-K$ degrees of freedom. Any such solution can be decomposed uniquely as
\beq\label{eq:Case2}
(\forall z>0) \quad f_+(z,v) = \sum_{k = K+1}^{N} b_k  \mathfrak{f}^k_+(z,v)\, .
\eeq
\end{enumerate}
\end{proposition}

\begin{proof}
The number of degrees of freedom on $R_-$ is determined by the number of negative velocities (relatively to $c$), here $K$. Indeed, analysis in Section \ref{sec:decoupled} reveals that bounded solutions are entirely determined by the incoming velocity profile, corresponding to $v<c$. A similar argument holds true for $z>0$. 
\end{proof}

In order to reconstruct an entire solution on $\R\times V$, we shall solve a transfer eigenvalue problem. The continuity of $f$ at $z = 0$, $f_-(0,v) = f_+(0,v)$, writes equivalently as follows, 
\beq\label{eq:transfer relation}
(\forall i)\quad \sum_{k = 1}^K a_k F_-^k(v_i) = \sum_{k = K+1}^N b_k F_+^k(v_i)\, .
\eeq
This writes in matrix block form as follows
\[
\begin{pmatrix}
-F_- & F_+ 
\end{pmatrix}
\begin{pmatrix}
a \\ b 
\end{pmatrix} = \begin{pmatrix}
0 \\ 0 
\end{pmatrix}\, .
\]
We realize immediately that a non-trivial solution $(a,b)$ exists, since the square matrix  $\begin{pmatrix}
-F_- & F_+ 
\end{pmatrix}$ has a left eigenvector associated with eigenvalue zero:
\[
\left(\omega (v-c)\right)\begin{pmatrix}
-F_- & F_+ 
\end{pmatrix} = \begin{pmatrix} - \sum_{i=1}^N \omega_i (v_i-c) F_-(v_i) & \sum_{i=1}^N \omega_i (v_i-c) F_+(v_i) \end{pmatrix} = \begin{pmatrix}
0 & 0 
\end{pmatrix} \, ,
\] 
where we have used both dispersion relations \eqref{eq:dispersion k-}-\eqref{eq:dispersion k+}. 
We can characterize $(a,b)$ in a unique way, by prescribing a unit total mass \eqref{eq:f unit mass},
\[
\iint f(z,v)\, d\nu(v) dz = \sum_{k = 1}^K \dfrac{a_k}{\lambda_-^k} \la F_-^{k} \ra + \sum_{k = K+1}^N \dfrac{b_k}{\lambda_+^k} \la F_+^{k} \ra = 1\, .
\]

Monotonicity of both $\rho_+$ and $\rho_-$ can be deduced from the results obtained in the continuous velocity case \Ac, by approximation of the discrete measure $\nu$ by a sequence of absolutely continuous measures. Indeed, under the crucial assumption $(\forall i) \; c\neq v_i$, regularity of the approximating sequence is guaranteed, provided it is chosen so as to avoid $c$ in its support. We are not aware of any direct argument based on the decomposition in Case's modes \eqref{eq:Case1}-\eqref{eq:Case2}.  

\subsection{Exchange of modes as $c$ crosses some discrete velocity}\label{sec:exchange of modes}
We investigate carefully the case where $c$ crosses one of the velocities $(v_i)$, say $v_{K+1}$, with increasing values. According to Section \ref{sec:spectral}, there is  a swap: one mode will disappear on the right side $(z>0)$, this is mode $\mathfrak{f}_+^{K+1}$, and one mode will appear on the left side, the new mode $\mathfrak{f}_-^{K+1}$. The other modes will be essentially unchanged. During the swap, we expect a singular transition around the origin, as the largest exponents become arbitrarily large on both sides. 

\begin{figure} 
\begin{center}
\includegraphics[width = 0.48\linewidth]{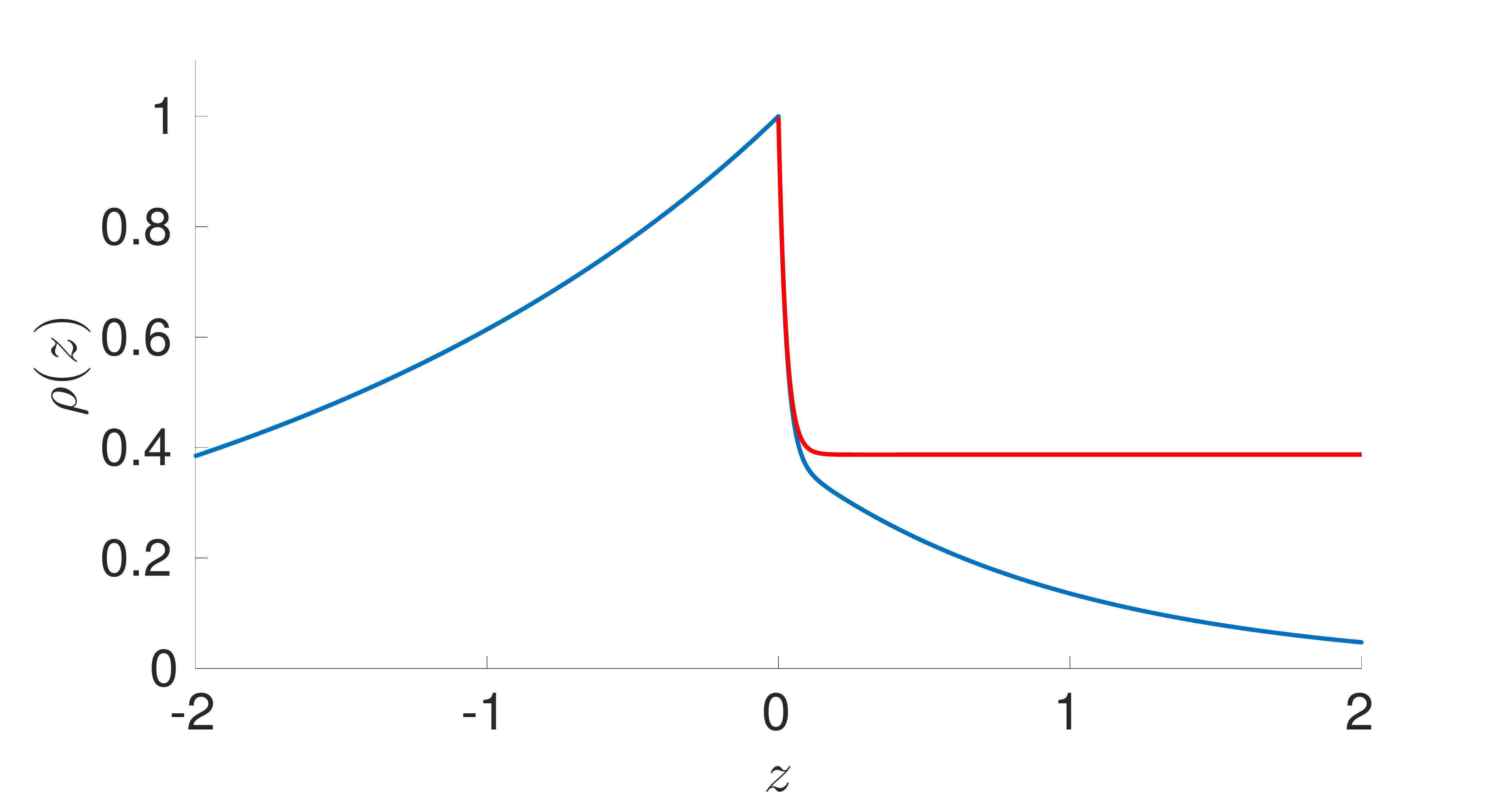}\; 
\includegraphics[width = 0.48\linewidth]{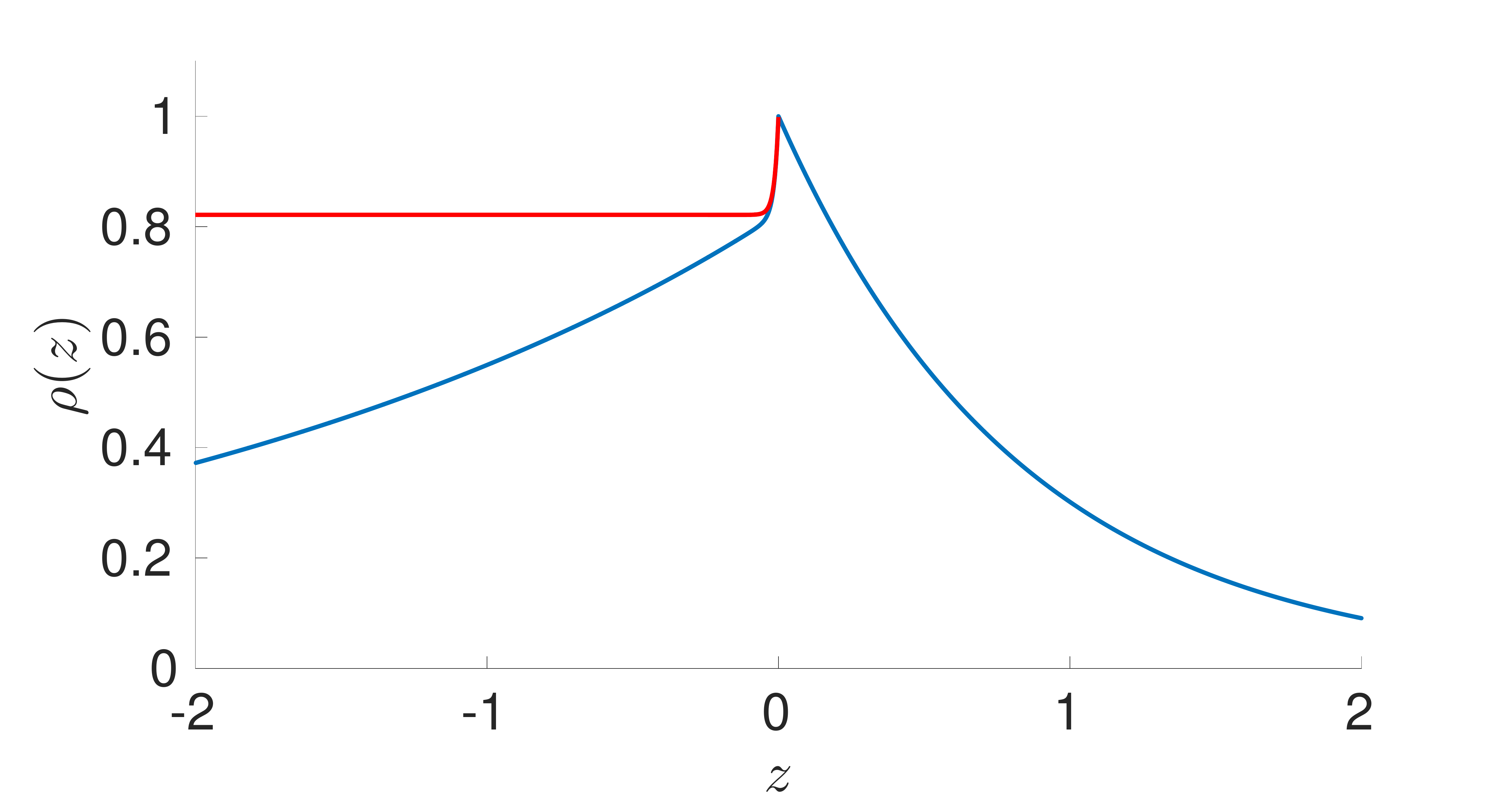}
\caption{\small Illustration of the swapping behaviour when $c$ crosses one of the discrete velocities, say $v_{K+1}$. The spatial density profile $\rho$ (blue line) experiences a strong shape discontinuity. (Left) As $c \nearrow v_{K+1}$, the profile becomes arbitrarily narrow on the right side (see Proposition \ref{prop:swap 1}). Moreover, during the  exchange of modes,   the amplitude of the mode that disappears on the right is bounded below explicitly (red line), see Proposition \ref{prop:swap 2}. (Right) The same behaviour occurs on the right side as $c \searrow v_{K+1}$. This shape discontinuity has dramatic consequences on the determination of travelling waves (Section \ref{sec:jumps})}
\label{fig:swap}
\end{center}
\end{figure}

\begin{proposition}\label{prop:swap 1}
Assume $c$ crosses $v_{K+1}$ with increasing values. Then, for $c<v_{K+1}$, we have
\[ \lim_{c \nearrow v_{K+1}} \lambda_+^{K+1}(c) = +\infty\, . \]
On the other hand, for $c>v_{K+1}$, we have
\[ \lim_{c \searrow v_{K+1}} \lambda_-^{K+1}(c) = +\infty\, . \]
The other roots have finite limits, denoted by $\lambda_\pm^k(v_{K+1})$ ($k\neq K+1$), which are roots of the function $\check{Q}_{\pm}$, obtained from $Q_\pm$ by removing the term $i = K+1$ from the summation, and evaluating at $c = v_{K+1}$. 
\end{proposition}

\begin{proof}
Limits exist by monotonicity of $\lambda^k_\pm$. For any $k\neq K+1$, the limit is clearly finite by \eqref{eq:asymptot -}-\eqref{eq:asymptot +}. Assume for instance that $\lambda_+^{K+1}(c)$ has finite limit $\bar{\lambda}_+^{K+1}$ as $c \nearrow v_{K+1}$. Then, it should satisfy the relation 
\beq \label{eq:count contradiction}\sum_{i \neq K+1 } \omega_i \dfrac{ 1}{\dfrac{T_+(v_i - v_{K+1})}{(v_i - v_{K+1})} - \bar{\lambda}_+^{K+1}} = 0\, ,\eeq
and so should the limits $\lambda^k_+(v_{K+1})$ for $K+2\leq k \leq N$. However, we deduce from Proposition \ref{prop:number roots} that relation \eqref{eq:count contradiction} has only $N-K-1$ roots. Moreover, it cannot have a root larger than $\frac{T_+^+}{v_{K+2}-v_{K+1}}$. This is a contradiction. We argue similarly for $c \searrow v_{K+1}$.
\end{proof}

The next result concerns the limit of the weights $(a_k,b_k)$ as $c\to v_{K+1}$. More precisely, we aim at describing the higher order contribution in the following decomposition of the spatial density in exponential modes (given here for $c<v_{K+1}$),
\beq\label{eq:rho Case} 
\begin{cases}
(\forall z>0)\quad &\displaystyle\rho_+(z) = \sum_{k = K+1}^N b_k e^{-\lambda_+^k z} \la F_+^k \ra\, , \medskip\\
(\forall z<0)\quad &\displaystyle\rho_-(z) = \sum_{k = 1}^K a_k e^{\lambda_-^k z} \la F_-^k \ra\, ,
\end{cases} 
\eeq
where we have opted for the notation $\la F  \ra = \sum_{i = 1}^N \omega_i F(v_i)$.  
Interestingly, the contribution of the higher mode, here mode $K+1$, has a non trivial contribution as $c\nearrow v_{K+1}$. The same behaviour happens as $c\searrow v_{K+1}$.  This emphasizes the dramatic discontinuity in the shape of the spatial profile $\rho$ as $c$ crosses the discrete velocity $v_{K+1}$, as illustrated in Figure \ref{fig:swap}.   
\begin{proposition}\label{prop:swap 2}
The weights $(a_k)_{1\leq k\leq K},(b_k)_{K+2\leq k\leq N}$ have finite limits $(\underline{a_k})_{1\leq k\leq K},(\underline{b_k})_{K+2\leq k\leq N}$ as $c \nearrow v_{K+1}$, resp. $(\overline{a_k})_{1\leq k\leq K},(\overline{b_k})_{K+2\leq k\leq N}$ as $c \searrow v_{K+1}$. The limits coincide up to a constant factor.  In addition, $b_{K+1} \la F_+^{K+1}\ra$ has a finite limit which satisfies,
\beq\label{eq:lim bF}  \lim_{c \nearrow v_{K+1}}  b_{K+1} \la F_+^{K+1}\ra > 0\, . \eeq
Similarly, $a_{K+1} \la F_-^{K+1}\ra$ has a finite limit which satisfies,
\[  \lim_{c \searrow v_{K+1}} a_{K+1}\la F_-^{K+1}\ra > 0\, . \]
\end{proposition}
\begin{proof}
\noindent\textbf{Step  \#1.} Firstly, we prove the existence of limits.
For that purpose, we pass to the limit in the transfer matrix $\begin{pmatrix}
-F_- & F_+ 
\end{pmatrix}$. We have
\[
\begin{pmatrix}
-F_- & F_+ 
\end{pmatrix} = \begin{pmatrix}
-F_-^1(v_1) &  \dots & - F_-^K(v_1) & F_+^{K+1}(v_1) & F_+^{K+2}(v_1) & \dots & F_+^{N}(v_1) \\
\vdots & \ddots & \vdots & \vdots & \vdots & \ddots & \vdots \\
-F_-^1(v_{K+1}) &  \dots & - F_-^K(v_{K+1}) & F_+^{K+1}(v_{K+1}) & F_+^{K+2}(v_{K+1}) & \dots & F_+^{N}(v_{K+1}) \\
\vdots & \ddots & \vdots & \vdots & \vdots & \ddots & \vdots \\
-F_-^1(v_N) &  \dots & - F_-^K(v_N) & F_+^{K+1}(v_N) & F_+^{K+2}(v_N) & \dots & F_+^{N}(v_N) 
\end{pmatrix}\, ,
\]
therefore, 
\beq \label{eq:limit pmatrix}
\lim_{c\nearrow v_{K+1}} \begin{pmatrix}
-F_- & F_+ 
\end{pmatrix} = 
\begin{pmatrix}
-F_-^1(v_1) &  \dots & - F_-^K(v_1) & 0 & F_+^{K+2}(v_1) & \dots & F_+^{N}(v_1) \\
\vdots & \ddots & \vdots & \vdots & \vdots & \ddots & \vdots \\
- \dfrac{1}{T_-^+} &  \dots & - \dfrac{1}{T_-^+} & ? & \dfrac{1}{T_+^+} & \dots & \dfrac{1}{T_+^+} \\
\vdots & \ddots & \vdots & \vdots &\vdots & \ddots & \vdots \\
-F_-^1(v_N) &  \dots & - F_-^K(v_N) & 0 & F_+^{K+2}(v_N)  & \dots & F_+^{N}(v_N) 
\end{pmatrix}\, ,
\eeq
where the question mark $?$ accounts for some undetermined limit.
The left eigenvector associated with eigenvalue zero becomes \begin{multline*} 
\omega(v-v_{K+1}) \\ = \begin{pmatrix}\omega_1 (v_1 - v_{K+1}) & \dots & \omega_K (v_K - v_{K+1}) & 0 & \omega_{K+2} (v_{K+2} - v_{K+1})& \dots& \omega_{N} (v_{N} - v_{K+1})\end{pmatrix}\, .
\end{multline*}
We deduce that the matrix $\begin{pmatrix}
-F_- & F_+ 
\end{pmatrix}_{(i\neq K+1, k\neq K+1)}$
obtained from \eqref{eq:limit pmatrix} by eliminating row $K+1$ and column $K+1$ also possesses an eigenvector $((\underline{a})_{1\leq k\leq K},(\underline{b})_{K+2\leq k\leq N})^T$. Finally, the missing term $b_{K+1}$ can be recovered using the identity \eqref{eq:transfer relation} for $i = K+1$. In the limit $c\nearrow v_{K+1}$, it becomes
\[ \sum_{k = 1}^K \underline{a_k} \frac{1}{T_-^+} =   \lim_{c \nearrow v_{K+1}} \left( b_{K+1} F_+^{K+1}(v_{K+1})\right) + \sum_{k = K+2}^N \underline{b_k}\frac{1}{T_+^+} \, .   \]  

We argue in a similar way in the case $c \searrow v_{K+1}$. Interestingly, the matrix $\begin{pmatrix}
-F_- & F_+ 
\end{pmatrix}_{(i\neq K+1, k\neq K+1)}$
is the same as when taking the limit from below $c \nearrow v_{K+1}$. Therefore, the weights coincide up to a constant factor.\medskip 

\noindent\textbf{Step  \#2.}  
Secondly, we prove that the limit \eqref{eq:lim bF} is positive. We consider the case where $c\nearrow v_{K+1}$. We argue by computing the derivative of $\rho(z)$ at $z = 0^+$ in two different manners. Differentiating \eqref{eq:rho Case} with respect to $z$, we obtain, 
\beq\label{mugmuifgmiu} \dfrac{d\rho_+}{dz}(0^+) = - \sum_{k = K+1}^N \lambda_+^k b_k \la F_+^k \ra\, . \eeq 
On the one hand, we notice that $\la F_+^{K+1} \ra \sim \omega_{K+1} F_+^{K+1}(v_{K+1})$ as $c\nearrow v_{K+1}$. Indeed, we have
\begin{multline*} (\forall i\neq K+1) \quad \dfrac{F_+^{K+1}(v_i)}{F_+^{K+1}(v_{K+1})} = \dfrac{T_+(v_{K+1} - c) - \lambda_+^{K+1}(v_{K+1}-c)}{T_+(v_i - c) - \lambda_+^{K+1}(v_{i}-c)} \\ \sim - \dfrac{T_+(v_{K+1} - c)}{\lambda_+^{K+1}(v_{i}-v_{K+1})} + \dfrac{v_{K+1} - c}{v_i - v_{K+1}} \to 0\, . 
\end{multline*} 
As a consequence, $ b_{K+1} \la F_+^{K+1} \ra $ has a finite limit too.
On the other hand, we learn from Section \ref{sec:monotonicity} that $\rho_+$ is decreasing. 
Dividing \eqref{mugmuifgmiu} by $\lambda_+^{K+1}$, then taking limit as $c\nearrow v_{K+1}$, we get immediately 
\[ 0\geq - \lim_{c \nearrow v_{K+1}} \left( b_{K+1} \la F_+^{K+1} \ra\right)\, . \]
Furthermore, we can separate the respective contributions of positive relative velocities, and negative relative velocities, $\rho_+ = \rho^+_+ + \rho^-_+$, as in the proof of Lemma \ref{lem:zero monotonicity}. In particular, we have:
\begin{align*}
\dfrac{d\rho_+^+}{dz}(0^+) & = \sum_{\{v_i>c\}} \omega_i \dfrac1{v_i - c} \left( I_+(0) - \mathfrak{f}_+(0,v_i)\right)& \\
& \sim \omega_{K+1} \dfrac1{v_{K+1} - c}\left( I_+(0) - \dfrac{T_+^+}{T_-^+}I_-(0)\right) & \text{as $c\nearrow v_{K+1}$}\, . 
\end{align*}
Recall the following identity from the proof of Lemma \ref{ugpmig} \eqref{eq:Delta I},
\begin{equation*}
T_-^+ I_+(0) - T_+^+I_-(0) = - 4\chi_S \rho^- (0)\, .
\end{equation*}
We get eventually,
\begin{align*}
\lim_{c \nearrow v_{K+1}} \left( b_{K+1} \la F_+^{K+1} \ra\right) &= - \lim_{c\nearrow v_{K+1}} \dfrac1{\lambda_+^{K+1}}  \dfrac{d\rho_+}{dz}(0^+)  \\
&\geq  - \lim_{c\nearrow v_{K+1}} \dfrac1{\lambda_+^{K+1}}  \dfrac{d\rho_+^+}{dz}(0^+) \\
& \geq   \lim_{c\nearrow v_{K+1}} \omega_{K+1}\dfrac{1}{ \lambda_+^{K+1}(v_{K+1} - c)} \dfrac{ 4\chi_S \rho^- (0)}{T^+_-}\, .   
\end{align*}
To conclude, recall from \eqref{eq:asymptot +} that $\lambda^{K+1}_+ (v_{K+1}-c) < T_+^+$. This yields the following bound from below,
\[ \lim_{c \nearrow v_{K+1}} \left( b_{K+1} \la F_+^{K+1} \ra\right) \geq \omega_{K+1} \dfrac{ 4\chi_S \rho^- (0)}{T_+^+T^+_-} >0\, . \]
\end{proof}

\subsection{Positive jumps in the function $\Upsilon$ and consequences} 

\label{sec:jumps}

In this section, we examine the matching condition on $c$, namely
\[ \Upsilon(c) =  \partial_z S(0) = 0\, . \]
Our main objective is to investigate existence, and possible uniqueness of $c$. As a matter of fact, we shall answer negatively to these questions, because the situation is by far more complicated than in the macroscopic diffusive limit discussed in Section \ref{sec:diffusion limit}. 

In the following proposition, we establish that variations of $\Upsilon(c)$ can be deduced from some monotonicity of the macroscopic profile $\rho$, with respect to $c$. We include the dependency with respect to $c$ in the notations to resolve any possible ambiguity. 

\begin{proposition}
Suppose the profile $\rho(z;c)$ has the following monotonicity, 
\beq\label{eq:monotonicity rho c} (\forall c)\quad  \begin{cases}
(\forall z<0)\quad & \dfrac{\partial\rho_-}{\partial c}(z;c) > 0    \medskip\\
(\forall z>0)\quad & \dfrac{\partial\rho_+}{\partial c}(z;c) < 0     \end{cases}\eeq
then the function $\Upsilon$ is decreasing.  
\end{proposition}
\begin{proof}
Recall that $\Upsilon$ is given by the following integral representation formula, 
\beq \Upsilon(c) = \partial_z S(0;c) =  - \int_{-\infty}^0 \mu_+(c) e^{\mu_+(c) z} \rho(z;c)\, dz + \int_0^{+\infty} \mu_-(c) e^{-\mu_-(c) z} \rho(z;c)\, dz\, , \label{eq:dxS0}\eeq
where the reaction-diffusion exponents $\mu_{\pm}(c)$ are defined as in \eqref{eq:exposant mu}. 
For our purpose, we compute 
\begin{align*}
\dfrac{d \Upsilon}{dc}(c)  
&=  - \int_{-\infty}^0 \dfrac{d\mu_+}{dc}(c)\left( 1 + \mu_+(c) z\right)e^{\mu_+(c) z} \rho(z;c)\, dz - \int_{-\infty}^0 \mu_+(c) e^{\mu_+(c) z} \dfrac{ \partial \rho}{\partial c}(z;c)\, dz \\ 
& \quad + \int_0^{+\infty} \dfrac{d\mu_-}{dc}(c) \left( 1 -  \mu_-(c) z\right) e^{-\mu_-(c) z} \rho(z;c)\, dz
+ \int_0^{+\infty} \mu_-(c) e^{-\mu_-(c) z} \dfrac{ \partial \rho}{\partial c}(z;c)\, dz\, .
\end{align*}
We  integrate by parts the  first contribution in each line of the r.h.s., in order to get the following expression,
\begin{align}
\dfrac{d \Upsilon}{dc}(c) 
&=  \int_{-\infty}^0 \dfrac{d\mu_+}{dc}(c) z e^{\mu_+(c) z} \dfrac{ \partial \rho}{\partial z}(z;c)\, dz - \int_{-\infty}^0 \mu_+(c) e^{\mu_+(c) z} \dfrac{ \partial \rho}{\partial c}(z;c)\, dz \nonumber\\ 
& \quad - \int_0^{+\infty} \dfrac{d\mu_-}{dc}(c) z e^{-\mu_-(c) z} \dfrac{ \partial \rho}{\partial z}(z;c)\, dz
+ \int_0^{+\infty} \mu_-(c) e^{-\mu_-(c) z} \dfrac{ \partial \rho}{\partial c}(z;c)\, dz\, .\label{eq:dcdxS}
\end{align}
To conclude it is sufficient to notice that 
\[ \begin{cases}
\dfrac{d\mu_+}{dc}(c) = \dfrac1{2D_S}\left ( 1 + \dfrac{c}{\sqrt{c^2 + 4 D_S\alpha}}\right ) >0 \medskip\\
\dfrac{d\mu_-}{dc}(c) = \dfrac1{2D_S} \left (- 1 + \dfrac{c}{\sqrt{c^2 + 4 D_S\alpha}}\right ) <0 
\end{cases} \]
Therefore, all signs in \eqref{eq:dcdxS} coincide to give 
\[\dfrac{d \Upsilon}{dc}(c)  < 0\, .\]
\end{proof}

\begin{remark}
There is some ambiguity when we study the variation with respect to $c$. Indeed, the profile $\rho(z;c)$ is defined up to a constant factor. To resolve this ambiguity, notice that any smooth normalization of the type $\tilde \rho(z;c) = R(c) \rho(z;c)$, $R(c)>0$, can be handled in the following way,
\begin{align*}
\dfrac{d}{dc}\left(\partial_z \tilde S(0;c) \right) 
& = R(c)\dfrac{d}{dc}\left(\partial_z   S(0;c) \right) + \dfrac{dR}{dc}(c)   \partial_z  S(0;c)  \, .
\end{align*}
In particular, we see that $\frac{d}{dc}\left(\partial_z \tilde S(0;c) \right) <0$ when $ \partial_z   S(0;c)  $ vanishes, so zeros of $\Upsilon(c)$ and $\tilde\Upsilon(c)$ do coincide. 
\end{remark}

Note that monotonicity condition \eqref{eq:monotonicity rho c} is exactly what is expected in the case of two velocities only $V = \{\pm \v0\}$. There, the spatial density is given by 
\[ \rho(z) = \begin{cases}
\rho(0) e^{\lambda_-(c) z}  & \text{for $z<0$} \medskip\\
\rho(0) e^{-\lambda_+(c) z}  & \text{for $z>0$}  \end{cases} \]
where the exponents $\lambda_\pm$ have the appropriate monotonicity with respect to $c$  \eqref{eq:monotonicity rho c}, as in \eqref{eq:monotonicity lambda}.

However, the situation is more tricky for a larger number of velocities, due to the superposition of modes, and in particular the breaking of monotonicity during exchange of modes, as in Figure \ref{fig:swap}. It can be seen in the latter that $\rho$ has exactly the reverse monotonicity with respect to $c$, close to the origin $z = 0$ (compare \eqref{eq:monotonicity rho c} and Figure \ref{fig:swap}). Of course, monotonicity \eqref{eq:monotonicity rho c} is recovered far from the origin, as the shape of $\rho$ is dominated by the lower mode. 

The objective of the following calculation is to estimate the variations of $\Upsilon(c)$ as $c$ crosses some discrete velocity, following Section \ref{sec:exchange of modes}. 

Firstly, we suppose that $c<v_{K+1}$.  We opt for the alternative formulation of $\Upsilon(c)$ after integration by parts in \eqref{eq:dxS0}, and using the continuity of $\rho$ at $z = 0$,
\begin{align*} 
\Upsilon(c) &=   \int_{-\infty}^0  e^{\mu_+(c) z} \partial_z \rho(z;c)\, dz - \int_0^{+\infty}  e^{-\mu_-(c) z} \partial_z \rho(z;c)\, dz\\ 
& = \sum_{k = 1}^K \dfrac{\lambda^k_-(c)}{\lambda^k_-(c) + \mu_+(c)}  a_k \la F_-^k\ra - \sum_{k = K+1}^N \dfrac{\lambda^k_+(c)}{\lambda^k_+(c) + \mu_-(c)}  b_k \la F_+^k\ra \, .
\end{align*}
Taking the limit $c\nearrow v_{K+1}$, we obtain
\begin{multline*}\lim_{c\nearrow v_{K+1}} \Upsilon(c) = \sum_{k = 1}^K \dfrac{\lambda^k_-(v_{K+1})}{\lambda^k_-(v_{K+1}) + \mu_+(v_{K+1})}  \underline{a_k \la F_-^k\ra} - \sum_{k = K+2}^N \dfrac{\lambda^k_+(v_{K+1})}{\lambda^k_+(v_{K+1}) + \mu_-(v_{K+1})}  \underline{b_k  \la F_+^k\ra} \\ - \lim_{c \nearrow v_{K+1}} \left( b_{K+1} \la F_+^{K+1} \ra\right) \, .
\end{multline*}
Similarly, we obtain after taking the limit $c\searrow v_{K+1}$,
\begin{multline*}\lim_{c\searrow v_{K+1}} \Upsilon(c) = \sum_{k = 1}^K \dfrac{\lambda^k_-(v_{K+1})}{\lambda^k_-(v_{K+1}) + \mu_+(v_{K+1})}  \overline{a_k  \la F_-^k\ra} - \sum_{k = K+2}^N \dfrac{\lambda^k_+(v_{K+1})}{\lambda^k_+(v_{K+1}) + \mu_-(v_{K+1})}  \overline{b_k \la F_+^k\ra} \\ + \lim_{c \searrow v_{K+1}} \left( a_{K+1} \la F_-^{K+1} \ra\right) \, .\end{multline*}
By taking the difference, we should pay attention to the fact that the averages $\underline{\la F_\pm^k\ra}$ and $\overline{\la F_\pm^k\ra}$ do not coincide, because there is one term changing in the sum, due to the discontinuity of $T$. More precisely, we have
\[ \begin{cases} 
(\forall k\leq K) \quad & \underline{\la F_-^k\ra} - \overline{\la F_-^k\ra}  = \omega_{K+1}\left( \dfrac1{T_-^+} - \dfrac1{T_-^-} \right) \medskip\\
(\forall k\geq K+2)& \underline{\la F_+^k\ra} - \overline{\la F_+^k\ra}  = \omega_{K+1}\left( \dfrac1{T_+^+} - \dfrac1{T_+^-} \right)
\end{cases}
\]
As a conclusion, we can decompose the jump in $\Upsilon$ at $c = v_{K+1}$ into three different contributions,
\begin{align*}
\Upsilon (v_{K+1}^+) - \Upsilon (v_{K+1}^-)  & = \left ( \sum_{k = 1}^K \dfrac{\lambda^k_-(v_{K+1})}{\lambda^k_-(v_{K+1}) + \mu_+(v_{K+1})}  a_k \right ) \omega_{K+1}\left( \frac1{T_-^-} - \frac1{T_-^+} \right) \\ 
& \quad + \left (\sum_{k = K+2}^N \dfrac{\lambda^k_+(v_{K+1})}{\lambda^k_+(v_{K+1}) + \mu_-(v_{K+1})}   b_k\right ) \omega_{K+1}\left( \dfrac1{T_+^+} - \dfrac1{T_+^-} \right) \\
& \quad + \lim_{c \searrow v_{K+1}} \left( a_{K+1} \la F_-^{K+1} \ra\right)  + \lim_{c \nearrow v_{K+1}} \left( b_{K+1} \la F_+^{K+1} \ra\right)\\
& \geq 
\left ( \sum_{k = 1}^K \dfrac{\lambda^k_-(v_{K+1})}{\lambda^k_-(v_{K+1}) + \mu_+(v_{K+1})}  a_k \right ) \omega_{K+1}\dfrac{-2(\chi_S + \chi_N)}{T_-^- T_-^+} \\ 
& \quad + \left (\sum_{k = K+2}^N \dfrac{\lambda^k_+(v_{K+1})}{\lambda^k_+(v_{K+1}) + \mu_-(v_{K+1})}   b_k\right ) \omega_{K+1}\dfrac{2(\chi_N - \chi_S)}{T_+^+ T_+^-} \\
& \quad +  \omega_{K+1} \dfrac{ 4\chi_S \rho^+ (0)}{T_+^-T^-_-}  + \omega_{K+1} \dfrac{ 4\chi_S \rho^- (0)}{T_+^+T^+_-}
\end{align*}
Interestingly, the last line does not depend on the reaction-diffusion coefficients $(\alpha,D_S)$. On the other hand, the remainder is arbitrarily small as $\mu_\pm \gg \lambda_\pm^k$, for instance as $\alpha \to \infty$, all other parameters being unchanged. In the latter case, the jump is positive as a consequence of Proposition \ref{prop:swap 2} (see Figure \ref{fig:nonuniqueness}).

\begin{figure}
\begin{center}
\includegraphics[width = 0.8\linewidth]{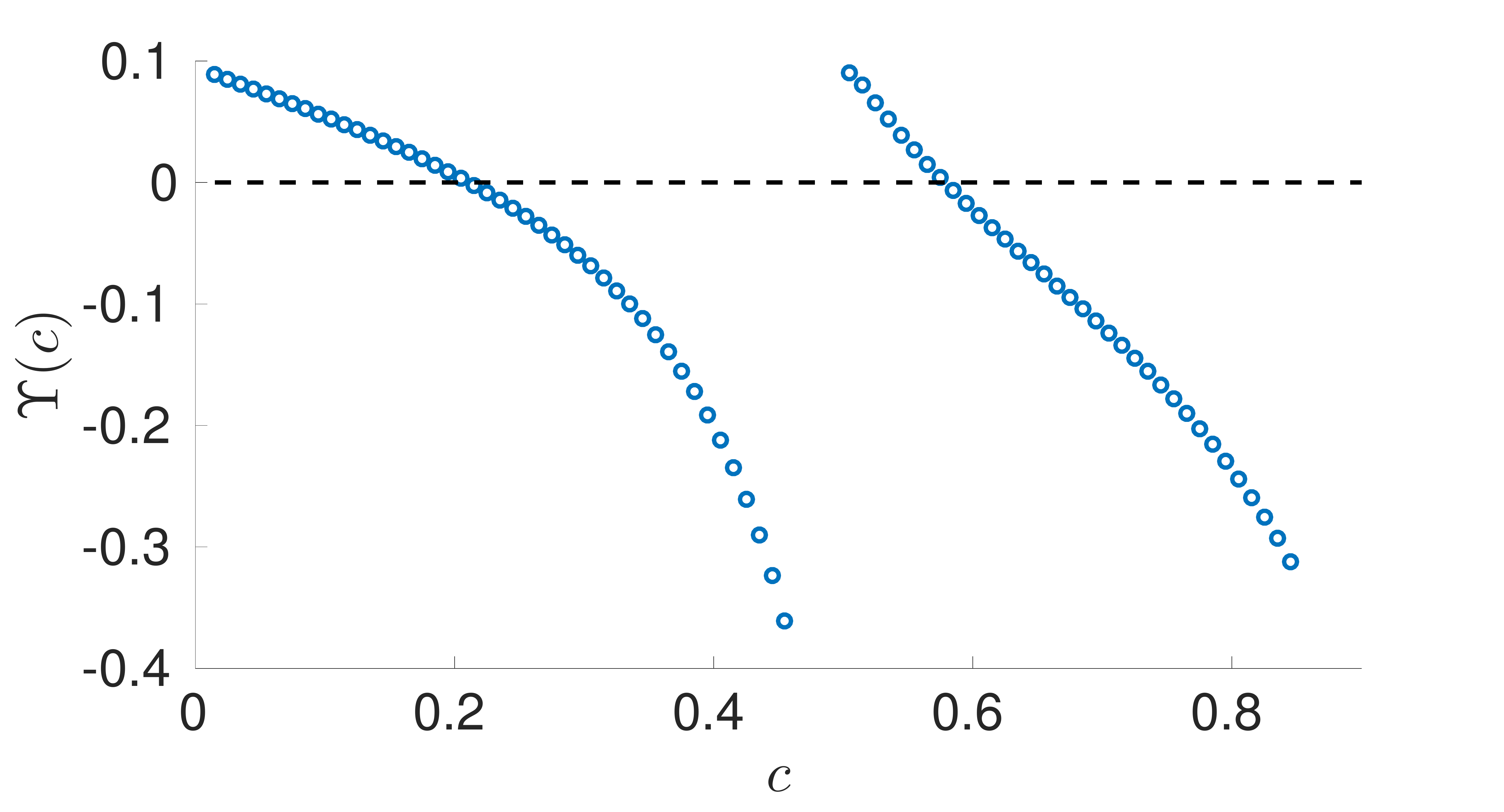}
\caption{\small Coexistence of multiple travelling waves, when the function $\Upsilon$ crosses the zero level several times.  Here, $V = \{-1,-0.5,0.5,1\}$, with uniform weights $\omega$. Other parameters are $\chi_S = 0.48$, $\chi_N = 0.44$, $\alpha = 50$ and $D_S = 0.5$.
}
\label{fig:nonuniqueness}
\end{center}
\end{figure}

\begin{figure} 
\begin{center}
\includegraphics[width = 0.8\linewidth]{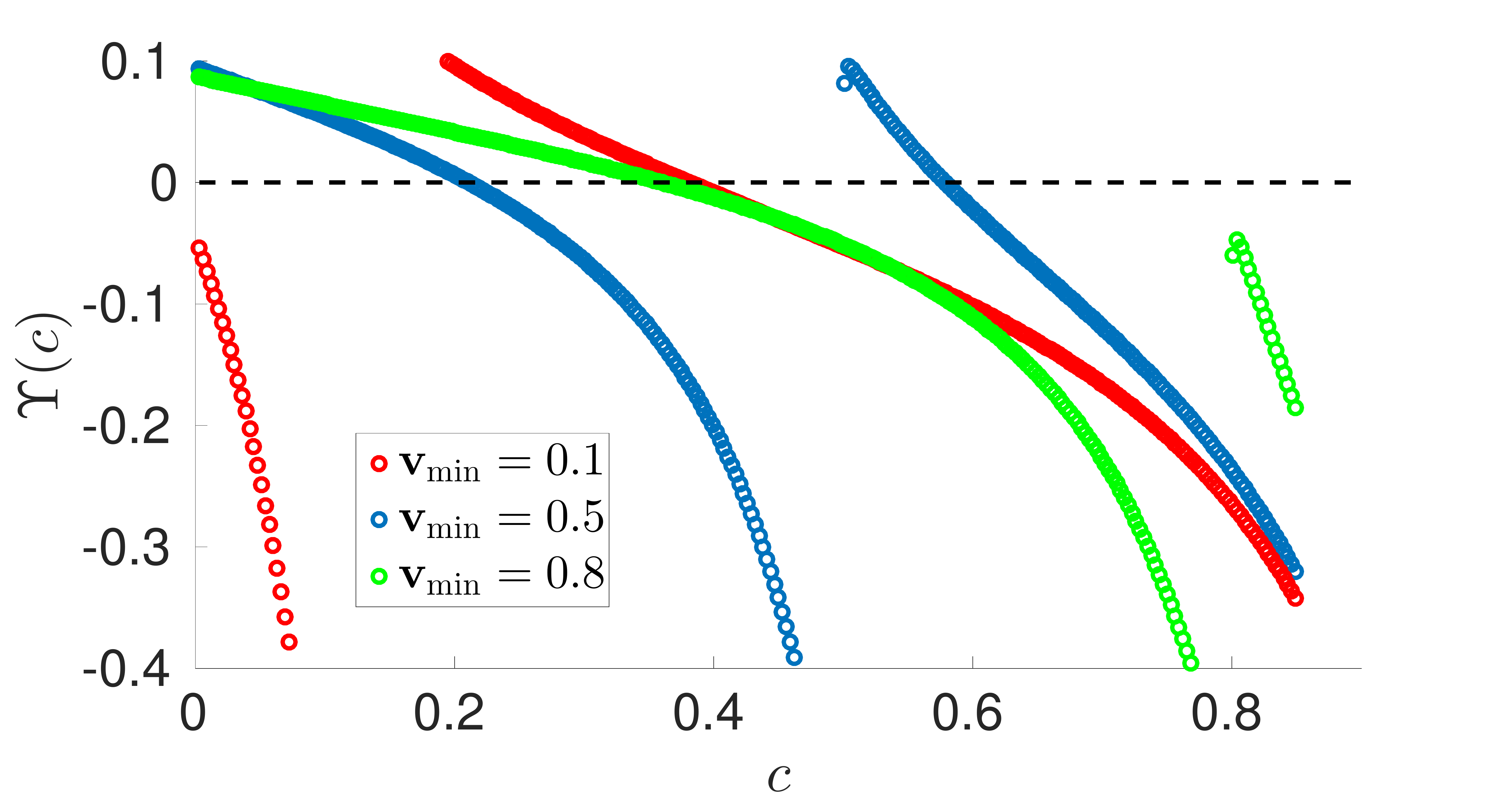}\; 
\caption{\small Variations around Figure \ref{fig:nonuniqueness}, with $\{-1,-\vmin,\vmin,1\}$. Interestingly, we observe a transition between coexistence of two travelling waves -- a slow wave ($c\approx 0.2$) and a fast wave ($c\approx 0.6$) -- and existence of a unique wave ($c\approx 0.4$ for both $\vmin = 0.1$ and $\vmin = 0.8$, by pure coincidence). There is some apparent paradox here: by increasing the minimal speed $\vmin$ from $0.5$ to $0.8$, the fast wave disappears (observe the behaviour of the "fast" branches on the right side). On the other hand, by decreasing the minimal speed, the slow wave disappears (observe the behaviour of the "slow" branches on the left side). In particular, decreasing $\vmin$ results in increasing the wave speed.}
\label{fig:bistability}
\end{center}
\end{figure}

\begin{figure} 
\begin{center}
\includegraphics[width = 0.48\linewidth]{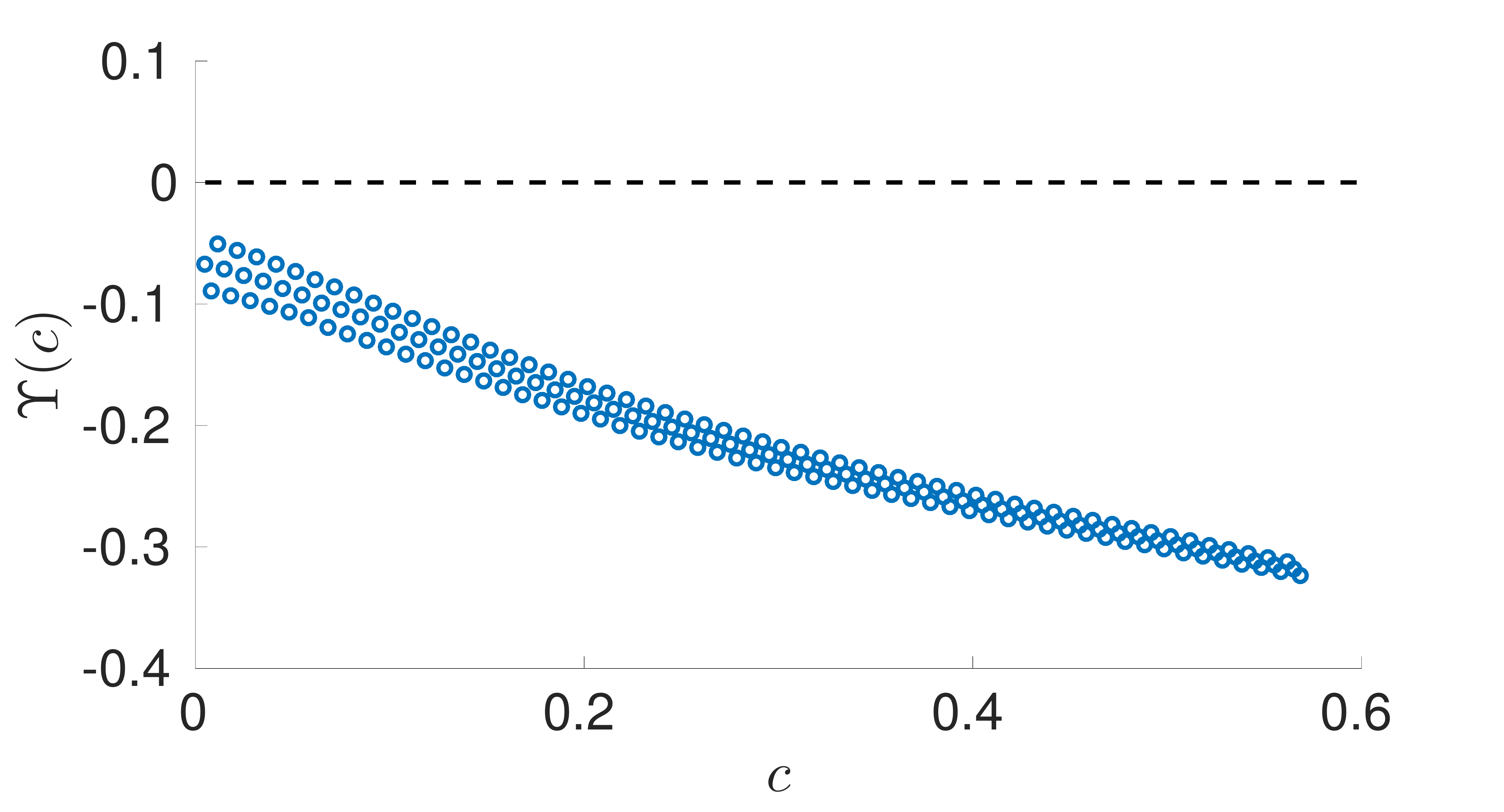}\; 
\includegraphics[width = 0.48\linewidth]{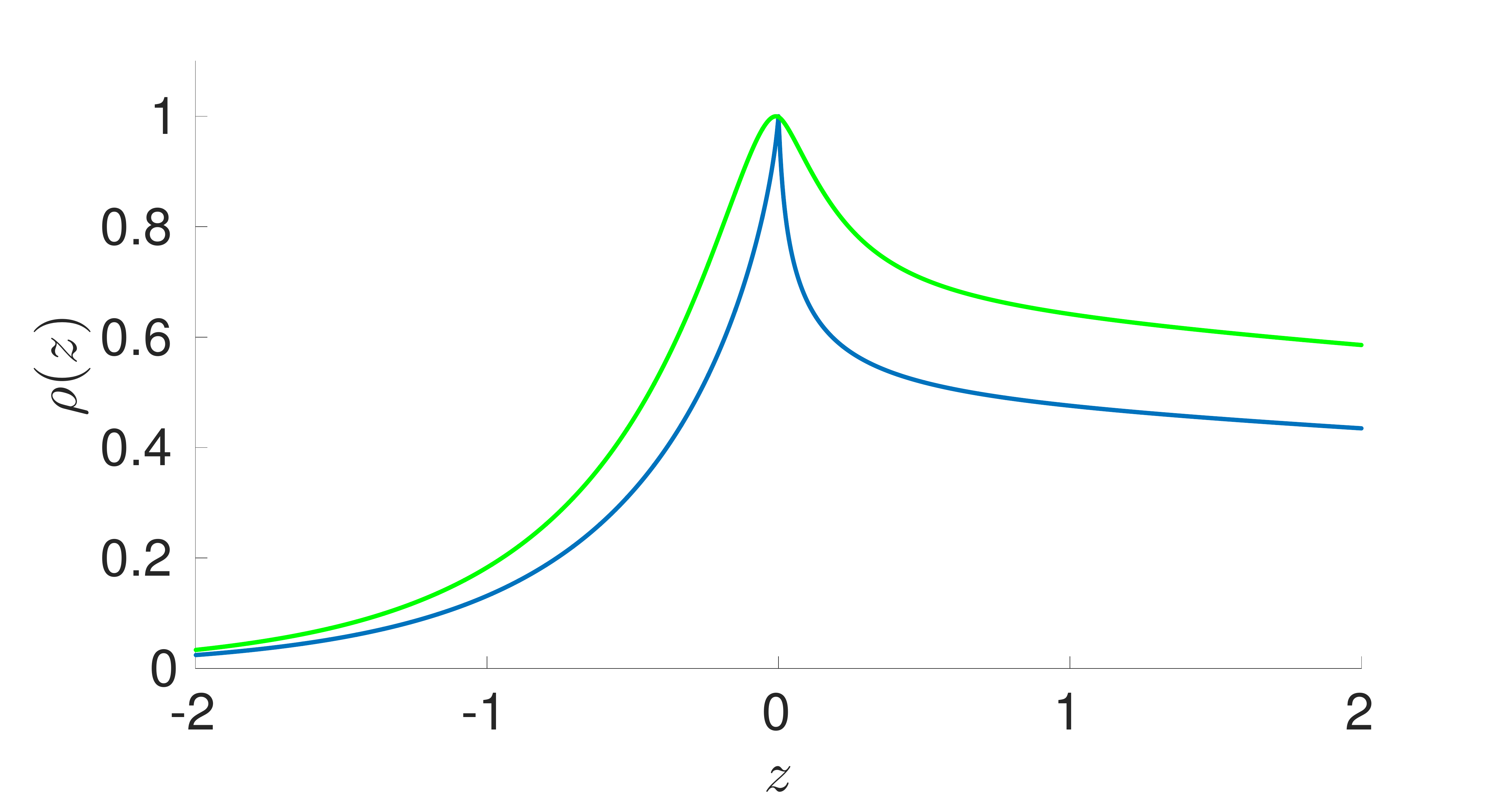}
\caption{\small Non existence of the travelling wave.  Here, $V = (-1,1)$, and $d\nu(v) = 1+5 \exp(-4|v|)dv$. Other parameters are $\chi_S = 0.48$, $\chi_N = 0.44$, $\alpha = 50$ and $D_S = 0.5$. (Left) The function $\Upsilon(c) = \partial_z S(0)$ is plotted for admissible values of $c\in (0,c^*)$. It is always negative, so that the peak of $S$ is always located on the left side of the origin. Hence, there is no solution to the travelling wave problem with basic monotonicity assumptions ($S$ uni modal, and $N$ increasing). The velocity set is discretized for numerical purposes, with a step $\Delta v = 0.01$. The smallness of increments in speed $c$ ($\Delta c = dv/3$) enables to catch small irregular structure of $\Upsilon$.  (Right) The macroscopic density profile $\rho(z)$ is shown for the specific value $c=0$ to illustrate the following paradox: despite the global tendency to move to the right ($\chi_N$ is relatively large), the profile $\rho$ is locally sharper on the right side of the origin, rather than on the left side. Of course, the global shift of the density is recovered at large scale (exponential decay is sharper on the left side), but only small scales matter here, because chemical parameters are chosen so that the typical length scale of the chemical signal transmission is $\sqrt{\alpha/D_S} = 0.1$. The stationary signal concentration $S$ is superimposed in green line.}
\label{fig:nonexistence}
\end{center}
\end{figure}

\subsection{Numerical investigation}

We conclude this section by some numerical illustrations.

Firstly, we investigate the case of four velocities, say $V = \{-1 , -\vmin, \vmin, 1\}$.  Figure \ref{fig:nonuniqueness} exhibits one particular case for which there is seemingly two roots of the matching equation $\Upsilon(c) = 0$. This is the exact counter-part of the exchange of modes happening when $c$ crosses the discrete velocity $\vmin$. There is a positive jump discontinuity in $\Upsilon$ as $c\to \vmin$. This corresponds precisely to the non-zero amplitude which is transferred instantaneously from the left to the right side of the origin as $c$ crosses $\vmin$. We believe that both travelling waves are stable (see \cite{companion} for more thorough discussion based on numerical experiments). This yields some interesting paradox (see Figure \ref{fig:bistability}). 

Secondly, we exhibit another intriguing case, with a larger number of velocities (an idealization of the continuous case) for which there is seemingly no solution to the matching equation $\Upsilon(c) = 0$ (Figure \ref{fig:nonexistence}). The long time behaviour of the Cauchy problem is very difficult to conceive in the mind.

\section{Perspectives}\label{sec:perspectives}

\subsection{Moderate signal integration}
 
It would be interesting to include more functions than the sign function $\boldsymbol\phi = - \sign$  in \eqref{eq:lambda} during the course of analysis. Indeed, a more appropriate choice would be a decreasing sigmoidal function with a stiffness parameter $\tau>0$ that is a typical time-scale \cite{saragosti_directional_2011},
\[ 
\phi\left( \left.\dfrac{D\log S}{Dt}\right|_{v'} \right) = - \tanh\left( \tau \left.\dfrac{D\log S}{Dt}\right|_{v'} \right)\, . 
\]
Time $\tau$ can be related to the adaptation time-scale of individual cells \cite{perthame_derivation_2015}. 

There are two important differences with the analysis performed here: first, the tumbling rate cannot be reduced to an elementary rule of signs. Also, the confinement effect is more involved, as the rate of confinement $\lambda$ strongly depends on the asymptotic decay of chemical concentrations. 

\subsection{Delay effects during signal processing}

Clearly, the assumption of instantaneous signal integration is very reductive. A way to relax this strong assumption consists in including more variables in the kinetic model, accounting for signal processing at the level of individual cells \cite{erban_signal_2005,xue_travelling_2010,
franz_travelling_2013,
xue_macroscopic_2013,perthame_derivation_2015}. However, this makes the analysis much more complicated. Reduced model different from \eqref{eq:meso model} could be investigated first.

\subsection{Angular persistence}

By relaxing the specific choice \eqref{eq:A1}, angular persistence during tumbling can be taken into account. In fact pre- and post-tumbling velocities $(v',v)$ are correlated. The mean value of the directional change is approx. $68^\circ$ \cite{berg_e._2004, saragosti_directional_2011}. This can be encoded in the model using a Gaussian p.d.f. to distribute the post-tumbling velocity, namely
\[ \boldsymbol K(v,v') = \dfrac1{Z}\exp\left(\dfrac{v\cdot v'}{\Sigma^2}\right)\, , \]
In \cite{saragosti_directional_2011}, it has been further evidenced that the standard deviation $\Sigma$ itself is modulated by the chemical gradients, as predicted in \cite{vladimirov_predicted_2010}. This can be put into the model via the dependency $\Sigma = \Sigma(t,z,v')$, where $\Sigma$ is an increasing function of the tumbling rate $\boldsymbol\lambda$: more likely is the tumble, larger the deviation between $v$ and $v'$ is expected. 

\subsection{Two-species concentration waves}

In \cite{almeida_existence_2014, emako_traveling_2016}, the authors investigate a two-species model for the interaction of two populations of bacteria with different intrinsic wave speeds. The analogue of the diffusive limit system \eqref{eq:macro model} is analysed by seeking travelling wave solutions. As far as we know, the corresponding kinetic system has not been studied yet.

\bibliographystyle{abbrv}
\bibliography{BIB-HDR}

\end{document}